\documentclass[12pt,reqno]{amsart}
\usepackage{graphicx,amsmath,amssymb,fullpage,url}

\usepackage{srcltx}
\usepackage[T1]{fontenc}
\usepackage[utf8]{inputenc}
\usepackage{ytableau}

\usepackage{xcolor}

\usepackage[russian,english]{babel}

\usepackage{color} 
\usepackage[
        colorlinks,
]{hyperref}
\newtheorem{X}{X}[section]
\newtheorem{corollary}[X]{Corollary}

\newtheorem{lemma}[X]{Lemma}
\newtheorem{proposition}[X]{Proposition}
\newtheorem{theorem}[X]{Theorem}

\def\sums{\mathop{\sum \Bigl.^{*}}\limits}
\theoremstyle{definition}
\newtheorem{example}[X]{Example}
\newtheorem{remark}[X]{Remark}
\newcommand{\V}{{\rm Var}}
\newcommand{\E}{\mathbb E}
\newcommand{\Q}{\mathbb Q}
\newcommand{\F}{\mathbb F}
\newcommand{\p}{\mathfrak p}
\renewcommand{\O}{\mathcal O}
\newcommand{\Z}{\mathbb Z}

\renewcommand{\r}{\,{\rm real}}
\renewcommand{\P}{\text{$P$}}
\renewcommand{\C}{\mathbb C}
\newcommand{\R}{\mathbb R}
\newcommand{\Gal}{{\rm Gal}}
\newcommand{\disc}{{\rm disc}}
\newcommand{\Irr}{{\rm Irr}}
\newcommand{\frob}{{\rm Frob}}
\newcommand{\ord}{{\rm ord}}

\newcommand{\res}{1-r}
\newcommand{\ee}{{\rm e}}
\renewcommand{\hat}{\widehat}
\newcommand{\norm}[1]{\lVert#1\rVert}

\renewcommand{\epsilon}{\varepsilon}

\title{Distribution of Frobenius elements \\ in families of Galois extensions}
\author{Daniel Fiorilli}
\address{Univ. Paris-Saclay, CNRS, Laboratoire de mathématiques d'Orsay, 91405, Orsay, France.} 
\email{daniel.fiorilli@universite-paris-saclay.fr}
\author{Florent Jouve}
\address{Univ. Bordeaux, CNRS, Bordeaux INP, IMB, UMR 5251,  F-33400, Talence, France.}
\email{florent.jouve@math.u-bordeaux.fr}
\date{\today}
 \dedicatory{
  \selectlanguage{russian} Памяти Алексея 
Зыкина и Татьяны Макаровой}

\begin{document}

\begin{abstract}

Given a Galois extension $L/K$ of number fields, we describe fine distribution properties of Frobenius elements \emph{via} invariants from representations of finite Galois groups and ramification theory. We exhibit explicit families of extensions in which we evaluate these invariants, and deduce a detailed understanding and a precise description of the possible asymmetries. We establish a general bound on the generic fluctuations of the error term in the Chebotarev density theorem which under GRH is sharper than the Murty--Murty--Saradha and Bella\"iche refinements of the Lagarias--Odlyzko and Serre bounds, and which we believe is best possible (assuming simplicity, it is of the quality of Montgomery's conjecture on primes in arithmetic progressions). Under GRH and a hypothesis on the multiplicities of zeros up to a certain height, we show that in certain families these fluctuations are dominated by a constant lower order term.
As an application of our ideas we refine and generalize results of K. Murty and of J. Bella\"iche and we answer a question of N. Ng. In particular, in the case where $L/\Q$ is Galois and supersolvable, we prove a strong form of a conjecture of K. Murty on the unramified prime ideal of least norm in a given Frobenius set. 
The tools we use include the Rubinstein--Sarnak machinery based on limiting distributions and a blend of algebraic, analytic, representation theoretic, probabilistic and combinatorial techniques. 
\end{abstract}

\maketitle
\tableofcontents

\section{Introduction}

\subsection{Background and perspective}

The Chebotarev density theorem, one of the major number theoretic achievements of the early $20$th century, has been proven to be of crucial importance in a variety of problems. 
Beyond knowing the exact asymptotic densities of primes in Frobenius sets, one often needs to understand the dependence of the involved error term as a function of the invariants of the associated field extension. In this direction, Lagarias and Odlyzko~\cite{LaOd} established an effective Chebotarev density theorem. Letting $L/K$ be a Galois extension of number fields and assuming GRH for $\zeta_L(s)$ (an assumption sometimes called ``ERH for $L$''), they showed that for any conjugacy class $C\subset \Gal(L/K)$, the function 
\begin{equation}
 \pi(x;L/K,C):= \# \{ \mathfrak p \subset \mathcal O_K  \text{ unram. in }L/K: \mathcal N \mathfrak p \leq x, {\rm Frob}_{\mathfrak p} = C \}\,, 
 \label{equation definition pi x L/K C}
\end{equation}
where ${\rm Frob}_{\mathfrak p}$ (resp. $\mathcal N\mathfrak p$) denotes the Frobenius conjugacy class (resp. the cardinality of the residue field $\mathcal O_K/\mathfrak p$) corresponding to a prime ideal $\mathfrak p$ of $\mathcal O_K$, satisfies the estimate~\cite[Th. 4]{SeIHES}
\begin{equation}
\pi(x;L/K,C) - \frac{|C|}{|G|} { \rm Li }(x) \ll \frac{|C|}{|G|} x^{\frac 12} \log(d_L x^{[L:\Q]})\,, 
\label{equation Lagarias Odlyzko}
\end{equation}
where ${\rm Li}(x):=\int_2^x  (\log t)^{-1}{\rm d}t$, $d_L$ is the absolute value of the discriminant of $L/\Q$ and the implied constant is absolute (unless otherwise specified, all implied constants in this paper will be absolute).
This estimate was a cornerstone in Serre's seminal work~\cite{SeIHES} with applications for example to the Lang--Trotter conjecture and the open image theorem for elliptic curves. Effective Chebotarev estimates also led Murty~\cite{Mur} and later Bucur--Kedlaya~\cite{BK} to develop applications to effective Sato--Tate distributions in a general context. 

Subsequently, the GRH result of Lagarias--Odlyzko was refined under Artin's Conjecture (AC) by Murty--Murty--Saradha~\cite{MMS}, 
 and more recently by Bella\"iche~\cite{Be} who adopted a new representation theoretic point of view and used extra inputs from Kowalski's axiomatic large sieve. It is now known that under GRH and AC, the right hand side of~\eqref{equation Lagarias Odlyzko} can be replaced with $\lambda_G(C)x^{\frac 12} \log(x|G| d_K^{1/[K:\Q]} R_{L})$,
 where $R_L$ is the product of the prime numbers ramified in $L/\Q$, and where the ``Littlewood norm'' $\lambda_G(C):=|C||G|^{-1}\sum_{\chi \in \Irr(G) } |\chi(C)|\chi(1)$ is $\leq |C|^{\frac 12}$ (which is the bound that~\cite{MMS} relies on), and can be significantly smaller in some families (see~\cite[\S 2.3]{Be}). These and further refinements were shown to have applications to the arithmetic of elliptic curves and bounds on the size of the least prime in a Frobenius set (see~\cite{LMO,Za,TZ1,TZ2}, see also~\cite{CK2,EM,GrMo,KNW,Winc}). Along these lines we mention the work of Cho--Kim~\cite{CK} and of Pierce--Turnage-Butterbaugh--Wood~\cite{PTW} (see also~\cite{TZ3}) who managed to further refine the estimates of Lagarias--Odlyzko in families and to deduce bounds on exponents and $\ell$-torsion of class groups of number fields. 
Remarkably, important applications of effective Chebotarev have also been obtained outside the realm of number theory. Indeed, Kuperberg~\cite{Kup} solved an important computability problem in knot theory, under the GRH for Artin $L$-functions.    

In this paper we investigate asymptotic properties of the limiting distribution of a suitable normalization of the error term 
\begin{equation} 
\label{equation error term chebotarev}
 \pi(x;L/K,C) - \frac{|C|}{|G|} {\rm Li}(x)
\end{equation}
in families of Galois extensions of number fields. 
In some of the families of number field extensions that we shall consider, we will show that for most values of $x$, the error term~\eqref{equation error term chebotarev} is dominated by a lower order term of constant size.

 To illustrate our results, we consider the family $\{K_d\}$ of Hilbert class fields of 
 $\Q(\sqrt d)$, with $d$ running over negative fundamental discriminants. 
 A general unconditional version of the following result will be stated in Theorem~\ref{theorem mean variance}.

\begin{theorem}
\label{theorem vitrine}
Let $d\leq -4$ be a fundamental discriminant, let $K_d$ be the Hilbert 
Class Field of $\Q(\sqrt{d})$, and write $G_d=\Gal(K_d/\Q)$.
Assuming the Riemann Hypothesis for $\zeta_{K_d}(s)$ (\emph{i.e.} ERH for $K_d$),
  the limiting distribution of 
$$  E(y;K_d/\Q,\{{\rm id}\}):= y\ee^{-y/2} \Big(\pi(\ee^y;K_d/\Q,\{{\rm id}\}) - \frac{{\rm Li}(\ee^y)}{|G_d|}\Big) $$
 exists, has mean 
 $  \asymp -1 $ and variance 
 $\ll h(d)^{-1}m_d\log |d|, $
 where $m_d$ is the maximal\footnote{Note that $m_d\ll h(d) \log |d|$. One even expects $m_d=2$ for large enough $|d|$, since $\chi(1)\leq 2$ for all irreducible characters $\chi$ of $G_d$ and since the zeros of $L(s,K_d/\Q(\sqrt d),\chi)$ are expected to be simple.
\label{footnote 1} } order of vanishing of $\zeta_{K_d}(s)$ in the region $\{s\in \mathbb C : 0<\Im(s) \leq h(d)(\log |d|)^3\}$, and $h(d)$ is the class number of $\Q(\sqrt d)$. 
 Assuming in addition that 
 $m_d$ is bounded by an absolute constant, the variance is $\asymp h(d)^{-1}\log |d|, $ and we have that
\begin{equation}
\liminf_{Y \rightarrow \infty}\frac{{\rm meas} \{y\leq Y : E(y;K_d/\Q,\{{\rm id}\})<0 \}}{Y}
\geq 
1-O\Big(\frac{\log |d| \log\log |d|}{\sqrt{|d|}}\Big),
\label{equation first chebyshev bias}
\end{equation}
where ${\rm meas}$ is the Lebesgue measure. 
\end{theorem}

The mean and variance calculations in Theorem \ref{theorem vitrine} imply that under ERH for $K_d$ and for most values of $\log x$,
$$ |G_d|\pi(x;K_d/\Q,\{{\rm id}\}) - {\rm Li}(x) = x^{\frac 12}(\log x)^{-1} (-c_dh(d) +O(\sqrt{h(d)m_d\log |d|})),  $$
where $\tfrac 12\leq c_d \leq 1+ h(d)^{-1}\ord_{s=\frac 12} \zeta_{K_d}(s) \ll 1$.
By means of comparison,~\eqref{equation Lagarias Odlyzko} (as well as the further refinements mentioned above) yield the error term $O(h(d)\log(x|d|)\log x)$. Our improved error term allows us to deduce that if $d$ is such that\footnote{See Footnote~\ref{footnote 1}, and recall that the bounds $\sqrt{|d|}(\log\log |d|)^{-1}\ll h(d)\ll \sqrt{|d|}\log\log |d|$ hold for any fundamental discriminant $d\leq -4$, under ERH for $K_d$.} $m_d=o(\sqrt{|d|}/(\log|d|\log\log |d|))$  then the error term in the Chebotarev density theorem is dominated by a significant lower order term.
Moreover, the lower bound \eqref{equation first chebyshev bias} can be interpreted by saying that when $|d|$ is large,  $\pi(\ee^y;K_d/\Q,\{{\rm id}\}) < {\rm Li}(\ee^y)/|G_d|$ for most values of $y$.

Theorem~\ref{theorem vitrine}
is a manifestation of an extreme
Chebyshev bias, which generalizes his observation made back in 1853 that in ``most intervals'' $[2,x]$, primes are more abundant in the residue class $3$ than in the class $1$, modulo $4$. The literature on this question is rich, and much progress has been made in the recent years. We mention the works~\cite{LitPrimes, KT, Ka2, Pu, RubSar, FiMa, La1} on the Shanks-R\'enyi problem, as well as generalizations over number fields~\cite{Maz,SarLet,Fi2, De,FoSn,DGK,LOS,Me} and over function fields~\cite{Cha,CI,DM,CFJ}.
For an exhaustive list of the numerous papers on the subject, 
 see \cite{GM,MaSc,MSetal}.

Following a suggestion made by Rubinstein and Sarnak, Ng~\cite{Ng} generalized the framework of~\cite{RubSar}
to the context of Galois extensions $L/K$ of number fields and performed extensive numerical computations. As is illustrated in Theorem~\ref{theorem vitrine}, the present work also considers the setting of Galois extensions of number fields. The more general context of Artin $L$-functions differs from the classical study of discrepancies in the distribution of primes in arithmetic progressions in several aspects. Notably, there are examples of Artin $L$-functions which vanish at $\tfrac 12$ (see~\cite{Ar} and Example~\ref{example ng}). Artin $L$-functions associated to irreducible representations of $\Gal(L/K)$ might also have non-simple complex zeros. This can substantially influence fine properties of the distribution of prime ideals in Frobenius sets (the influence of real zeros was predicted in~\cite{RubSar} and further explored in~\cite{Ng}). This makes the obvious extension of the linear independence assumption in~\cite{RubSar} (used to evaluate densities of subsets of primes) trivially false. Consequently, the notion of \emph{primitive} $L$-function, highlighted by Rudnick--Sarnak in~\cite{RudSar}, will be central in our analysis. 
The Artin $L$-function of an irreducible representation
 of $\Gal(L/K)$ will typically factorize as a product (with multiplicities) of primitive $L$-functions.

 By introducing a reduction of prime ideal counting functions in the relative extension $L/K$ to prime counting functions in $L/\Q$, we will express the former  
in terms of sums of zeros of $L$-functions that are expected to be primitive. This is the key observation that will allow us to refine K. Murty's bound on the unramified prime ideal of least norm (as well as the Bella\"iche improvements) in a given Frobenius set.
We will then apply the Rubinstein--Sarnak machinery involving limiting distributions arising from Besicovitch $B^2$ almost-periodic functions. Finally, after translating the problem to a probabilistic setting, we will establish central limit and large deviation type results in various families of Galois extensions. This will allow us to understand the distribution of the error term in the Chebotarev density theorem, and in turn to deduce precise asymptotic estimates on Chebyshev's bias.

 In our first main result (see Theorem~\ref{theorem mean variance}), 
we prove new estimates on the mean and variance of the limiting distribution of~\eqref{equation error term chebotarev} in terms of the ramification data of $L/K$ as well as representation theoretic invariants of 
$\Gal(L/K)$. 
Secondly, in Theorems~\ref{theorem least prime ideal} and~\ref{theorem chebotarev all x} we settle and refine a conjecture of K. Murty on the unramified prime ideal of least norm in a given Frobenius set (this takes into account Bella\"iche's improvements), and we refine the bounds of Murty--Murty--Saradha and Bella\"iche on the error term in the Chebotarev density theorem. 
Thirdly, under suitable hypotheses such as the Artin holomorphicity conjecture and the Riemann hypothesis, we apply our limiting distribution estimates to provide an asymptotic description of Chebyshev's bias in terms of the characters of $\Gal(L/K)$ and the discriminant of $L/\Q$, reducing the question to an effective inverse Galois problem. We tackle these invariants in several important families that are well studied in the literature, and deduce asymptotic estimates on this bias. In the generic case where $\Gal(L/K)=S_n$, we are able to apply powerful combinatorial estimates such as Roichman's bound~\cite{Ro} in order to deduce a precise asymptotic formula for the bias which we show is best possible (see Theorem~\ref{theorem S_n races}); this settles quantitatively a question of Ng.

The paper is organized as follows. 
In~\S\ref{section:results} we state our main results, which are of two distinct types. On one hand we obtain general information on the limiting distribution of~\eqref{equation error term chebotarev} and we give an asymptotic description of the densities in terms of invariants of the extension $L/K$.
 On the other hand we establish precise estimates on the mean and variance of this limiting distribution in the case of specific families of Galois extensions: abelian, dihedral, radical, and $S_n$ extensions, as well as Hilbert class fields of quadratic fields. 
We devote \S\ref{section:explicit} to explicit formulas and their translation into the probabilistic setting that is well suited to our approach. 
The arithmetic core of our method is described in~\S\ref{section:ArtinCond} where we relate the mean and the variance of the limiting distribution of~\eqref{equation error term chebotarev} to sums of characters of $\Gal(L/K)$ and Artin conductors, and prove our unconditional results (see Theorem~\ref{theorem chebotarev all x}) as well as Murty's conjecture in any Galois number field extension for which Artin's conjecture is known to hold (see Theorem~\ref{theorem least prime ideal}). Our main probabilistic results, effective central limit theorems and large deviation estimates, are then stated and proved in~\S\ref{section central limit theorem}. 
In~\S\ref{section:proofsgeneral} we conclude the proofs of our general results. We devote~\S\ref{section:ProofsSn} to the case of extensions $L/K$ for which $L/\Q$ is Galois of group $\Gal(L/\Q)=S_n$. We establish precise estimates on the mean and variance by exploiting the description of the irreducible representations in terms of Young tabloids and the associated combinatorial formulas for character values (chiefly the hook-length formula). In~\S\ref{section:AbelGalois} we prove the statements relative to some families of abelian extensions, including the case of the Hilbert class field of a quadratic field $K_d/\Q(\sqrt{d})$. Finally, in~\S\ref{section:supersolvable}, we focus on three specific families of supersolvable extensions of $\Q$. First, we investigate a 
family of dihedral extensions with controlled discriminant that was constructed by Kl\"uners. Second, in the case of the Hilbert class field of a quadratic field 
$\Q(\sqrt{d})$ seen as an extension of the rationals, we apply bounds on class numbers of (real and imaginary) quadratic fields due to Montgomery--Weinberger and Chowla. Third, we study radical extensions $\Q(\zeta_p,a^{1/p})/\Q$, where $a,p$ are distinct odd primes such that $p$ is not Wieferich to base $a$, making heavy use of Viviani's explicit computation~\cite{Vi} of the filtration of inertia at $a$ and $p$.

\subsection{Statement of assumptions}\label{section:assumptions}
We now state the hypotheses which will be used in this paper. We stress that some of our results are
 unconditional, and some depend on one or more of the hypotheses below (see for example Theorem~\ref{theorem mean variance}). In fact, much of our work is done without assuming GRH or LI.

We fix an absolute positive constant $M_0$ (say $M_0=10^5$). 
We let $L/K$ be an extension of number fields for which $L\neq \Q$ is Galois over $\Q$, and define $G=\Gal(L/K), $ $G^+=\Gal(L/\Q)$. For a finite group $\Gamma$, we denote by $\Irr(\Gamma)$ the set of irreducible characters and by $\Gamma^\sharp$ the set of conjugacy classes. For any number field $M
$ we denote by $d_M$ the absolute value of its absolute discriminant. The hypotheses below will depend on the extension $L/\Q$ rather than on $L/K$; as mentioned earlier (see also Example \ref{example ng}),
the Artin $L$-functions associated to irreducible characters of $\Gal(L/K)$ are not primitive in general. For $\chi \in \Irr(G)$, 
we will denote by $L(s,L/K,\chi)$ the associated Artin $L$-function (see \cite[Chapter 1, \S 4]{Mar} for a definition).

\begin{enumerate}

\item[\textbf{(AC)}] We assume Artin's holomorphicity conjecture
 which states that for every nontrivial $ \chi\in \Irr(G^+)$, the associated Artin $L$-function $L(s,L/\Q,\chi)$ is entire\footnote{It should be noted that Artin's holomorphicity conjecture is known for extensions whose Galois group is supersolvable, that is, when the Galois group $G$
 admits a sequence of \emph{normal subgroups of $G$}
 $$
 \{{\rm id}\}=H_0\subseteq H_1\subseteq\cdots \subseteq H_n=G
 $$
 with \emph{cyclic} successive quotients $H_i/H_{i-1}$. The irreducible representations of such groups are all monomial so that Brauer's Theorem implies that Artin's conjecture holds in this case (we refer the reader \emph{e.g.} to~\cite[Chap. 2]{MuMu} for further details). By the structure theorem for finite abelian groups we see in particular that any semi-direct product of an abelian group by a cyclic group is supersolvable. The examples we study in~\S\ref{section:AbelGalois} and~\S\ref{section:supersolvable} are all instances of supersolvable extensions and thus Artin's conjecture is known to hold for such examples.}. 

\item[\textbf{(GRH$^-$)}] 
We assume that for every $\chi \in \Irr(G^+)$, $\sup\{\Re(\rho): L(\rho,L/\Q,\chi)=0\}<1,$ and moreover
$L(s,L/\Q,\chi)$ has a zero on the line $\Re (s)=\sup\{\Re(\rho): L(\rho,L/\Q,\chi)=0\}$.

\item[\textbf{(GRH)}] 
We assume the Riemann Hypothesis for the extension $L/\Q$, that is every nontrivial zero of $L(s,L/\Q,\chi)$ lies on the line $\Re (s)=\frac 12$, for every $\chi \in \Irr(G^+)$.

\item[\textbf{(BM)}] 
We assume that every
nonreal zero of 
$$ \prod_{ \chi \in \Irr(G^+)}L(s,L/\Q,\chi)$$ up to height $(\log d_L\log \log d_L)^2$ has multiplicity at most $M_0$. Moreover, for each $\chi \in \Irr(G^+),$ 
$${\rm ord}_{s=\frac 12}L(s,L/\Q,\chi)\leq M_0.$$

\end{enumerate}

The following generalizes a classical and widely used hypothesis of Wintner~\cite{Win} on the diophantine properties of the zeros of the Riemann zeta function. As discussed in \cite[Section 5]{RubSar}, in the case of Artin $L$-functions 
it is quite delicate to state. First, it is believed that $L(s,L/\Q,\chi)$ is primitive whenever $\chi$ is irreducible. Second, we need to take into account the 
potential existence of real zeros. This is strongly linked to the Frobenius--Schur indicator of the corresponding character $\chi$ and the Artin root number of $L(s,L/\Q,\chi)$.
As illustrated in Example~\ref{example ng} below, those two properties do not necessarily hold for $L(s,L/K,\chi)$.

\begin{enumerate}

\item[(\textbf{LI$^-$})]
We assume that the multiset of positive imaginary parts of the 
 zeros of all Artin $L$-functions $L(s,L/\Q,\chi)$ in the region $\{s\in \C :\Re(s)\geq \tfrac 12\}$, with $\chi \in \Irr(G^+)$, are linearly independent over the rationals. 

\item[(\textbf{LI})] We assume LI$^-$. Moreover, we assume that $L(\frac 12,L/\Q,\chi)\neq 0$ if $\chi$ is an orthogonal or unitary irreducible character of $G^+$, and that for any symplectic irreducible character $\chi$ of $G^+$ one has the uniform bound ${\rm ord}_{s=\frac 12}L(s,L/\Q,\chi)\leq M_0$ (see Theorem~\ref{theorem Frobenius Schur} for the definition of orthogonal, unitary, and symplectic character). Finally, we assume that for every $\beta\in (0,1)\setminus \{ \tfrac 12\}$ and $ \chi \in \Irr(G^+)$, $L(\beta,L/\Q,\chi) \neq 0$.

\end{enumerate}

\begin{remark}
\label{remark LI}
We actually believe that a stronger statement holds (say LI$^+$), that is in addition to LI, for any symplectic character $\chi\in \Irr(G^+)$, we have that $${\rm ord}_{s=\frac 12}L(s,L/\Q,\chi)=\frac{1-W(\chi)}{2}$$
(in other words, $L(\frac 12,L/\Q,\chi)=0$ may only occur as a consequence of $W(\chi)=-1$).
Hypothesis LI$^{+}$ generalizes its counterpart for Dirichlet $L$-functions. 
In this case there is theoretical progress in~\cite{MN-AP} and~\cite{LR} (see also the 
very interesting discussions therein on linear independence properties of $L$-function zeros in general) as well as computational verification up to a fixed height (see~\cite{BT,MOT}). 
In the general case, the reason why Hypothesis LI includes a statement about vanishing at $s=\frac 12$ comes from the existence of Galois extensions $L/\Q$ with Galois group admitting a symplectic irreducible character $\chi$ of Artin root number $W(\chi)=-1$ (see \emph{e.g.}~\cite[Chapter 1, \S 4(ii)]{Mar} for the definition), 
so that\footnote{The first examples exhibiting such a real zero were found by Armitage \cite{Ar} 
and Serre (see~\cite[\S 5.3.3]{Ng}).} $L(\frac 12,L/\Q,\chi)=0$. 
It is known~\cite{FrQu} that in the general case of a Galois extension of number fields $L/K$ the Artin root number of an orthogonal irreducible representation is $1$. LI asserts that for unitary characters associated to $L/\Q$, one has $W(\chi)\neq -1$ (this is not necessarily true for relative extensions $L/K$; see Example \ref{example ng}).

\end{remark}

We now give an explicit example to illustrate that Artin $L$-functions attached to irreducible characters $\chi$ of relative extensions are not primitive in general, and might vanish at $s=\frac 12$ for reasons independent of their root number. 
\begin{example}[Serre, see \emph{e.g.} {\cite[\S 5.3.3]{Ng}}]
\label{example ng}
 Let $L=\Q(\theta)$ where $\theta$ is a root of the $\Q$-irreducible polynomial $x^8-205x^6+13\,940x^4-378\,225x^2+3\,404\,025$. Serre shows that $L/\Q$ is Galois of group isomorphic to the quaternion group $\mathbb{H}_8$ of order $8$ and moreover that the only non-abelian irreducible character (denoted $\chi_5$ in
\emph{loc. cit.}) of $G=\Gal(L/\Q)\simeq\mathbb{H}_8$ is symplectic of degree $2$ and satisfies $W(\chi_5)=-1$ so that $L(\frac 12,L/\Q,\chi_2)=0$. There are $5$ irreducible characters of $\mathbb{H}_8$ all real valued; $4$ of them have degree $1$ and thus correspond to the Kronecker symbol attached to a fundamental discriminant computed by Ng. Artin's factorisation property gives rise to the following decomposition of the Dedekind zeta function of $L$:
\begin{align*}
\zeta_L(s)&=\prod_{\chi\in \Irr(G)}L(s,L/\Q,\chi)^{\chi(1)}\\
&=\zeta(s)L\bigg(s,L/\Q,\Big(\frac{5}{\cdot}\Big)\bigg)L\bigg(s,L/\Q,\Big(\frac{41}{\cdot}\Big)\bigg)L\bigg(s,L/\Q,\Big(\frac{205}{\cdot}\Big)\bigg)L(s,L/\Q,\chi_5)^2.
\end{align*}
Ng numerically checks the nonvanishing at $\frac 12$ of the three Dirichlet $L$-functions of quadratic characters appearing above so that $L(s,L/\Q,\chi_5)$ is entirely responsible for the vanishing of $\zeta_L$ at $\frac 12$. There are $3$ quadratic subextensions of $L/\Q$ with respective discriminant $5$, $41$ and $205$; if we fix one such discriminant $D$ then the corresponding subfield $K_D$ of $L$ has the property that $G=\Gal(L/K_D)$ is cyclic of order $4$. Thus $G$ has $4$ irreducible representations of degree $1$; two of them are orthogonal (the trivial representation and a character of order $2$), and two of them (denoted $\psi$ and $\bar{\psi}$) are unitary. A straightforward group theoretic computation shows that 
$$
{\rm Ind}_G^{\mathbb H_8}\psi={\rm Ind}_G^{\mathbb H_8}\bar{\psi}=\chi_5\,.
$$
By properties of Artin root numbers and Artin $L$-functions (see \emph{e.g.}~\cite[\S 4]{Mar}) 
we have $W(\psi)=W(\bar{\psi})=W(\chi_5)=-1$ and $L(\frac 12,L/K_D,\psi)=L(\frac 12,L/K_D,\bar{\psi})=L(\frac 12,L/K_D,\chi_5)=0$. Therefore Serre's example shows that in the case of a \emph{relative} 
extension of number fields $L/K$ one may have $L(\frac 12,L/K,\chi)=0$ for a \emph{unitary} representation of $\Gal(L/K)$. What assumption LI asserts in this case is that vanishing at $\frac 12$ for $L(s,L/K,\chi)$ is explained by the symplectic irreducible representation of root number $-1$ that appears in the character induced by $\chi$ on the Galois group of the normal closure of $L$ (in the example, it is $L$ itself) over $\Q$.

Another interesting phenomenon that the same example illustrates is the potential \emph{multiplicity} of $L$-factors in a relative extension, for a given $L$-function. Indeed, let $Z=\{\pm 1\}$ be the center of $\mathbb{H}_8$ and consider this time 
the quadratic extension $L/L^Z$. Let $\epsilon$ be the nontrivial character of $\Gal(L/L^Z)$. This gives rise to a new factorization of $\zeta_L(s)$:
\begin{align*}
\zeta_L(s)&=\zeta(s)L\bigg(s,L/\Q,\Big(\frac{5}{\cdot}\Big)\bigg)L\bigg(s,L/\Q,\Big(\frac{41}{\cdot}\Big)\bigg)L\bigg(s,L/\Q,\Big(\frac{205}{\cdot}\Big)\bigg)L(s,L/\Q,\chi_5)^2\\&=L(s,L/L^Z,1)L(s,L/L^Z,\epsilon)\,.
\end{align*}
The factor $L(s,L/L^Z,1)$ is the Dedekind zeta function of $L^Z$, an abelian extension of $\Q$ of degree $4$. Therefore
$$
L(s,L/L^Z,1)=\zeta(s)L\bigg(s,L/\Q,\Big(\frac{5}{\cdot}\Big)\bigg)L\bigg(s,L/\Q,\Big(\frac{41}{\cdot}\Big)\bigg)L\bigg(s,L/\Q,\Big(\frac{205}{\cdot}\Big)\bigg)
$$
and in turn
$$
L(s,L/L^Z,\epsilon)=L(s,L/\Q,\chi_5)^2\,.
$$
We deduce the existence of orthogonal representations with associated $L$-function vanishing at $\tfrac 12$, and we also see that 
the multi-set of critical zeros of $L(s,L/L^Z,\epsilon)$ has repeated elements.
\end{example}

\section*{Acknowledgements}
The work of the first author was supported by a Postdoctoral Fellowship as well as a Discovery Grant from the NSERC, and a Postdoctoral Fellowship from the Fondation Sciences Math\'ematiques de Paris. The work of both authors was partly funded by the ANR through project FLAIR (ANR-17-CE40-0012). 
We thank B. Allombert, P. Autissier, M. Balazard, M. Bardestani, K. Belabas, E. Bombieri, L. Devin, \'E. Fouvry, M. Hayani, C. Meiri, S. D. Miller, N. Ng, C. Pomerance, P. Sarnak and G. Tenenbaum for very fruitful conversations. We are especially thankful to A. Bailleul for his patience, careful reading and numerous helpful remarks. He notably spotted a serious mistake in a preliminary version of this work.
This work was accomplished while the first author was at the Institute for Advanced Study, Universit\'e Paris Diderot, University of Ottawa and Universit\'e Paris-Saclay, and while the second author was at Universit\'e Paris-Saclay, ENS Paris and  Universit\'e de Bordeaux.  We would like to thank these institutions for their hospitality.

\section{Statement of results} \label{section:results}

We consider a Galois extension $L/K$ of number fields, and set $G=\Gal(L/K)$. If $L/\Q$ is also Galois, then we write $G^+=\Gal(L/\Q)$. 
We let $G^\sharp$ denote the set of conjugacy classes of $G$, and $\Irr(G)$ denote its set of irreducible characters. Given $\chi \in \Irr(G)$ and a class function $t:G\rightarrow \C$, we define the Fourier transform
$$ \widehat t(\chi) := \langle \chi, t \rangle_G = \frac 1{|G|}\sum_{g\in G} \chi(g) \overline{t(g)}\,,$$
its support ${\rm supp}(\widehat{t}):=\{\chi\in\Irr(G)\colon \langle \chi,t\rangle_G\neq 0\}$, as well as the norms 
\begin{equation}
\label{equation definition norms}
 \norm{t}_1:=\frac 1{|G|}\sum_{g\in G} |t(g)|; \hspace{2cm}  \norm{t}^2_2:= \langle t,t\rangle_G =\frac 1{|G|}\sum_{g\in G} |t(g)|^2.
\end{equation}
Note that for class functions $t_1,t_2:G\rightarrow \C$, Parseval's identity reads 
\begin{equation} \label{eq:parseval}
 \overline{\langle t_1 ,t_2\rangle_G} = \langle \widehat {t_1} ,\widehat{t_2} \rangle_{\Irr(G)} := \sum_{\chi \in \Irr(G)} \widehat{t_1}(\chi)\overline{\widehat{t_2}(\chi)}\,;
 \end{equation}
in particular $\norm{\widehat{t}}_2 = \norm{t}_2.$
We also consider the Littlewood norm~\cite[\S 1.2, (1)]{Be}
$$\lambda(t) := \sum_{\chi \in \Irr(G)} \chi(1)|\widehat t(\chi)|.$$
In the case where $L/\Q$ is Galois, we extend $t$ to the well defined class function $t^+={\rm Ind}_G^{G^+}(t):G^+ \rightarrow \C$ that satisfies for all $g\in G^+$,
\begin{equation}
 t^+(g)= \sum_{\substack{aG\in G^+/G : \\a^{-1} g a \in G }} t(a^{-1} g a). 
 \label{equation definition t+}
\end{equation}
We also extend conjugacy classes $C\in G^\sharp$ to well-defined\footnote{See Remark~\ref{remark C+ is a cc}. Note that we may not apply this definition to $G$ itself, since it is not a conjugacy class (unless $|G|=1$). } conjugacy classes of $G^+$ by setting
\begin{equation}
C^+:=\bigcup_{aG\in G^+/G} a C a^{-1}.
\label{equation definition C+}
\end{equation} 
We consider the Frobenius counting function 
$$ \pi(x;L/K,t):=\sum_{\substack{\mathfrak p\triangleleft \mathcal O_K \text{ unram.}\\ \mathcal N\mathfrak p \leq x}} t({\rm Frob}_\mathfrak p),  $$
as well as its normalization\footnote{The reason why we work on the logarithmic scale is explained in \cite{Ka}.}
\begin{equation}
  E(y;L/K,t):= 
  y\ee^{-y\beta^t_{L}} (\pi(\ee^y;L/K,t) -  \widehat t(1){\rm Li}(\ee^y)),
  \label{equation definition E}
\end{equation}
where, for $t\not \equiv 0$, 
\begin{equation}
\label{equation definition beta}
\beta^t_{L}= \begin{cases}
 \sup\{ \Re(\rho) : L(\rho,L/\Q,\chi)=0 ; \chi \in {\rm supp}(\widehat{t^+})  \} & \text{ if AC holds for } L/\Q ;\\
  \sup\{ \Re(\rho) : \zeta_L(\rho)=0 \} &\text{ otherwise}.
\end{cases}
\end{equation}
(If $t\equiv 0$, we set $\beta^t_{L}=\frac 12$.)
If $t$ is real-valued, then we also define the densities
\begin{equation} \label{eq:densitiest}
 \overline{\underline{\delta}}(L/K;t):= \overline{\underline{\lim}}_{Y\rightarrow \infty} \frac{{\rm meas} \{ y\leq Y :  E(y;L/K,t)>0 \}}{Y}\,, 
 \end{equation}
which will measure to which extent a constant lower order term is dominating the fluctuations of the error term in the Chebotarev density theorem.
If the upper and lower limits coincide, 
their common value are denoted by 
$\delta(L/K;t)$.

The prime example of class function we will consider is 
$$ t_{C_1,C_2}:=\frac{|G|}{|C_1|}1_{C_1} - \frac{|G|}{|C_2|}1_{C_2},$$
 where $C_1,C_2\in G^\sharp $ are distinct and for any conjugacy invariant set $D\subset G$, $1_D$ is the indicator function of $D$. By convention, we will allow $C_2$ to be equal to $0$, in which case we define $t_{C_1,0}:= |G||C_1|^{-1}1_{C_1}$, and we write $C_2^+=0, |C_1^+|^{-1}+|C_2^+|^{-1} := |C_1^+|^{-1}$.

In \cite{RubSar}, it is noted that there are discrepancies in the distribution of primes in residue classes $a \bmod q$ towards values of $a$ that are quadratic nonresidues. Accordingly, the authors considered the distribution of the natural counting function 
\begin{equation}
 \#\{ p\leq x : p\not \equiv \square \bmod q\} - \#\{ p\leq x : p \equiv \square \bmod q\} 
 \label{equation RS NR vs R}
\end{equation} 
for moduli $q$ for which there exists a primitive root modulo $q$. For general moduli one should add a weight~\cite{Fi1}, \emph{e.g.} if $q=15$ one should consider $\pi(x;15,2)+\pi(x;15,7)+\pi(x;15,8)+\pi(x;15,11)+\pi(x;15,13)+\pi(x;15,14)-3\pi(x;15,1)-3\pi(x;15,4)$; note that $-3=1-\#\{x\bmod 15\colon x^2=1\}=1-\#\{x\bmod 15\colon x^2=4\}$. We generalize this by considering the class function
\begin{equation}
r(g)=r_G(g) := \#\{ h\in G :  h^2 =g\}.
\label{equation definition r(C)}
\end{equation}
Then, $\pi(x;L/K,1-r)$ is the natural generalization of the counting function~\eqref{equation RS NR vs R}. For the concrete example of the different possible values of the weight $1-r({ \rm Frob}_p)$ in the case $G=S_6$, we refer the reader to Table \ref{table S_6}. 

In the following table we highlight three important particular cases of class functions, where $C,C_1,C_2 \in G^\sharp$ and $C_1\neq C_2$. Here, we compute $t^+:G^+\rightarrow \C$ using~\eqref{equation 1C+}, and $\widehat{t^+}$ using Frobenius reciprocity and Lemma~\ref{lemma elements of order m} (the case $k=2$ of~\eqref{equation definition epsilon m} gives the definition of the Frobenius--Schur indicator $\epsilon_2$). Note also that for any $C\in G^\sharp$ and $\chi \in \Irr(G^+)$, $\chi(C^+)=\chi|_G(C)$.
\begin{center}
\begin{tabular}{c|c|c|c}
$t$& $\pi(x;L/K,t)$ & $t^+$  & $\widehat{t^+}(\chi)$\\
\hline
$t_{C,0}$ & $\frac{|G|}{ |C|}\pi(x;L/K,C) - {\rm Li}(x)$ &  $t_{C^+,0}$& $\chi(C^+)$  \\
$t_{C_1,C_2}$ & \tiny{ $\frac{|G|}{ |C_1|}\pi(x;L/K,C_1) - \frac{|G|}{ |C_2|}\pi(x;L/K,C_2)$}&  $t_{C_1^+,C_2^+}$ & $\chi(C_1^+)-\chi(C_2^+)$  \\
$1-r$ & $\pi(x;L/K,1-r)$ & $\frac{|G^+|}{|G|}\underset{{D\in G^\sharp}}{\sum} \frac{|D|}{|D^+|}1_{D^+}-r^+$ & $\underset{{D\in G^\sharp}}{\sum}\frac {|D|}{|G|} \chi(D^+)-\epsilon_2(\chi|_G)$
\end{tabular}
\end{center}
We also compute the corresponding inner products and norms which will appear later.
\begin{center}
\begin{tabular}{c|c|c|c}
$t$  & $-\langle t,r \rangle_G $ & $\norm{t^+}_2$ & $\norm{t^+}_1$ \\
\hline
$t_{C,0}$& $-r(C) $  & $\frac{|G^+|^{\frac 12 }}{|C^+|^{\frac 12}}$ & $1$ \\
$t_{C_1,C_2}$& $r(C_2)-r(C_1)$ &  $\big(\frac{|G^+|}{|C_1^+|}+\frac{|G^+|}{|C_2^+|}\big)^{\frac 12}$ & $2$
\end{tabular}
\end{center}
Finally, in the case $K=\Q$, we have that\footnote{Here, $\widehat G_{\r}$ denotes the set of real irreducible characters of $G$.} $-\langle 1-r,r \rangle_G = |\widehat G_{\r}|-1$,   $\norm{1-r}_2=(|\widehat G_{\r}|-1)^{\frac 12}$,  and $\norm{1-r}_1 = 2-2|G|^{-1} \#\{ g\in G : r(g)\geq 1 \}$.
Note that if $|G|$ is odd, then $1-r \equiv 0$.

\subsection{General Galois extensions}\label{subsec:general}

As mentioned in the introduction, we will translate fine distribution properties of Frobenius elements in terms of the 
 representation theory of $G=\Gal(L/K)$ and the ramification data of $L/K$. In this section we state the precise results spelling out this idea. We refer the reader to~\cite[(5.2)]{LaOd}
   for the definition of the Artin conductor $A(\chi)$. Moreover, we let ${\rm rd}_L$ be the root discriminant of $L$, that is 
\begin{equation}\label{eq:defrd}
{\rm rd}_L = d_L^{\frac 1{[L:\Q]}}\,,
\end{equation}
where we recall that $d_L$ is the absolute value of the absolute discriminant of $L$. For convenience, we associate to any class function $t:G\rightarrow \C$ a formal object $L(s,L/K,t)$ for which we define the log derivative by extending the case of Artin $L$-functions:
\begin{equation}
\frac{L'(s,L/K,t)}{L(s,L/K,t)}:= \sum_{\chi \in \Irr(G)} {\overline{\widehat t(\chi)}} \frac{L'(s,L/K,\chi)}{L(s,L/K,\chi)}\,. 
 \label{equation definition generalized Artin L function}
\end{equation}
Accordingly, we define the order of vanishing at some $s_0\in\C$ as follows:
\begin{equation}
{\rm ord}_{s=s_0} L(s,L/K,t) := \sum_{\chi\in\Irr(G)} {\overline{\widehat t(\chi)}}\cdot {\rm ord}_{s=s_0} L(s,L/K,\chi)\,.
\label{equation definition order vanishing}
\end{equation}

Our first main result is the following. We say that the function $E: \mathbb R_+\rightarrow \mathbb C$ admits the limiting distribution $\nu$ if $\nu$ is a probability measure on $\mathbb C$ such that for any bounded continuous function $f:\mathbb C\rightarrow \mathbb R$,
$$ 	\lim_{Y\rightarrow \infty} \frac 1Y\int_0^Y f(E(y)) {\rm d}y  = \int_{\mathbb C} f {\rm d} \nu. $$
The mean and variance of the associated random variable $Z_\nu$ are defined by
$$  \E[Z_\nu] =\E[\Re(Z_\nu)]+i \E[\Im(Z_\nu)]; \qquad   \V[Z_\nu] =\E[|Z_\nu-\E[Z_\nu]|^2]. $$

\begin{theorem}
\label{theorem mean variance}

Let $L/K$ be a Galois extension of number fields, and fix a class function $t:G\rightarrow \C$. Recall that $\beta_L^t$ is defined by~\eqref{equation definition beta} and that the class function $r$ is defined by~\eqref{equation definition r(C)}.
Then,
 $E(y;L/K,t)$ 
 admits a limiting distribution whose mean is 
$$\mu_{L/K,t}= -\langle t,r \rangle_G \delta_{\beta^t_L=\frac 12}-\frac 1{\beta_L^t}{\rm ord}_{s=\beta_L^t}L(s,L/K,t),$$
where $\delta$ is Kronecker's delta, and whose variance is
\begin{equation}
\sigma^2_{L/K,t}\ll  \norm{t}_1\log (d_L+2) \min(M_{L},\log (d_L+2)),
\label{equation first upper bound variance}
\end{equation}
 where 
$M_{L} := \max \big \{  {\rm ord}_{s=\rho}  \zeta_L(s)  : \Re(\rho)=\beta_L^t, 0<|\Im(\rho)|< \log (d_L+2)(\log\log (d_L+2))^2\big\}. $
In other words, for generic values of $y$ we have the estimate 
\begin{equation}
E(y;L/K,t)= \mu_{L/K,t}+O(\norm{t}^{\frac 12}_1 (\log (d_L+2))^{\frac 12} \min(M_{L},\log (d_L+2))^{\frac 12}). 
\label{equation improved bound chebotarev}
\end{equation}
Assuming that $L/\Q$ is Galois and that AC holds, we have the more precise\footnote{Keeping in mind that BM implies the bounds $m_{L}\ll 1$, $M_L\ll \max_{\chi \in \Irr(G^+)} \chi(1)$, compare this bound with~\eqref{equation true order magnitude variance}. Note also that the sum over characters in \eqref{equation upper bound variances in main theorem} is $ \ll \norm{t^+}_2^2\max_{\chi \in \Irr(G^+)}\chi(1)$. }
 bound
\begin{equation}
\sigma^2_{L/K,t}\ll (m_{L})^2  \log({\rm rd}_L+2)\sum_{ \chi\in \Irr(G^+) } \chi(1)|\widehat{t^+}(\chi)|^2.
\label{equation upper bound variances in main theorem}
\end{equation}
Here, 
denoting $T_{L}:= \log({\rm rd}_L+2) \sum_{\substack{\chi \in \Irr(G^+) }} \chi(1) $,
$$m_{L} := \max \Big \{  {\rm ord}_{s=\rho}  \Big(\prod_{\substack{\chi \in {\rm supp}(\widehat{t^+}) }} L(s,L/\Q,\chi)\Big)  : \Re(\rho)=\beta_L^t,0<|\Im(\rho)|< (T_L\log T_L)^2\Big\}.$$
Assuming moreover GRH and\footnote{In the particular case $(C_1,C_2)=(1,0)$, we do not need LI here.} LI, we have the lower bound
\begin{equation}
\sigma^2_{L/K,t}\gg \sum_{ \chi\in \Irr(G^+) } \chi(1)|\widehat{t^+}(\chi)|^2.
\end{equation}
Finally, assuming in addition that each irreducible representation of $G^+=\Gal(L/\Q)$ of dimension $\geq \norm{t^+}_2 \norm{t^+}_1^{-1} (2\#\Irr(G^+))^{-\frac 12}$ satisfies the bound $\max_{1\neq C \in (G^+)^\sharp} |\chi(C)| \leq (1-\eta)\chi(1)$ for some real number $0<\eta <1$ which depends on $t$ and $G^+$, then we have that
\begin{equation}
\sigma^2_{L/K,t}\gg \eta\log({\rm rd}_L+2)\sum_{ \chi\in \Irr(G^+) } \chi(1)|\widehat{t^+}(\chi)|^2.
\label{equation true order magnitude variance}
\end{equation}
If $\widehat t\not \equiv 0$, then the character sum in~\eqref{equation true order magnitude variance} satisfies the general bounds
$$  \frac{\norm{t^+}_2^3}{\norm{t^+}_1(\#{ \rm supp}(\widehat{t^+}))^{\frac 12}} \leq \sum_{ \chi\in \Irr(G^+) } \chi(1)|\widehat{t^+}(\chi)|^2 \leq |G^+|^{\frac 12} \norm{t^+}_2. $$
 \end{theorem}

\begin{remark}
The error term in~\eqref{equation improved bound chebotarev} is significantly sharper than that in~\eqref{equation Lagarias Odlyzko} as well as the further refinements of Murty--Murty--Saradha and Bella\"iche. As a matter of comparison, taking $q\geq 3$, $L=\Q(\zeta_q)$, $K=\Q$, $t=\phi(q) 1_a$ for some $a\in (\Z/q\Z)^\times$ and assuming BM,~\eqref{equation improved bound chebotarev} translates to 
$$ E(y;\Q(\zeta_q)/\Q,t)   = \mu_{\Q(\zeta_q)/\Q,t} + O((q\log q)^{\frac 12}), $$
which under GRH  is of the strength of Montgomery's Conjecture on primes in arithmetic progressions (since $E(y;\Q(\zeta_q)/\Q,t) = y\ee^{\frac y2}\phi(q)(\pi(\ee^y;q,a)-\phi(q)^{-1}{ \rm Li}(\ee^y))$). This answers a question of Murty--Murty--Saradha~\cite[\S 3.13]{MMS} about the ``true size'' of this error term, at least for generic and large enough values of $y$. A detailed generalization of Montgomery's conjecture with a range of validity and various applications will appear in a forthcoming paper joint with Morrison and Thorner.
Comparing this with~\cite[Théorème 1]{Be} (which holds for \emph{all} $y$), we see that for an extension $L/\Q$ with $G=\Gal(L/\Q)$, the bound on $E(y;L/\Q,t)$ in \emph{loc. cit.} is
\begin{align*}
\gg
y \frac{ \log(\ee^y {\rm rd}_L)}{(\log({ \rm rd}_L+2))^{\frac 12}}\frac{\sum_{\chi \in \Irr(G)}  \chi(1)|\widehat t(\chi)|}{\Big(\sum_{\chi \in \Irr(G)}  \chi(1)|\widehat t(\chi)|^2\Big)^{\frac 12} }   s_{L/\Q,t}\,, 
\end{align*} 
where $ s_{L/\Q,t}$ is the bound on the typical size of the error term of $E(y;L/\Q,t)$ in Theorem~\ref{theorem mean variance}. One can do a similar comparison with~\cite{MMS} for relative extensions.

As for our similar looking estimates~\eqref{equation first upper bound variance} and~\eqref{equation upper bound variances in main theorem}, we see that for the extension $\Q(\zeta_q)/\Q$, they are of the same quality, since $\chi(1)=1$. However, if for example we work with the class function $t_{C_1,C_2}$ in a family where $\Gal(L/\Q)=S_n$, then the ratio between~\eqref{equation first upper bound variance} and~\eqref{equation upper bound variances in main theorem} is $\gg \min(|C_1|,|C_2|)$. As a more extreme example, we will see in Theorem~\ref{th:BiasRadical} that there exist extensions for which the upper bound in~\eqref{equation upper bound variances in main theorem} is identically zero. 
As this suggests, to fully understand the fluctuations of $E(y;L/K,t)$, it is not sufficient to decompose it using the characters of the group $G$ -- zeros that are either multiple or common to different characters significantly affect the formula for the variance. 
 To take this into account, we will formulate a transfer principle relating $E(y;L/K,t)$ to $E(y;L/\Q,t^+)$ in Proposition~\ref{proposition psi bridge} and Corollary~\ref{cor:explicitformulae1} (see also Remark~\ref{remark variance small group}).

An interesting consequence of Theorem~\ref{theorem mean variance} (more precisely of Proposition~\ref{proposition link with random variables}, combined with~\eqref{equation link sum real zeros big to small group} and Lemma \ref{lemma elements of order m}) is that under AC and GRH, the limiting distributions of the functions $y\ee^{-y/2} (|G^+| \pi(\ee^y;L/\Q,\{{\rm id}\}) - {\rm Li}(\ee^y))$ and $y\ee^{-y/2} (|G| \pi(\ee^y;L/K,\{{\rm id}\})-{\rm Li}(\ee^y))$ have the same variance, however the mean of the first is always less than or equal to that of the second. 
\end{remark}
We now discuss applications of our ideas. We first focus on Linnik type problems for Frobenius sets. Lagarias, Montgomery and Odlyzko~\cite{LMO} showed that under GRH and for a given extension of number fields $L/K$ and for any conjugacy class $C\subset \Gal(L/K)$, there exists an unramified prime ideal $\mathfrak p \triangleleft \mathcal O_K$ of norm $ \ll (\log (d_L+2))^2 $ for which ${\rm Frob}_{\mathfrak p}=C$. Bella\"iche~\cite[Proposition 1]{Be2} has shown that in the case $C=\{ {\rm id} \}$, the exponent $2$ in this bound is best possible (see also~\cite{Fio}). However, K. Murty conjectured~\cite[Conjecture 2.2]{Mu2} that under GRH,
we have the general bound 
  $ \ll (\log (d_L+2))^2/|C| $, which decreases when $|C|$ grows. 
   This conjecture was motivated by~\cite[Theorem 3.1]{Mu2}, which shows that under the Riemann Hypothesis and Artin's Conjecture for every $L(s,L/K,\chi)$ with $\chi \in \Irr(G)$, we have the bound
$ \ll [K:\Q]^2([L:K]\log [L:K] +\log (d_L+2))^2/|C| $ (the additional factors here come in part from the contribution of ramified prime ideals). In the case where $K=\Q$ and the condition that $\mathfrak p \triangleleft \mathcal O_K$ is unramified is dropped, Bella\"iche~\cite[Théorème 3]{Be} showed that under AC and GRH, one can obtain sharper results in several important families. More precisely, one can obtain a bound in terms of the invariant 
\begin{equation}
\varphi_G(C) := \inf\Big\{ \frac{\lambda(t)}{\widehat t(1)} \,\,\Big|\,\, t:G\rightarrow \mathbb R ;\widehat t(1)>0; t(g) > 0\Rightarrow  g\in C\Big\} \ll \frac{|G|}{ |C|^{\frac 12}}. 
\label{equation bellaiche invariant}
\end{equation}
(the bound follows from taking $t=1_C$ and applying Cauchy--Schwarz.) 
 We are now ready to state our bounds on the least unramified prime ideal in a given Frobenius set.
\begin{theorem}

\label{theorem least prime ideal}
Let $L/K$ be a Galois extension of number fields and assume that the Riemann Hypothesis and Artin's Conjecture hold for each $L(s,L/K,\chi)$ with $\chi \in \Irr(\Gal(L/K))$. Then, Murty's conjecture holds. In other words, for any conjugacy class $C\subset  G$ there exists an unramified prime ideal $\mathfrak p \triangleleft \mathcal O_K$ for which $\frob_{\mathfrak p} = C$ and
\begin{equation}
 \mathcal N\mathfrak p\ll \frac{(\log (d_L+2))^2}{|C|}. 
 \label{equation theorem least ideal first bound}
\end{equation} 
More precisely, taking into account Bella\"iche's refinement\footnote{In particular, one can apply this bound to the class functions described in~\cite[Definition 1]{Be}.}, for any class function $t:G \rightarrow \mathbb R$ such that $\widehat t(1)> 0$, there exists an unramified prime ideal $\mathfrak p \triangleleft \mathcal O_K$ for which $t(\frob_\mathfrak p)>0$ and 
\begin{align}
 \mathcal N\mathfrak p\ll \Big( \frac{\lambda(t)}{\widehat t(1)}\log({\rm rd}_L+2) [K:\Q] \Big)^2.
 \label{equation theorem least ideal second bound}
\end{align}
If in addition $L/\Q$ is Galois and AC holds, then 
there exists an unramified prime ideal $\mathfrak p \triangleleft \mathcal O_K$ for which $\frob_{\mathfrak p} = C$ and
\begin{equation}
\mathcal N\mathfrak p \ll  
   \frac{  (\log (d_L+2) )^2}{|C^+|}+ \frac{  (\log (d_K+2) )^{\frac 43} |G|^{\frac 43}}{|C|^{\frac 23}}.
   \label{equation G+ improvement least ideal simple}
\end{equation}
(Note that the second term in~\eqref{equation G+ improvement least ideal simple} is $\ll (\log (d_L+2))^{\frac 43}  |C|^{-\frac 23} \ll (\log (d_L+2))^{2}  (|G^+||C|)^{-\frac 23}  $.)
Finally, under the same hypotheses and incorporating Bella\"iche's refinement, we obtain that for any class function $t:G \rightarrow \mathbb R$ such that $\widehat t(1)>|G|^{-100}\sup|t|$, there exists an unramified prime ideal $\mathfrak p \triangleleft \mathcal O_K$ for which $t(\frob_\mathfrak p)>0$ and 
\begin{multline}
\mathcal N\mathfrak p \ll \Big( \frac{\lambda(t^+)}{\widehat {t}(1)}\log({\rm rd}_L+2)  \Big)^2 + \frac{\lambda(t)}{\widehat t(1)}[K:\Q] \log({\rm rd}_L+2) \\+\sum_{\substack{2\leq \ell \ll \log\log d_L  \\ \mu^2(\ell)=1}} \Big(\Big( \frac{|\langle t,r_\ell \rangle_G|}{\widehat t(1)} \Big)^{\frac{\ell}{\ell-1}}+\Big(\frac{\lambda((t(\cdot^\ell) )^+)}{\widehat t(1)}\log({\rm rd}_L+2) \Big)^{\frac{2\ell}{2\ell-1}}\Big).
 \label{equation improved bound least ideal with G+ }
 \end{multline}

\end{theorem}
\begin{example}
As an example in which~\eqref{equation improved bound least ideal with G+ } and~\eqref{equation G+ improvement least ideal simple} are significantly sharper than~\eqref{equation theorem least ideal first bound} and~\eqref{equation theorem least ideal second bound}, consider any $S_n$ extension $L/\Q$ and $K=L^{\langle(12\cdots n)\rangle}$. Clearly for $\sigma=(1\,2\cdots n)$, one has
$|\{\sigma\}^+|  = (n-1)!, $ and likewise 
 for any $k$ coprime to $n$, taking $C=\{\sigma^k\}$ we have that 
$|C|=1$
and $|C^+|=(n-1)!$. Thus our bound on the unramified prime ideal $\mathfrak p$ of least norm for which $\frob_{\mathfrak p}=C$ is $ \mathcal N\mathfrak p\ll (\log (d_L+2))^2/n!^{\frac 23} $. 
In comparison, the bounds~\eqref{equation theorem least ideal first bound} and~\eqref{equation theorem least ideal second bound} are both $\asymp(\log (d_L+2))^2$ (see~\cite[Proposition 17]{Be}).

More generally, considering the extension $L/L^H$ where $H \triangleleft S_n$ is a subgroup containing an element $h $ of cycle type $\lambda = (\lambda_1,\dots \lambda_k) \vdash n$, the bound~\eqref{equation G+ improvement least ideal simple} is 
$$\ll (\log (d_L+2))^2 \Big( \frac{\prod_{1\leq j \leq n} j^{a_j} a_j!}{n!}+ \frac 1{n!^{\frac 23}}\Big),
$$ where $a_j = \# \{ i\leq k : \lambda_i=j\}$. 
\end{example}
\begin{remark}
 One can give a simple heuristic argument that shows why we expect the bound~\eqref{equation G+ improvement least ideal simple} rather than~\eqref{equation theorem least ideal first bound}. If $L$ and $K$ are both Galois over $\Q$ and $p$ is a prime number that splits completely in $K$ and for which $\frob_p=C^+$, then for any prime ideal $\mathfrak p \triangleleft \mathcal O_K$ above $p$, we have that $\frob_{\mathfrak p} =C$. 
\end{remark}

Using similar arguments as in Theorem~\ref{theorem least prime ideal}, we obtain a refinement of the Lagarias--Odlyzko--Serre, Murty--Murty--Saradha and Bella\"iche bounds on the error term in Chebotarev's density theorem.
\begin{theorem}
\label{theorem chebotarev all x}
Let $L/K$ be a Galois extension of number fields for which $L/\Q$ is Galois, and assume AC and GRH. Then for all $x\geq 2$ we have the bound
\begin{multline}
 \pi(x;L/K,t) - \widehat t(1){\rm Li}(x) \ll \lambda(t^+)x^{\frac 12} \log ( { \rm rd}_L x)\log x\\+  \sum_{\substack{2\leq \ell \leq 2\log x \\ \mu^2(\ell)=1}}(x^{\frac 1\ell} |\langle  t,r_\ell\rangle_G|+x^{\frac 1{2\ell}} \lambda((t(\cdot^\ell))^+) \log ( { \rm rd}_Lx )\log x).
\label{equation bound improved chebotarev } 
 \end{multline}
 Moreover, the quantity $\lambda((t(\cdot^\ell))^+)$ can be replaced by $[K:\Q]\lambda(t(\cdot^\ell))$. 
In the particular case $t=|G||C|^{-1}1_C$ where $C\subset G$ is a conjugacy class, the right hand side of~\eqref{equation bound improved chebotarev } is
$\ll ( |G^+||C^+|^{-\frac 12}  x^{\frac 12} +|G^+||G|^{\frac 12}|C|^{-1} x^{\frac 14} )\log( {\rm rd}_L x) \log x.$ Here, 
 $G^+:=\Gal(L/\Q)$ and $C^+$ is defined by~\eqref{equation definition C+}.
\end{theorem}
Next we turn to applications of our results to discrepancies in the distribution of Frobenius elements in conjugacy classes. We will combine Theorem~\ref{theorem mean variance} with estimates on Artin conductors (see Lemma~\ref{lemma first bound on Artin conductor}) and probabilistic bounds on large deviations of random variables to detect when $\delta(L/K;t)$ (see~\eqref{eq:densitiest}) is very close to $1$, conditionally on AC, GRH and BM.

\begin{theorem}
\label{theorem general criterion conjugacy classes biased}
Let $L/K$ be an extension of number fields such that $L/\Q$ is Galois, and fix $t:G\rightarrow \R$ a class function such that\footnote{If $\langle t ,r \rangle_G >0$, then we may apply the theorem to $-t$ and deduce that $\overline{\delta}(L/K;t)$ is close to $0$.} $\langle t ,r \rangle_G <0$ and $t^+\not\equiv 0$.
Assume that AC, GRH, and BM hold. If for some small enough $\epsilon >0$ the inequality\footnote{See~\eqref{equation definition epsilon m} and Theorem~\ref{theorem Frobenius Schur} for the definition and properties of the Frobenius-Schur indicator $\epsilon_2$.}
\begin{equation} 
-\langle \widehat {t} , \epsilon_2 \rangle_{\Irr(G)}
-2\ord_{s=\frac 12} L(s,L/K,t)>
\Big(\epsilon^{-1} \log ({\rm rd}_L+2) \sum_{\chi \in \Irr(G^+)} \chi(1) |\widehat{t^+}(\chi)|^2\Big)^{\frac 12}
\label{equation:conditionHighBiasC1C2}
\end{equation}
holds, then the fluctuations of $E(y;L/K,t)$ are dominated by a constant term, that is
$$\underline{\delta}(L/K;t) > 1-c_1\epsilon. $$ Under the additional assumption LI, we have the refined bound
 $$\delta(L/K;t) > 1-\exp(-c_2\epsilon^{-1}). $$
 Finally, if $K=\Q$ and $|\widehat{t}(\chi)|\in \{0,1\}$ for all $\chi\in \Irr(G)$, then we also have the upper bound
  $$\delta(L/\Q;t) < 1-\exp(-c_3\epsilon^{-1}). $$
Here, $c_1,c_2,c_3>0$ are absolute constants.
\end{theorem}

As a partial converse to Theorem \ref{theorem general criterion conjugacy classes biased}, we show using an effective central limit theorem that up to the factor $\log ({\rm rd}_L+2)$, the condition~\eqref{equation:conditionHighBiasC1C2} is also sufficient. Here, the condition LI is required since we need a lower bound on the variance and an estimate on higher moments.

\begin{theorem}
\label{theorem general criterion conjugacy classes not biased}
Let $L/K$ be an extension of number fields such that $L/\Q$ is Galois, and for which AC, GRH$^-$, and
LI hold. 
Fix a class function $t:G\rightarrow \R$ such that $\widehat{t^+}\not \equiv 0$, 
and let $\epsilon>0$ be small enough. If the condition
\begin{equation}
\Big(\langle \widehat {t} , \epsilon_2 \rangle_{\Irr(G)}
+2\ord_{s=\frac 12} L(s,L/K,t) \Big)^2 < \epsilon^2 \sum_{\chi \in \Irr(G^+)} \chi(1) |\widehat{t^+}(\chi)|^2\,
\label{equation condition theorem not biased}
\end{equation}
is satisfied, then
\begin{equation}
\delta(L/K;t) - \frac 12 \ll \epsilon + \frac{\norm{t^+}_1(\#\Irr(G^+))^{\frac 16}}{\norm{t^+}_2} \,.
\label{equation theorem cc not biased}
\end{equation} 
 Assuming further that 
$
|\langle \widehat t,\varepsilon_2\rangle_{\Irr(G)}+2\ord_{s=\frac 12} L(s,L/K,t)|\geq \epsilon^{-\frac 12}, 
$
then the second error term on the right hand side of \eqref{equation theorem cc not biased} can be deleted.
\end{theorem}

\begin{remark}
The reason why the factor $\log({\rm rd}_L+2)$ appearing in Theorem \ref{theorem general criterion conjugacy classes biased} does not appear in Theorem \ref{theorem general criterion conjugacy classes not biased} is because of our lower bound for the Artin conductor in Lemma \ref{lemma first bound on Artin conductor}. If the trivial bound $|\chi(g)|\leq \chi(1)$ can be improved to a bound of the form $|\chi(g)|\leq (1-\eta)\chi(1)$ for some fixed $\eta>0$, for many characters $\chi$ of $G$ and for every $g\neq 1$, then we can deduce a sharper lower bound for the Artin conductor of these characters (see Lemma~\ref{lemma finer bounds on Artin conductor}). Such is the case for $G=S_n$ thanks to Roichman's bound (see~\eqref{equation bound Roichman} and Proposition~\ref{lemma lower bound variance S_n big conj classes}), and this allows for a more precise evaluation of $\delta(L/K;t)$ (see Theorem \ref{theorem S_n races}).
\end{remark}

\begin{example} Take $K=\Q$ and $L/\Q$ of even degree (so that there is at least one nontrivial real character) and $t=1-r$, so that $\widehat{t}(\chi)=1_{\chi=1}-\epsilon_2(\chi)$. Assuming BM, we have the upper bound
$$
|\langle \widehat {1-r} ,  \epsilon_2 \rangle_{\Irr(G)}
+2\ord_{s=\frac 12} L(s,L/\Q,1-r)| \leq (M_0+1)\sum_{ \substack{1\neq\chi \in \Irr(G) \\ \chi \text{ real}  }} 
1,
$$
and hence~\eqref{equation condition theorem not biased} holds whenever
\begin{equation}
\sum_{ \substack{1\neq \chi \in \Irr(G)\\ \chi \text{ real} }} 
1 < (M_0+1)^{-1}\epsilon \Big(\sum_{ \substack{1\neq \chi \in \Irr(G)\\ \chi \text{ real}}} \chi(1)\Big)^{\frac 12}.
\label{equation small bias example}
\end{equation}
If this is the case, then thanks to~\eqref{equation delta Ci close to 1/2} we conclude under AC, GRH$^{-}$ and LI that (see the proof of Theorem~\ref{theorem asymptotic formula for moderately biased races} in which $\E[X(L/\Q;1-r)]\in \mathbb Z$)
$$\delta(L/\Q;1-r) -\frac 12 \ll \varepsilon. $$

Moreover, we have the lower bound
\begin{multline}
\langle \widehat {1-r} ,  \epsilon_2 \rangle_{\Irr(G)}+ 2\sum_{\substack{ \chi \in \Irr(G)  }} 
\widehat{1-r}(\chi)\ord_{s=\frac 12}L(s,L/\Q,\chi) \\ \geq 
\#\{1\neq  \chi\in \Irr(G): \chi \text{ real}\}-M_0\#\{ \chi\in \Irr(G): \epsilon_2(\chi)=-1\},
\label{equation bound mean res}
\end{multline}
and hence, if\footnote{This mild condition is satisfied by most of the extensions mentioned in this paper. However, we will see in Remark~\ref{remark:Sl23} that it is essential.} $2M_0 \#\{ \chi\in \Irr(G): \epsilon_2(\chi)=-1\} \leq \#\{ 1\neq \chi\in \Irr(G): \chi \r\},$ then the condition~\eqref{equation:conditionHighBiasC1C2} holds whenever
\begin{equation}
\sum_{\substack{ 1\neq  \chi \in \Irr(G)\\ \chi \r  }} 
1> \epsilon^{-\frac 12}\Big( \log ({\rm rd}_L+2) \sum_{\substack{1\neq \chi \in \Irr(G)\\ \chi \r}} \chi(1) \Big)^{\frac 12}.
\label{equation big bias example}
\end{equation}
\label{example real characters}
\end{example}
We expect the condition~\eqref{equation small bias example} to hold for many extensions, and hence under AC, GRH$^-$ and LI, $\delta(L/\Q;\res)$ is often close to $\tfrac 12$. Precisely, this holds if $G=G^+$ has a real irreducible representation of degree $d$ and admits $o(\sqrt d)$ irreducible real representations. In the generic case $G=S_n$, there exists exactly $p(n)\sim \ee^{\pi\sqrt{ \frac{2n}3}}/(4n\sqrt 3)$ (the number of partitions of $n$) 
real irreducible representations, one of which has degree $n!^{\frac 12-o(1)}$ (see Theorem \ref{theorem S_n races}).

\begin{example}
Take $L=\mathbb Q(\zeta_q)$ with $q\geq 3$ odd and squarefree in Example~\ref{example real characters}. Then, the inequality~\eqref{equation big bias example} holds provided 
\begin{equation}
2^{\omega(q)} \gg \epsilon^{-1} \log q,
\label{equation condition Dirichlet}
\end{equation} 
which already appeared in \cite{Fi1}. For general finite abelian groups, the number of real characters is equal to the number of elements of order at most two,
hence the inequality~\eqref{equation big bias example} translates to
$$ |\{ 1\neq g \in G : g^2 = 1\}| \gg \epsilon^{-1} \log ({\rm rd}_L+2).$$
As a consequence, for $\delta(L/\Q;\res)$ to be close to $1$, it is sufficient that $\Gal(L/\Q)$ contains a substantial $2$-torsion subgroup and that $d_L$ is of controlled size. A good example of such an extension is $\Q(\sqrt {p_1},\sqrt {p_2},...,\sqrt{p_k})/\Q$ (where the $p_i$'s are pairwise distinct primes). In this case~\eqref{equation big bias example} holds provided
$$ 2^k \gg \epsilon^{-1} \sum_{i=1}^k \log p_i.$$
Interestingly, if we put $q:=\prod_{i=1}^kp_i$, then this is exactly~\eqref{equation condition Dirichlet}. We will see that the inequality~\eqref{equation condition Dirichlet} plays an explicit role in the statement of Theorem~\ref{th:HighlyBiasedAbelian} and Theorem~\ref{th:RelHilb} (see also Remark~\ref{rem:Tenenbaum} that discusses the density of integers $q$ such that~\eqref{equation condition Dirichlet} holds).
\end{example} 

We now derive group theoretic criteria that ensure that~\eqref{equation theorem cc not biased} holds.

\begin{corollary}
\label{corollary general criterion conjugacy classes not biased}
Let $L/K$ be an extension of number fields that are both Galois over $\Q$ for which AC, GRH$^-$, and LI hold. 
Fix a class function $t:G\rightarrow \R$ such that $\widehat{t^+} \not \equiv 0$,
and fix $\epsilon>0$ small enough. Then \eqref{equation theorem cc not biased} holds, provided either of the following conditions\footnote{
The notion of symplectic character is defined in Theorem \ref{theorem Frobenius Schur}.} holds:
\begin{enumerate}
\item  $\norm{t^+}_1^{\frac 12} (\norm{t}_2+\norm{t^+}_2)(\#\Irr(G^+))^{\frac 14}\cdot (\#\{ \chi \in \Irr(G)\cup \Irr(G^+) \colon \chi \text{ real} \})^{\frac 12}< \varepsilon \norm{t^+}_2^{\frac 32},$

\item 
$ |\langle t,r \rangle_G|+\sum_{\substack{\chi \in \Irr(G) \\ 
\chi \text{ symplectic}
}}|\widehat{t^+}(\chi)| <\varepsilon \norm{t^+}^{\frac 32}_2\norm{t^+}^{-\frac 12}_1({\rm supp}(\widehat t^+))^{-\frac 14} \,. $
\end{enumerate}
\end{corollary}

So far we have shown that the limiting values $1$ or $\frac 12$ 
are expected for the density $\delta(L/K;t)$ in many natural examples. Taking $t=\res$ and $K=\Q$, one could ask whether $\delta(L/\Q;\res)$ can plainly equal those limiting values. The following general result gives an effective negative answer to this question.

\begin{theorem}\label{th:DistTo1orHalf}
Let $L/\Q$ be a Galois extension for which AC, GRH$^-$, and LI hold, and let $d_L$ be the absolute discriminant of $L$. 
\begin{enumerate}
\item We have the bound
$$
\delta(L/\Q;\res)\leq 1-c_1\exp(-c_2\#\{\chi \in \Irr(G) : \chi \r\})
$$
with positive absolute constants $c_1$, $c_2$. 

\item Assuming moreover GRH, recalling that $M_0>0$ is a fixed absolute constant, and assuming that there is 
a constant $\kappa\in (0,1)$ satisfying:
\begin{itemize}
\item $\#\{\chi\in\Irr(G)\colon \chi\text{ real}\}>2\kappa^{-1}$,
\item $\#\{\chi \in \Irr(G)\colon \chi \text{ symplectic}\} \leq 
\frac{1-\kappa}{2M_0} \#\{\chi\in\Irr(G)\colon \chi\text{ real}\}$,
\end{itemize}
then for $\max(d_L,\sum_{\substack{\chi \in \Irr(G) \\ \chi \r}} \chi(1))$ large enough we have: 
$$
\delta(L/\Q;\res)-\frac 12 \geq c (\log ({\rm rd}_L+2))^{-\frac 12}\Big( \sum_{ \substack{ \chi \in \Irr(G) \\  \chi \r}}\chi(1)^2\Big)^{-\frac 14}\,,
$$
where $c>0$ is absolute.
\end{enumerate}
\end{theorem}

These bounds are essentially optimal. The first one is sharp (up to a log factor in the exponent) in the following cases:
\begin{itemize}
\item the dihedral extensions considered in Theorem~\ref{th:dihedralext},
\item the extension $K_d/\Q$ where $K_d$ is the Hilbert class field of a quadratic field
$\Q(\sqrt{d})$ (see Theorems~\ref{theorem vitrine} and~\ref{th:HilbertOverQ}),
\item the abelian extension $\Q(\sqrt{p_1},\ldots,\sqrt{p_m})/\Q$ (see Theorem~\ref{th:HighlyBiasedAbelian}).
\end{itemize}
As for (2) of Theorem~\ref{th:DistTo1orHalf}, it is sharp in the case of $p$-cyclotomic extensions (where $p$ is a prime number) as shown in~\cite[(3.20)]{FiMa}.
There are also cases where the value of $\delta(L/\Q;\res)$ differs significantly from this bound, notably: 
\begin{itemize}
\item the 
case of the radical extensions considered in Theorem~\ref{th:BiasRadical},
\item the case of $S_n$-extensions (see Theorem~\ref{theorem S_n races}).
\end{itemize}

 The representation theoretic assumptions in Theorem~\ref{th:DistTo1orHalf}(2) are essential since in the case where $G$ is a generalized quaternion group (see~\cite{Ba}) or $G={\rm SL}_2(\F_3)$ one can have $\delta(L/\Q;\res)=\frac 12$ (see Remark~\ref{remark:Sl23}).

\begin{remark}
Note that even in the case where $G$ admits no symplectic character, it would still be possible to have $\delta(L/\Q;t)=\frac 12$. However,
 if one moreover assumes that 
$\langle t,r \rangle_G \neq 0$ (recall \eqref{equation definition r(C)}), then a lower bound on  $|\delta(L/\Q;t)-\tfrac 12|$ could be deduced from an estimate on 
$$
\log ({\rm rd}_L+2)\sum_{\chi\in\Irr(G)}\chi(1)|\widehat{t}(\chi)|^2\,.
$$
This could be done by following the lines of the proof of Theorem~\ref{th:DistTo1orHalf}.
\end{remark}

In the following sections we focus on the cases where the class function $t$ considered is either $t=\res$ (see~\eqref{equation number of square roots})  or $t=t_{C_1,C_2}=|G||C_1|^{-1}1_{C_1} - |G||C_2|^{-1}1_{C_2}$ for distinct conjugacy classes $C_1, C_2$ of $G$.

\subsection{Generic case: $S_n$-extensions}

The case where $\Gal(L/\Q)=S_n$ is ``generic'' in the sense that according to many orderings of number fields (see \emph{e.g.}~\cite{Gal,Mal}), $S_n$ is the most common Galois group. In this case our results rely on the rich and beautiful representation theory of the symmetric group that involves the combinatorics of partitions and tabloids. As an application we answer positively and quantitatively a question of Ng \cite[Section 5.3.5]{Ng} about whether for any conjugacy class $C\neq \{{\rm id} \} $ we have $ r(\{ {\rm id}\})>r(C)$, and as a result $\delta(L/\Q;t_{C, {\rm id}} )> \frac 12$ (see \eqref{equation theorem Sn races best possible} below). The exact bound we obtain in~\eqref{equation lower bound mean Sn} is 
$$  r(\{{ \rm id}\})-r(C) \geq n!^{\frac 12};$$
one can deduce sharper bounds for specific conjugacy classes using bounds on the characters of $S_n$. Such bounds have been established in the important papers of Roichman~\cite{Ro}, Larsen--Shalev~\cite{LaSh} and F\'eray--\'Sniady~\cite{FeSn}. In our context we are able to apply Roichman's bound
to obtain estimates for $\delta(L/\Q;t_{C_1,C_2})$ that take into account the ramification data. This is specific to $S_n$ since the factor $(\log({\rm rd}_L))^{-\frac 12}$ appearing in~\eqref{eq:UnbiasedCCSn} is not present in Theorem \ref{theorem general criterion conjugacy classes not biased}. This leads to an estimate for Chebyshev's bias that is superior to that following from Theorem~\ref{theorem general criterion conjugacy classes not biased}. The resulting bound shows that the Chebyshev bias dissolves both in the horizontal (\emph{i.e.} as the size of the root discriminant increases) and the vertical (\emph{i.e.} as the size of the Galois group increases) limits.

\begin{theorem}
\label{theorem S_n races}
Let $L/K$ be an extension of number fields for which  $L$ is Galois over $\Q$. Assume that $G^+=\Gal(L/\Q)=S_n$ with $n\geq 2$, and that AC, GRH and LI hold. Fix $\epsilon>0$ and let $C_1,C_2$ be distinct elements
 of  $G^\sharp\cup \{0\}$ for which $C_1^+\neq C_2^+$ and $\min(|C_1^+|,|C_2^+|) \leq n!^{1-\frac {4+\varepsilon}{\ee \log n}}$. Then, the functions $ E(y;L/K,t_{C_1,C_2}), E(y;L/\Q,\res)$ admit limiting distributions whose respective means are\footnote{Note that the number of partitions $p(n)$ of $n$ satisfies the Hardy-Ramanujan asymptotic $p(n)\sim \ee^{\pi\sqrt{ \frac{2n}3}}/(4n\sqrt 3)$. Moreover, if $C_1$ and $C_2$ are both composed of only odd cycles, then the mean of $E(y;L/K,t_{C_1,C_2})$ vanishes. Finally, see \cite[Section 5.3.5]{Ng} for a combinatorial formula for this mean in some cases.} 
$$\ll \Big( \frac{n! p(n)}{\min(|C_1^+|,|C_2^+|)} \Big) ^{\frac 12}  ; \hspace{1cm} =p(n)-1,$$ 
 and whose respective variances are
 $$   \ll \frac{n!^{\frac 32}\log ({\rm rd}_L) }{ \min(|C_1^+|,|C_2^+|) } ;\hspace{1cm} \asymp
  \log ({\rm rd}_L) (n/\ee)^{n/2}\ee^{\sqrt{n}}\,.
$$
 Moreover, the variance of the limiting distribution of $E(y;L/K,t_{C_1,C_2})$ is
 $$ \gg \Big(1-\frac{\log \min(|C_1^+|,|C_2^+|)}{\log n!} \Big)  \frac{n!^{\frac 32}\log ({\rm rd}_L) }{ \min(|C_1^+|,|C_2^+|)^{\frac 32} p(n)^{\frac 12}},$$
and as a consequence we have the upper bound
\begin{equation} \label{eq:UnbiasedCCSn}
\delta(L/K;t_{C_1,C_2})-\frac 12 \ll \Big(1-\frac{\log \min(|C_1^+|,|C_2^+|)}{\log n!}\Big)^{-\frac 12}\cdot \frac{n!^{-\frac 1{4}} p(n)^{\frac 34} \min(|C_1^+|,|C_2^+|)^{\frac 1{4}}}{   (\log ({\rm rd}_L) )^{\frac 12}}\,. 
\end{equation}
This estimate is essentially best possible in the sense that specializing to $K=\Q$ and $C_2=\{{\rm id}\}$, we have, for any conjugacy class $C_1$,
the lower bound
\begin{equation}
 \delta(L/\Q;t_{C_1,\{{\rm id}\}})-\frac 12 \geq c \frac{n!^{-\frac 1{4}} }{   (\log ({\rm rd}_L) )^{\frac 12}},  
 \label{equation theorem Sn races best possible}
\end{equation}
where $c>0$ is absolute.
Finally,
 \begin{equation}\label{eq:UnbiasedNRRSn}
\delta(L/\Q;\res) - \frac 12 \asymp \frac{  n!^{-\frac 14} p(n) \ee^{-\frac{\sqrt{n}} 2}n^{\frac 18}}{(\log ({\rm rd}_L))^{\frac 12}}.
 \end{equation}

\end{theorem}

As a consequence we can quantify the idea that a ``random'' Galois extension of the rationals rarely produces a high Chebyshev bias. This is the purpose of the following statement.

\begin{corollary}\label{cor:RandomUnbiased}
For a polynomial $f\in\Z[T]$, let $K_f\subset \C$ denote its splitting field over $\Q$. For fixed integers $n,N\geq 2$ set:
$$
E_n(N)=\{f\in \Z[T]\colon f\text{ monic of degree $n$ with all its coefficients in }[-N,N]\}\,.
$$   
The proportion $\eta_{n,N}$ of polynomials $f\in E_n(N)$ such that
\begin{equation}\label{eq:gallagher}
\delta(K_f/\Q;\res) - \frac 12  \asymp \frac{  n!^{-\frac 14} p(n) \ee^{-\frac{\sqrt{n}} 2}n^{\frac 18}}{(\log ({\rm rd}_{K_f}))^{\frac 12}}
\end{equation}
satisfies, under AC, GRH and LI for every $K_f/\Q$,
$$
\eta_{n,N}\geq 1-O\Big(n^3\frac{\log N}{\sqrt{N}}\Big)\,.
$$
\end{corollary}

For the lower bound of the corollary to make sense, one should first pick a large value of $n$ so that~\eqref{eq:gallagher} 
implies that the density $\delta(L/\Q;\res)$ is close to $\frac 12$.
Then one selects a large value of $N$ (explicitly, $N$ of size $n^{6+\varepsilon}$ suffices) so that the upper bound of the corollary is small, that is the proportion of admissible polynomials is close to $1$.

The proof of the corollary follows easily from combining Theorem~\ref{theorem S_n races} with Gallagher's Theorem (see \emph{e.g.}~\cite[Th. 4.2]{Kow}) that quantifies the fact that generically the splitting field over $\Q$ of a random monic integral polynomial of degree $n$ has Galois group isomorphic to $S_n$. Note that Gallagher's bound has been improved (see \emph{e.g.}~\cite{Di}) and therefore the lower bound in Corollary~\ref{cor:RandomUnbiased} is not best possible (one conjectures that $1-\eta_{n,N}\asymp_n N^{-1}$); we will still apply Gallagher's bound because of its uniformity with respect to $n$.

\subsection{Explicit families} \label{subsec:semi}
In this section we discuss our results for some families of supersolvable extensions of number fields (we recall that AC is known for such extensions).
 \subsubsection{Dihedral extensions}

 Recall that for $n\geq 1$ the dihedral group $D_n$ is defined by
 \begin{equation}\label{eq:presentDihedral}
 D_n:=\langle \sigma,\tau\colon \sigma^n=\tau^2=1,\, \tau\sigma\tau=\sigma^{-1}\rangle\,.
 \end{equation}
Dihedral groups have a substantial proportion of elements of order $2$ and this translates into the existence of many real irreducible characters (set $h=1$ in~\eqref{equation elements of order m} and note that $D_n$ only has irreducible representations of degree
 bounded by $2$ and admits no symplectic character). As a consequence, dihedral Galois extensions are natural candidates for extensions that
may exhibit extreme biases in the distribution of Frobenius elements. The following result confirms this intuition by using a construction due to Kl\"uners~\cite{kluners06}. 

\begin{theorem}\label{th:dihedralext}
There exists a sequence $(K_\ell/\Q)_{\ell\geq 7}$ of dihedral extensions indexed by prime numbers $\ell\geq 7$ such that $\Gal(K_\ell/\Q)\simeq D_\ell$ and such that, conditionally on GRH and BM for $K_\ell/\Q$ and for the choice 
$(C_1,C_2)=(\{\tau\sigma^k\colon 0\leq k\leq \ell-1\},\{{\rm id}\})$, the functions $ E(y;K_\ell/\Q,t_{C_1,\{{\rm id}\}}), E(y;K_\ell/\Q,\res)$ admit limiting distributions whose means are both  $\gg \ell$, and whose variances are both $\ll \ell \log \ell$. As a result, the fluctuations of these functions are dominated by a constant term, and one has that
$$
\min\big(\underline{\delta}(K_\ell/\Q;t_{C_1,C_2}),
\underline{\delta}(K_\ell/\Q;\res)\big)\geq 1-O\Big(\frac{\log\ell}\ell\Big)\,.
$$ 
If one additionally assumes LI for
 $K_\ell/\Q$ then both $\delta(K_\ell/\Q;\res)$ and $\delta(K_\ell/\Q;t_{C_1,C_2})$ exist and one has the refined bounds
$$
\exp(-c_1\ell)\leq 1-\delta(K_\ell/\Q;\res)\leq \exp\Big(-c_2\frac{\ell}{\log\ell}\Big);
\qquad \delta(K_\ell/\Q;t_{C_1,C_2})\geq 1-\exp\Big(-c_3\frac{\ell}{\log\ell}\Big),
$$
where the constants $c_1, c_2, c_3>0$ are absolute.
\end{theorem}

 \subsubsection{Hilbert Class Fields of quadratic extensions: the absolute case}
 \label{subsec:HilbertAbs}
 
 From the group theoretic point of view, this section is a slight generalization of the previous one. 
We consider the extensions $K_d/\Q$, where $K_d$ is the Hilbert class field of the quadratic extension $\Q(\sqrt{d})/\Q$.
 The \emph{relative} Galois extension $K_d/\Q(\sqrt{d})$ is abelian and will be considered in \S\ref{subsec:RelHilb}. As in the case of dihedral extensions there are many elements of order $2$ in $\Gal(K_d/\Q)$; this results in estimates similar to those stated in
 Theorem~\ref{th:dihedralext}. 

\begin{theorem}\label{th:HilbertOverQ}
Let $d\neq 1$ be a fundamental discriminant and let $K_d$ be the Hilbert 
Class Field of $\Q(\sqrt{d})$. Then $K_d/\Q$ is Galois; fix a representative $\tau_0$ 
of the nontrivial left coset of $\Gal(K_d/\Q)$ modulo $\Gal(K_d/\Q(\sqrt{d}))$ and assume GRH and BM for $K_d/\Q$. Fix an element $\sigma\in \Gal(K_d/\Q(\sqrt{d}))$ and let $C_1,C_2$ be the conjugacy classes of $\tau_0\sigma$ and $1$, respectively. Then the functions $ E(y;K_d/\Q,t_{C_1,\{ {\rm id}\}}), E(y;K_d/\Q,\res)$ admit limiting distributions whose means are both  $\gg h(d)$, and whose variances are both $\ll h(d) \log |d|$. As a result, the following holds.

\begin{enumerate}
\item For every fundamental discriminant $d\leq -4$ we have the bound
$$
\min\big(\underline{\delta}(K_d/\Q;t_{C_1,C_2}),
\underline{\delta}(K_d/\Q;\res)\big)\geq 
1-O\Big(\frac{\log |d| \log\log |d|}{\sqrt{|d|}}\Big).
$$

\item There exists an unbounded family of fundamental discriminants $d\geq 5$ such that
$$
\min\big(\underline{\delta}(K_d/\Q;t_{C_1,C_2}),
\underline{\delta}(K_d/\Q;\res)\big)\geq 
1-O\left(\frac{(\log |d|)^{2}}{\sqrt{|d|}\log\log |d|}\right).
$$

\item If one additionally assumes LI for each extension
 $K_d/\Q$ then both the densities $\delta(K_d/\Q;\res)$ and $\delta(K_d/\Q;t_{C_1,C_2})$ exist and one has the refined bounds:
$$
\exp\left(-c_1\frac{\sqrt{|d|}}{ \log\log |d|}\right) \leq 1-\delta(K_d/\Q;\res)\leq 
\exp\left(-c_2\frac{\sqrt{|d|}}{\log |d| \log\log |d|}\right)\, \hspace{.3cm}(d<0);
$$
$$
\exp\left(-c_3\frac{\sqrt{|d|}\log\log |d|}{(\log |d|)^{\frac{1}{2}+\frac{{\rm sgn}(d)}{2}}}\right) \leq 1-\delta(K_d/\Q;\res)\leq 
\exp\left(-c_4\frac{\sqrt{|d|}\log\log |d|}{(\log |d|)^{\frac{3}{2}+\frac{{\rm sgn}(d)}{2}}}\right)\, \hspace{.2cm}(d \text{ as in (2)}).
$$
Here, the constants $c_i>0$ are absolute. Both upper bounds also hold for 
$1-\delta(K_d/\Q;t_{C_1,C_2})$.
\end{enumerate}

\end{theorem}

\subsubsection{Radical extensions} 

In contrast with the two previous families, 
we now consider extensions of number fields exhibiting this time a moderate Chebyshev bias.
We will consider splitting fields $K_{a,p}/\Q$ of 
polynomials $f(X)=X^p-a$, where $p$ and $a$ are distinct odd prime numbers. To simplify the analysis we make the extra assumption that $a^{p-1}\not\equiv 1 (\bmod\ p^2)$, in other words ``$p$ is not a Wieferich prime to base $a$'' (see \emph{e.g.}~\cite[\S1--3]{Katz} for a nice account on the theory of such prime numbers).
Let $G=\Gal(K_{a,p}/\Q)$. 
We have the following group isomorphism:
\begin{equation}\label{eq:GaloisSemiDirect}
G\simeq\left\{ \left(\begin{array}{cc}
c & d \\ 
0 & 1
\end{array}\right)  : c\in \F_p^*, d\in \F_p  \right\}\,.
\end{equation}
In particular $G$ is supersolvable (consider the cyclic maximal unipotent subgroup $H$ of $G$) so that Artin's conjecture holds for $K_{a,p}/\Q$. In this case, we apply the work of~\cite{Vi} and explicitly compute the filtration of inertia and in particular obtain an exact formula for the Artin conductor of each irreducible character of $G$.

\begin{theorem}\label{th:BiasRadical}

Let $a,p$ be primes such that $a^{p-1}\not\equiv 1 (\bmod\ p^2)$, and assume GRH
and LI for the extension $K_{a,p}/\mathbb Q$. 
Let $C_1$, $C_2$ be distinct conjugacy classes of $G$. Then the functions $ E(y;K_{a,p}/\Q,t_{C_1,C_2})$, $E(y;K_{a,p}/\Q,\res)$ admit limiting distributions and if $C_1,C_2\neq \{{\rm id}\}$, then the means are both $\ll 1$, and the variances are both $\ll p\log(ap)$ and $\gg p\log p$. If one of $C_1$ or $C_2$ (say $C_2$) is the trivial conjugacy class, then the mean of $ E(y;K_{a,p}/\Q,t_{C_1,\{{\rm id}\}})$ is $\asymp p$ and the variance $\asymp p^3 \log(ap)$. As a result, we have the following estimates.
\begin{enumerate}
\item  For the class function $t=1-r$,
\begin{equation}\label{eq:BiasRNRRadical}
\delta(K_{a,p}/\Q;\res) - \frac 12 \asymp \frac 1{\sqrt{p\log (a p)}}\,.  
\end{equation}
\item
For $c\in\F_p^\times\setminus\{1\}$ let $c^+$ be the conjugacy class of $G$ (recall~\eqref{eq:GaloisSemiDirect}) consisting of all matrices with first row $(c,d)$ where $d$ runs over $\F_p$. If $C_1$ and $C_2$ are not both of type $c^+$, then\footnote{See \eqref{equation expectancies of radical} for an exact determination of the sign of $\delta(K_{a,p}/\Q;t_{C_1,C_2})-\frac 12$ (which coincides with that of $\E[X(K_{a,p}/\Q;t_{C_1,C_2})]$ in the notation of Proposition~\ref{proposition link with random variables} and Lemma~\ref{lemma random var under LI}).}
$$
 \left| \delta(K_{a,p}/\Q;t_{C_1,C_2})-\frac 12 \right| \asymp \frac 1{\sqrt{p\log(ap)}}\,.
$$

\item If $C_1=x^+$ and $C_2=y^+$ for distinct $x,y\in\F_p^\times\setminus\{1\}$, then $\delta(K_{a,p}/\Q;t_{x^+,y^+})= \delta(p;x,y)$, which denotes the density of the  classical Chebyshev bias\footnote{see \cite[Theorem 1.1]{FiMa} for a precise estimation of this bias.} for the couple of residue classes $(x,y)$ modulo $p$. 
\end{enumerate}
Finally, in the relative case $K=\Q(\zeta_p)$ (where $\Gal(K_{a,p}/K)\simeq \Z/p\Z$ is the maximal unipotent subgroup of $G$), for any distinct $d_1,d_2\in\Z/p\Z$, the function
$E(y;K_{a,p}/K,t_{\{d_1\},\{d_2\}})$ admits a limiting distribution. 
The mean is always $0$, and the variance is $\asymp p^3 \log(ap)$ in the case $d_1d_2=0$, and $0$ otherwise. If $d_1d_2=0$, then we have $\delta(K_{a,p}/K;t_{\{d_1\},\{d_2\}})=\tfrac 12$. (If $d_1d_2\neq 0$, then we have no result on $\delta(K_{a,p}/K;t_{\{d_1\},\{d_2\}})$.) 
\end{theorem}

By estimating the density of the couples of primes $(a,p)$ such that $p$ is not Wieferich to base $a$, we deduce the following statement.

\begin{corollary}
\label{corollary admissible couples in radical extensions}
Assume GRH and LI for every $K_{a,p}$ with $a,p$ running over all primes. The proportion of couples of primes $(a,p)$ with $a\leq A$ and $p\leq P$ such that~\eqref{eq:BiasRNRRadical} holds is 
$$
1- O\Big(\frac{\log P(\log A\log\log A)^{\frac 12}}{P} + \frac{P\log A}{A}\Big)
$$
in the range $A,P\geq 3$, $ P\log P \leq  A \leq  \ee^{P^2/(\log P)^3}$.

\end{corollary}

We proceed by considering abelian extensions of number fields. 

\subsubsection{Iterated quadratic extensions}

We first describe the case of a Galois group with a ``big'' $2$-torsion subgroup. We see that if the product of ramified primes belongs to a certain subset of $\mathbb{N}$ of density $0$ (see Remark~\ref{rem:Tenenbaum}) then we obtain an extreme Chebyshev bias.

\begin{theorem} \label{th:HighlyBiasedAbelian}
Let $L=\Q(\sqrt{p_1},...,\sqrt{p_m})$, where $p_1<p_2<...<p_m$ are distinct odd primes. Let $G=\Gal(L/\Q)\simeq \{\pm 1\}^m$ and $q:=\prod_i p_i$. Assume\footnote{In the case $2^{-\omega(q)}\log q \leq \epsilon$, assumption LI can be replaced with BM at the cost of a weaker lower bound on $1-\delta(L/\Q;t_{a,{\bf 1}})$.} GRH and LI and let $\epsilon>0$ be small enough. Then for any $a,b\in G$, the functions $ E(y;L/\Q,t_{a,b}) $ and $E(y;L/\Q,\res)$ admit limiting distributions of respective means $|G|(\delta_{b={\bf 1}}-\delta_{a={\bf 1}})$ (where ${\bf 1}=(1,\ldots,1)$) and $|G|$, and both of variance $\asymp |G|\log q$. 
As a result, for any $a\in G\setminus\{ {\bf 1}\}$, we have that
$$
\delta(L/\Q;t_{a,{\bf 1}})= \begin{cases} 1-O(\exp(-c 2^{\omega(q)}/\log q )) &\text{ if } 
                                             2^{-\omega(q)}\log q \leq \epsilon  \\ 
                                  \frac 12+O\Big(\sqrt{2^{\omega(q)}/\log q}\Big) &\text{ if } 
                                             2^{-\omega(q)}\log q \geq \epsilon^{-1} 
                       \end{cases} 
$$
where $c$ is some positive absolute constant. The same estimate holds for $\delta(L/\Q;\res)$.
\end{theorem}

\begin{remark}
The reason we chose $b={\bf 1}$ in the second part of the statement is because one can show that $\delta(L/\Q;t_{a,b})=\frac 12$ as soon as $a\neq {\bf 1}$ and $b\neq {\bf 1}$.
\end{remark}

 \begin{remark}\label{rem:Tenenbaum}
 It follows from Theorem~\ref{th:HighlyBiasedAbelian} and~\cite[Chap. 2, Th. 6.4]{Ten} (summing over integers exceeding $\log\log x/\log 2$ in the equality stated in \emph{loc. cit.}) that there exists a subset 
 $S\subseteq \mathbb{N}$ such that for any nontrivial $a\in\Gal(\Q(\sqrt{p_1},\ldots,\sqrt{p_m})/\Q)$
 \begin{align*}
 & \#S\cap [1,X]=\frac{X}{(\log X)^{1-\frac{1+\log\log 2}{\log 2}+o(1)}}\,,\\
 &\delta(K_d/\Q(\sqrt{d});a,{\bf 1}) =1-o_{q\rightarrow\infty}(1)\qquad (q\in S)\,.
 \end{align*}
 \end{remark}

\subsubsection{Hilbert class fields of quadratic extensions: the relative case}
\label{subsec:RelHilb}
 
 In this section the setting is as in \S\ref{subsec:HilbertAbs}: $d$ is a fundamental discriminant satisfying $|d|>1$ and $K_d$ denotes the Hilbert class field of the quadratic field $\Q(\sqrt{d})$, therefore $K_d/\Q(\sqrt{d})$ is Galois with group $G\simeq {\rm Cl}_d$. 
In~\cite[\S 6.2]{Ng}, Ng studies discrepancies in the distribution of prime ideals according to their class in ${\rm Cl}_d$ being either trivial or 
 any fixed nontrivial class. 
In the next result, we consider two possible choices for $(C_1,C_2)$; in the case $(C_1,C_2)=(\{\overline{\mathfrak a}\},\{\overline{1}\})$ we recover precisely Ng's Theorem~\cite[Th. 6.2.1]{Ng}.

 \begin{theorem}
 \label{th:RelHilb}
 Let $d$ be a fundamental discriminant and assume that $h(d)>1$. 
Let $K_d$ be the Hilbert class field  of $\Q(\sqrt{d})$ and assume that GRH holds for the
 extension $K_d/\Q$. We identify $\Gal(K_d/\Q(\sqrt{d}))$ with the class group ${\rm Cl}_d$ and we let
  $\overline{\mathfrak a}$ be a nontrivial ideal class. Choosing 
  $(C_1,C_2)$ to be either $(\{\overline{\mathfrak a}\},\{\overline{1}\})$ or 
  $(\{\overline{1}\},0)$, the Frobenius counting function $E(y;K_d/\Q(\sqrt{d}),t_{C_1,C_2})$
   defined in~\eqref{equation definition E} admits a limiting distribution of mean and variance respectively denoted
   $\mu_{K_d/\Q(\sqrt d)}(C_1,C_2)$, and $\sigma^2_{K_d/\Q(\sqrt d)}(C_1,C_2)$, satisfying
\begin{align*}
|\mu_{K_d/\Q(\sqrt d)}(C_1,C_2)|&\leq 2^{\omega(d)}+4\sum_{1\neq\chi\in\Irr({\rm Cl}_d)}{\rm ord}_{s=\frac 12}L(s,K_d/\Q(\sqrt{d}),\chi)\,, 
\\
  \sigma^2_{K_d/\Q(\sqrt d)}(C_1,C_2)& \gg h(d)\,\,\,(\text{only for the choice } (C_1,C_2)=(\{\overline{1}\},0))\,. 
\end{align*}
 Under the extra assumption LI, one has that $|\mu_{K_d/\Q(\sqrt d)}(C_1,C_2)|\leq 2^{\omega(d)}$ and $\sigma^2_{K_d/\Q(\sqrt d)}\gg h(d)$ for both choices of $(C_1,C_2)$. In particular,
 for $d$ running over any family of fundamental discriminants such that $2^{-\omega(d)}\sqrt{h(d)}\rightarrow\infty$ (e.g the set of all negative fundamental discriminants, or the family of positive discriminants in Lemma~\ref{propo:HighBiasHCF}), one has
 $$
 \delta(K_d/\Q(\sqrt{d}),t_{C_1,C_2}\})-\frac 12\ll \frac{2^{\omega(d)}}{\sqrt{h(d)}}\,.
 $$
 \end{theorem}

\begin{remark}
\label{remark:Sl23}
 Number field extensions with Galois group $G={\rm SL}_2(\mathbb F_p)$ have the peculiarity that many representations of $G$ are symplectic, and hence there is potentially a large supply of real zeros. There exists extensions $L/\Q$ for which the existence of real zeros has the dramatic effect that $\delta(L/\Q;\res) =\tfrac 12$ (the details will appear in future work; see also~\cite{Ba} for the impact on the bias of central zeros in the case of generalized quaternion extensions). This is in total contradiction with the usual Chebyshev bias philosophy which says that ``primes congruent to quadratic nonresidues are more abundant than primes congruent to quadratic residues''. For ${\rm SL}_2(\mathbb F_p)$-extensions $L/\Q$, one can show\footnote{By Propositions~\ref{proposition link with random variables} and~\ref{proposition asymptotic for the variance}, and Lemma \ref{lemma first bound on Artin conductor}, one can show using \cite[Table 5, Appendix C]{Kow} that $\E[X(L/\Q;\res)]\ll p$ and $\V[X(L/\Q;\res)]\gg p^2$. Then this should be combined with Theorem \ref{theorem asymptotic formula for moderately biased races}.} that $\delta(L/\Q;\res)$ takes values between $\tfrac 12$ and  $\eta$ for some absolute $\eta<1$, and one expects that $\tfrac 12 \leq \delta(L/\Q;\res) \leq \tfrac 12+c(\log ({\rm rd}_L))^{-\frac 12}$ for some absolute $c>0$. The case $\delta(L/\Q;\res)=\tfrac 12$ is achieved\footnote{This is because there are exactly $\frac{p+1}2+\left( \frac{-1}p \right)$ symplectic and $\frac{p+3}2+\left( \frac{-1}p \right)$ orthogonal representations, hence by Proposition~\ref{proposition link with random variables}, $\E[X(L/\Q,\res)]=0$.} with extensions for which the Artin root number of every symplectic character is $-1$ (see~\cite{Ba} where the case of generalized quaternion extensions of number fields is studied). 

\end{remark}

\section{Distribution of Frobenius elements \emph{via} Artin $L$-functions}\label{section:explicit}

\subsection{Representation theory of finite groups}\label{section:repgroup}

For completeness and because of its crucial importance in our work, let us first recall some basic representation theory of finite groups.
We let $\Irr (G)$ and $G^\sharp$ denote respectively the set of irreducible characters and the set of conjugacy classes of the group $G$. Keeping the notation of Bella\"iche~\cite{Be}, for a class function $t:G \rightarrow \C$ and a character $\chi \in \Irr(G)$, we define the Fourier transform
$$ \widehat t (\chi):=\langle \chi, t\rangle_G=\frac 1{|G|}\sum_{g\in G}   \chi(g) \overline {t(g)} =   \sum_{C \in G^\sharp} \frac{|C|}{|G|} \chi(C) \overline {t(C)} .$$
Note that if $1_D$ is the characteristic function of a given conjugacy-invariant set $D\subset G$, then 
$$\widehat{1_D} (\chi) =\frac 1{|G|}\sum_{\substack{C\in G^\sharp: \\ C \subset D}}\chi(C)|C|.$$
\begin{lemma}[Orthogonality Relations]\label{lem:orthrel}
Let $G$ be a finite group. If $g_1,g_2 \in G$, then
\begin{equation}
 \sum_{\chi \in \Irr(G)} \chi(g_1)\overline{\chi}(g_2) = \begin{cases}
\frac{|G|}{|C|} & \text{ if } g_1 \text{ and } g_2 \text{ are both in the same conj. class }C,\\
0 & \text{ otherwise.}
\end{cases}
\label{orthogonality of congugacy classes}
\end{equation}
Moreover, if $\chi,\psi\in \Irr(G)$, then
\begin{equation}
\frac 1{|G|} \sum_{g\in G} \chi(g)\overline{\psi}(g) = \begin{cases}
1 & \text{ if } \chi = \psi,\\
0 & \text{ otherwise}.
\end{cases}
\label{orthogonality of characters}
\end{equation}
\end{lemma}
As a consequence of~\eqref{orthogonality of characters}, we have the formula 
\begin{equation}
 t = \sum_{\chi \in \Irr(G)}  \overline{\widehat t (\chi)} \chi. 
 \label{equation Fourier transform}
\end{equation}

We will often count elements of order $2$ with characters, using the following Lemma.
\begin{lemma}
\label{lemma elements of order m}
Let $G$ be a finite group and $ k\in \mathbb N$. Then for any $h\in G$ we have the identity
\begin{equation} 
\#\{g\in G : g^k = h\} =
 \sum_{\chi \in \Irr(G)} \overline{\varepsilon_k(\chi)} \chi(h),
\label{equation elements of order m}
\end{equation}
\begin{equation}
 \text{where}\qquad \varepsilon_k(\chi):= \frac 1{|G|} \sum_{g\in G}\chi (g^k).
 \label{equation definition epsilon m}
\end{equation}
\end{lemma}

\begin{proof}
Since $r_k(h):=\#\{g\in G : g^k = h\}$ defines a class function on $G$, we have
$$ \#\{g\in G : g^k = h\} = \sum_{\chi \in \Irr(G)}\overline{ \widehat{r_k}(\chi)}\chi(h)\,. $$
The proof follows by definition of Fourier coefficients.
\end{proof}

The number 
$$\epsilon_2(\chi)=\frac{1}{|G|}\sum_{g\in G} \chi(g^2) $$ 
is called the \emph{Frobenius--Schur indicator} of $\chi$ and is central in our analysis. If $\chi $ is irreducible, then $\epsilon_2(\chi)=\widehat{r}(\chi) \in \{ -1,0,1\}$ (see~\cite[Th. 8.7]{Hup}), and moreover each of these three possible values has a precise meaning in terms of the $\R$-rationality of $\chi$ and of the underlying representation $\rho$. 

\begin{theorem}[Frobenius, Schur]
\label{theorem Frobenius Schur}
Let $G$ be a finite group, and let $\chi\in \Irr(G)$ be the character of an irreducible complex representation $\rho\colon G\rightarrow {\rm GL}(V)$. 

(1) If $\epsilon_2(\chi)=0$, then $\chi\neq \overline{\chi}$, $\chi$ is not the character of an $\R[G]$-module, and there does not exist a $G$-invariant, $\C$-bilinear form $\neq 0$ on $V$. We say that $\rho$ is a unitary representation.

(2) If $\epsilon_2(\chi)=1$, then $\chi=\overline{\chi}$ is the character of some $\R[G]$-module, and there exists a $G$-invariant, $\C$-bilinear form which is symmetric and nonsingular, unique up to factors in $\C$. We say that $\rho$ is an orthogonal representation.

(3) If $\epsilon_2(\chi)=-1$, then $\chi=\overline{\chi}$ is not the character of any 
$\R[G]$-module, and there exists a $G$-invariant, $\C$-bilinear form which is skew-symmetric and nonsingular, unique up to factors in $\C$. We say that $\rho$ is a symplectic (or quaternionic) representation.
\end{theorem}
\begin{proof}
See for instance~\cite[Th. 13.1]{Hup}.
\end{proof}

A direct consequence of Lemma \ref{lemma elements of order m} and Theorem \ref{theorem Frobenius Schur} is the following formula for the class function $r$ introduced in \eqref{equation definition r(C)}:
\begin{equation} 
r=\sum_{\substack{\chi \in \Irr(G)\\ \chi \r}} \epsilon_2(\chi) \chi.
\label{equation number of square roots} 
\end{equation}
\begin{lemma}
Let $G$ be a finite group, and let $t:G \rightarrow \C$ be a class function. We have the identity 
\begin{equation}
\norm {\widehat t} _2^2 :=\sum_{\chi \in \Irr(G)} |\widehat t (\chi)|^2 = \sum_{C\in G^\sharp}\frac{|C|}{|G|} |t(C)|^2 = \norm t^2_2.
\label{eq:lemmasumsquareorthogonality}
\end{equation} 
In particular, if $C_1,...,C_k \in G^\sharp$ are disctinct and $\alpha_1,...,\alpha_k \in \C$, then 
$$ \sum_{\chi \in \Irr(G)} |\alpha_1\chi(C_1)+\dots +\alpha_k \chi(C_k) |^2 = |\alpha_1|^2 \frac{|C_1|}{|G|} +\cdots +|\alpha_k|^2 \frac{|C_k|}{|G|}.$$
\label{lemma:sumsquareorthogonality}
\end{lemma}
\begin{proof}
This is Parseval's identity~\eqref{eq:parseval}.
\end{proof}

We will also need a pointwise bound on the Fourier coefficients $\widehat t(\chi)$. 
 \begin{lemma}
 \label{lemma pointwise bound Fourier transform}
 Let $G$ be a finite group, $t:G \rightarrow \C$ be a class function and $\chi\in \Irr(G)$. Then we have the bound
 $$ |\widehat t(\chi)| \leq \chi(1) \norm{t}_1.  $$
 \end{lemma}
\begin{proof}
This follows directly from the definition of Fourier transform.
\end{proof}
From the definition $t^+:={\rm Ind}_G^{G^+} (t)$, one easily sees that
\begin{equation}
\label{equation definition t+}
t^+= \sum_{\chi \in \Irr(G)} \overline{\widehat t(\chi)} {\rm Ind}_G^{G^+} (\chi) .
\end{equation}
Applying Frobenius reciprocity, we can compute the Fourier transform of $t^+$ in terms of that of $t$.
\begin{lemma}
\label{lemma Fourier transform t+}
Let $G^+$ be a finite group, let $G$ be subgroup, and let $t\colon G\rightarrow \C$ be a class function on $G$. Then,
 for any $\chi \in \Irr(G^+)$, 
we have the formula 
\begin{equation}
\label{equation fourier trasform t+}
\widehat {t^+}(\chi)=\langle\chi|_G;t\rangle_G\,.
\end{equation}
\end{lemma}

\begin{proof}
Frobenius reciprocity gives that
$$\widehat {t^+}(\chi)  =   \sum_{\eta \in \Irr(G)} \widehat t(\eta) \langle \chi, {\rm Ind}_G^{G^+} (\eta) \rangle_{G^+} = \sum_{\eta \in \Irr(G)} \widehat t(\eta) \langle \chi|_{G}, \eta\rangle_{G} = \langle \chi|_G,t\rangle_G\,.
$$
\end{proof}

Under certain assumptions we can also compare the $2$-norm of $t^+$ in terms of that of $t$.
\begin{lemma}
\label{lemma 2norm t+}
Let $G^+$ be a finite group, let $G$ be subgroup of $G^+$ and and let $t\colon G\rightarrow \C$ be a class function on $G$.  We have the following.
\begin{enumerate}
\item  If $G$ is a normal subgroup of $G^+$, then
$$ \norm{t^+}^2_2 \leq \frac{|G^+|}{|G|}\norm{t}^2_2\,. $$
\item If $t$ only takes nonnegative values, then 
$$\norm{t^+}_2\geq \norm{t}_2\,. $$
\end{enumerate}
\end{lemma}

\begin{remark}
The upper bound is attained by the function $t=|G||C|^{-1}1_C$, where $C\in G^\sharp$ is such that $ |C|=|C^+|$ (for example when $G^+$ is abelian). As for the lower bound, it requires a condition since
for example we could take $t=|G||C_1|^{-1}1_{C_1}-|G||C_2|^{-1}1_{C_2}$ for distinct $C_1,C_2 \in G$ for which $C_1^+=C_2^+$, and as a result $ t^+ \equiv 0$. 
\end{remark}

\begin{proof}[Proof of Lemma~\ref{lemma 2norm t+}]
We first expand the norm of $t^+$:
\begin{equation}
\label{equation 2 norm t}
\norm{t^+}^2_2 = \frac 1{|G^+|} \sum_{g\in G^+} \Big| \sum_{\substack{aG \in G^+/G \\ a^{-1}ga\in G}} t(a^{-1}ga)\Big|^2 = \frac 1{|G^+|} \sum_{bG \in G^+/G} \sum_{g\in G} \Big| \sum_{\substack{aG \in G^+/G \\ a^{-1}bga\in G}} t(a^{-1}bga)\Big|^2\,.
\end{equation}
Now we prove (1). Using Cauchy--Schwarz, we deduce from~\eqref{equation 2 norm t} that
$$
\norm{t^+}^2_2 \leq \frac 1{|G^+|} \sum_{bG \in G^+/G} \sum_{g\in G}\Big( \sum_{\substack{aG \in G^+/G \\ a^{-1}bga\in G}}  1 \Big)\Big(\sum_{\substack{aG \in G^+/G \\ a^{-1}bga\in G}} |t(a^{-1}bga)|^2\Big)\,.
$$
We bound the first sum in parentheses trivially, and we exploit the fact that $bg\in aGa^{-1}$ is equivalent to $bg\in G$ for any $a\in G^+$, since $G$ is a normal subgroup. Therefore the condition on the left coset $bG$ in the second sum in parentheses imposes $bG$ to be the trivial left coset $G$. We conclude that
$$
\norm{t^+}_2 \leq \frac 1{|G|}  \sum_{g\in G}\sum_{\substack{aG \in G^+/G }}  |t(a^{-1}ga)|^2 = \frac 1{|G|}|G^+/G|  \sum_{g\in G} |t(g)|^2,
$$
using the fact that conjugation by any $a\in G^+$ induces a bijection of $G$. The claimed upper bound follows. 

For the lower bound (2), we apply positivity in~\eqref{equation 2 norm t} and deduce that

\begin{align*}
 \norm{t^+}^2_2& \geq \frac 1{|G^+|} \sum_{bG \in G^+/G}\sum_{g\in G}  \sum_{\substack{aG \in G^+/G \\ a^{-1}bga\in G}} |t(a^{-1}bga)|^2 \\
 &=\frac 1{|G^+|} \sum_{g\in G}  \sum_{\substack{aG \in G^+/G }} \sum_{\substack{bG = aGa^{-1}} } |t(a^{-1}bga)|^2= \norm{t}^2_2.
\end{align*}
\end{proof}

We finish this section by computing $r^+$ (recall~\eqref{equation number of square roots}) and by deducing a consequence which will be useful in showing that up to ramified primes, $\pi(x;L/K,t)$ is determined by $t^+$.

\begin{lemma}
\label{lemma induction of premiage function}
Let $G^+$ be a finite group, and let $G$ be a normal subgroup. For any $k\in \mathbb N$ and for any class function $t\colon G^+\rightarrow \C$, we define the class function $r_{k,t}\colon G\rightarrow \C$ by setting
$$ r_{k,t}(g) := \sum_{\substack{h\in G  \\ h^k=g }} t(h).$$
 Then we have the following equality of class functions:
$$ r_{k,t}^+\big|_G = [G^+:G] \cdot r_{k,t}. $$  
\end{lemma}
\begin{proof}
First note that if for some $g,h\in G$ and $a\in G^+$ we have that $h^k =a^{-1} g a$, then $(aha^{-1})^k=g$. In other words, since for each fixed value of $a$ we have $aG=Ga$, there is a bijection between the sets $\{ h\in G: h^k=a^{-1} g a\}$ and $\{ h\in G : h^k= g\}$.
Hence, for any $g\in G$, 
$$ r_{k,t}^+\big|_G(g) = \sum_{\substack{aG \in G^+/G \\ a^{-1} g a \in G}}r_{k,t}(a^{-1} g a) = \sum_{h\in G} t(h) \sum_{\substack{aG \in G^+/G \\ a^{-1} g a =h^k}} 1 =\sum_{\substack{aG \in G^+/G }}  \sum_{\substack{h\in G \\ h^k=g}} t(a^{-1}ha) . $$
The claim follows since $t$ is a class function on $G^+$.
\end{proof}
\begin{corollary}
\label{corollary when t+ is 0}
Let $G^+$ be a finite group and $G$ a normal subgroup. If $t\colon G\rightarrow \C$ is a class function such that $t^+\equiv 0$, then for any $k\geq 1$,  one has $\langle t, r_k \rangle_G=0$. Here, $r_k\colon G\rightarrow \C$ is defined by $r_k(g)=|\{ h\in G \colon h^k =g\}|$, \textit{i.e.} $r_k=r_{k,1}$.
\end{corollary}
\begin{proof}
By Frobenius reciprocity and Lemma~\ref{lemma induction of premiage function}, we have that
$$\langle t, r_k \rangle_G\cdot [G^+:G]=\langle t, r_k^+\big|_G \rangle_G = \langle t^+, r_k^+ \rangle_{G^+}=0. $$
\end{proof}

\subsection{Explicit formulas and limiting distributions}

We fix a Galois extension of number fields $L/K$ and let $G=\Gal(L/K)$. For a class function $t\colon G \rightarrow \mathbb C$,
we define the following prime ideal counting function:
$$ \psi(x;L/K,t):=\sum_{\substack{\mathfrak p\triangleleft \mathcal O_K \\ m\geq 1\\ \mathcal N\mathfrak p^m \leq x}} t(\varphi_\mathfrak p^m) \log (\mathcal N\mathfrak p), $$
where
 $\varphi_{\mathfrak p}$ is shorthand for ${\rm Frob}_{\mathfrak p}$, the conjugacy class of a lift (defined up to inertia) of the Frobenius automorphism on the residue field $\mathcal O_L/\mathfrak P$ for some (any) $\mathfrak P\triangleleft \mathcal O_L$ above $\mathfrak p$, and
\begin{equation}
t(\varphi_\mathfrak p^m) := \frac 1{|I_{\mathfrak p}|} \sum_{ i \in I_{\mathfrak p}} t(\varphi_\mathfrak p^m i),
\label{equation definition ramified primes}
\end{equation}
where $I_{\mathfrak p}$ is the inertia group attached to $\mathfrak p$ and any $\mathfrak P\triangleleft\mathcal O_L$ above $\mathfrak p$.
If $D\subset G$ is conjugacy invariant, then we define $\psi(x;L/K,D):=\psi(x;L/K,1_D)$, where $1_{ D}$ is the indicator function of $D$.
We also recall the definition \eqref{equation definition pi x L/K C} of the prime ideal counting function attached to a conjugacy class $C$ of $G$ which we extend in the obvious way to conjugacy invariant sets $D\subset G$.

Our goal is to express $\pi(x;L/K,t)$ in terms of the zeros of \emph{primitive} Artin $L$-functions; this will prevent arithmetic multiplicities from occuring in our formulas. To do so, we will first relate the prime ideal counting functions $\psi(x;L/K,t)$ and $\psi(x;L/\Q,t^+)$ using the induction property for Artin $L$-functions.

\begin{proposition}
\label{proposition psi bridge}
Let $L/K/M$ be a tower of number fields for which $L/M$ is Galois, let $G=\Gal(L/K)$ and $G^+=\Gal(L/M)$. For any class function $t\colon G\rightarrow \C$, we have the identity
\begin{equation}
 \psi(x;L/K,t) = \psi(x;L/M,t^+),
 \label{equation psi bridge}
\end{equation}
As a consequence, if $D\subset G$ is conjugacy invariant, then
\begin{equation}
\psi(x;L/K,D)=\frac{|G^+|}{|G|} \sum_{\substack{C\in G^\sharp: \\ C \subset D}} \frac{|C|}{|C^+|}\psi(x;L/M,C^+),
\label{equation link between C+ G+ with stabs}
\end{equation}
where $C^+$ is defined by~\eqref{equation definition C+}.
\end{proposition}

\begin{remark}
\label{remark C+ is a cc}
Note that if $C\in G^\sharp$, then $C^+\in (G^+)^\sharp$. Indeed, $C^+$ is clearly closed under conjugation. Moreover, if $k_1,k_2\in C^+$, say $k_i=a_i c_ia_i^{-1}$, then since $c_i \in C$, there exists $g\in G$ for which $c_2=gc_1 g^{-1}$. Hence, $k_2 = a_2 g a_1 ^{-1}(a_1c_1 a_1^{-1})a_1 g^{-1}a_2^{-1} =(a_2 ga_1^{-1}) k_1 (a_2 ga_1^{-1})^{-1} $, that is $k_1,k_2$ are $G^+$-conjugates.
\end{remark}

\begin{proof}[Proof of Proposition~\ref{proposition psi bridge}]
For any $\chi\in \Irr(G)$, we have the identity (see \emph{e.g.}~\cite[\S 4]{Mar})
\begin{equation}
 L(s,L/K,\chi)= L(s,L/M, {\rm Ind}_G^{G+}(\chi)),
 \label{equation Artin induction L functions}
\end{equation}
and hence
\begin{equation}
\psi(x;L/K,\chi) = -\frac 1{2\pi i}\int_{\Re(s)=2} \frac{L'(s,L/K,\chi)}{L(s,L/K,\chi)} \frac{x^s}{s} ds =  \psi(x;L/M,{\rm Ind}_G^{G+}(\chi)).
\label{equation Artin induction}
\end{equation}
As a consequence, 
$$ \psi(x;L/K,t) = \sum_{\chi \in \Irr(G)} \overline{\widehat {t}(\chi)} \psi(x;L/K,\chi)=\sum_{\chi \in \Irr(G)} \overline{\widehat {t}(\chi)} \psi(x;L/M,{\rm Ind}_G^{G^+}\chi)=\psi(x;L/M,t^+). $$
Now, if $D\subset G$ is conjugacy invariant and $\chi \in \Irr(G^+)$, then by Lemma~\ref{lemma Fourier transform t+},
$$\widehat {1^+_{D}}(\chi) = \widehat{1_{D}}(\chi|_G) = \frac 1{|G|} \sum_{\substack{C\in G^\sharp: \\ C \subset D}}\chi(C^+)|C|,$$ since $\chi|_G(C)=\chi(C^+)$. It follows that
\begin{equation}
1^+_{D}  = \sum_{\substack{C\in G^\sharp: \\ C \subset D}}  \frac{|C|}{|G|}\sum_{\chi \in \Irr(G^+)} \overline{\chi}(C^+)\chi =  \sum_{\substack{C\in G^\sharp: \\ C \subset D}}  \frac{|C|}{|G|}\frac{|G^+|}{|C^+|} 1_{C^+}\,. 
\label{equation 1C+}
\end{equation}
The proof of~\eqref{equation link between C+ G+ with stabs}
 follows from combining this with~\eqref{equation psi bridge} in the form  
$$ \psi(x;L/K,1_D) = \psi(x;L/M,1^+_D). $$
\end{proof}

In the next lemma we show that up to ramified primes, the counting function $\pi(x;L/K,t)$ is determined by $t^+$, rather than by $t$. Note however that $\pi(x;L/K,t)$ and $\pi(x;L/\Q,t^+)$ are not equal in general. 
\begin{lemma}
\label{lemma pi+ equality}
Let $L/K$ be an extension of number fields such that $L/\Q$ and $K/\Q$ are Galois. If $t_1,t_2:G\rightarrow \mathbb C$ are class functions such that $t_1^+=t_2^+$, then
 $$ |\pi(x;L/K,t_1)-\pi(x;L/K,t_2)|\leq \sup(|t_1-t_2|) \cdot \#\{ \mathfrak p\triangleleft \mathcal O_K  \text{ ramified in }L/K\}. $$
\end{lemma}
\begin{proof}
We let $t:=t_1-t_2$, so that $t^+\equiv 0$. For any $\chi\in \Irr(G^+)$ and $\ell\in \mathbb N$, Lemma~\ref{lemma induction of premiage function} implies that 
$$  r_{\ell,\chi}^+|_G = [G^+:G] \cdot r_{\ell,\chi}.$$
Hence, 
$$ \widehat{(t(\cdot^\ell))^+}(\chi)=  \langle \chi, (t(\cdot^\ell))^+ \rangle_{G^+}=   \langle \chi|_G, t(\cdot^\ell) \rangle_{G} = \langle r_{\ell,\chi|_G} ,t\rangle_G=\langle r_{\ell,\chi} ,t\rangle_G ,$$
which by the class function equality above equals
\begin{align*}
[G^+:G]^{-1}\langle  r_{\ell,\chi}^+|_G,t\rangle_G& = [G^+:G]^{-1} \langle  r_{\ell,\chi}^+,t^+\rangle_{G^+}=0\,.  
\end{align*}
We deduce that $(t(\cdot^\ell))^+ \equiv 0$, and as such, denoting by $D_{L/K}$ the relative discriminant of $L/K$ and applying inclusion-exclusion,

\begin{align*}
 \theta(x;L/K,t)&:=  \sum_{ \substack{ \mathfrak p \triangleleft \mathcal O_K }} t(\varphi_{\mathfrak p}) h( \mathcal N\mathfrak p/ x) \log (\mathcal N\mathfrak p)  = \sum_{\ell\geq 1} \mu(\ell) \psi(x^{\frac 1\ell},L/K,t(\cdot^\ell)) \\&=\sum_{\ell\geq 1} \mu(\ell) \psi(x^{\frac 1\ell},L/\Q,(t(\cdot^\ell))^+)=0, 
\end{align*}
by Proposition~\ref{proposition psi bridge}. Applying summation by parts, we deduce that
$$ \pi(x;L/K,t)+\sum_{\substack{\mathfrak p \mid D_{L/K} \\ \mathcal N\mathfrak p\leq x}} t(\varphi_\mathfrak p)  = \int_1^{x} \frac{d \theta(u;L/K,t) }{\log u}=0. $$
\end{proof}

In Proposition~\ref{proposition psi bridge} we reduced our counting problem to one which will involve zeros of primitive $L$-functions, at the cost of working in a larger Galois group. In some of our results we will circumvent AC by doing the exact opposite (as is done classically): we will work in an abelian Galois group, and allow imprimitive $L$-functions.
This will be done using a result of Serre which goes back to Deuring and is in the spirit of Chebotarev's original reduction.
\begin{lemma}[{\cite[Section 2.7]{SeIHES}}]
\label{lemma transferance murmursar serre}
Let $L/K$ be a Galois extension with Galois group $G$, and let $C\subset G$ be a conjugacy class. For any $g\in C$, let $\ord(g)$ denote the order of the subgroup $\langle g\rangle $ generated by $g$. We have the equality
$$\psi(x;L/K , 1_C) = \frac{|C| \ord(g)}{|G|}\psi(x;L/L^{\langle g\rangle},1_{\{g\}}). $$

\end{lemma}
This identity follows again from Artin induction in the form
\begin{equation}
L(s,L/K,\frac{|G|}{|C|}1_C) = L(s,L/L^{\langle g \rangle},\ord(g_C) 1_{\{g_C\}}),
\label{equation Artin induction small group}
\end{equation}
which holds since 
$ {\rm Ind}_{\langle g \rangle}^G (\ord(g)1_{\{g\}})= {|G|}{|C|^{-1}}1_C$ (recall~\eqref{equation 1C+}).

We will also need the following consequence of Proposition~\ref{proposition psi bridge} and Lemma~\ref{lemma transferance murmursar serre}.

\begin{lemma}
\label{lemma order of vanishing}
Let $L/K$ be a Galois extension of number fields, and let $C\in G^\sharp $.  For any $s_0\in \mathbb C$ and $g\in C$, we have that
\begin{equation}
\sum_{ \chi\in \Irr(G)} \chi(C^{-1}) \ord_{s=s_0} L(s,L/K,\chi) = \sum_{ \chi\in \Irr(\langle g \rangle)} \chi(g^{-1}) \ord_{s=s_0} L(s,L/L^{\langle g \rangle},\chi).
\label{equation link sum real zeros H_g}
\end{equation}
Assuming moreover that $L/\Q$ is Galois,
for any class function $t:G\rightarrow \mathbb C$ and $s_0 \in \mathbb C$ we have that
\begin{equation}
\sum_{ \chi\in \Irr(G)} \overline{\widehat {t} (\chi)} \ord_{s=s_0} L(s,L/K,\chi) = \sum_{ \chi\in \Irr(G^+)} \overline{\widehat {t^+} (\chi)} \ord_{s=s_0} L(s,L/\Q,\chi).
\label{equation link sum real zeros big to small group}
\end{equation}
\end{lemma}
\begin{proof}
The first claimed identity clearly follows from the following:
$$-\sum_{ \chi\in \Irr(G)} \chi(C^{-1})  \frac{L'(s,L/K,\chi)}{L(s,L/K,\chi)} = -\sum_{ \chi\in \Irr(\langle g \rangle)} \chi(g^{-1})\frac{L'(s,L/L^{\langle g \rangle},\chi)}{L(s,L/L^{\langle g \rangle},\chi)},  $$
which we will establish for $s>1$ using the induction property for Artin $L$-functions. By uniqueness of analytic continuation, this is sufficient. 

The summatory function of the coefficients of the Dirichlet series on the left hand side is given by
\begin{equation}
 \sum_{ \chi\in \Irr(G)} \chi(C^{-1}) \sum_{\substack{\mathfrak p\triangleleft \mathcal O_{K}\\  \mathcal N \mathfrak p^m \leq x \\ m\geq 1}}\chi(\varphi_\mathfrak p^m)\log (\mathcal N \mathfrak p)=\frac{|G|}{|C|}\sum_{\substack{\mathfrak p\triangleleft \mathcal O_{K} \\  \mathcal N \mathfrak p^m \leq x \\ m\geq 1}} \mathbf 1_{ C}(\varphi_\mathfrak p^m) \log (\mathcal N \mathfrak p);  
 \label{equation equation to establish for tranfer of orders}
\end{equation}
the same holds for the coefficients of the Dirichlet series on the right hand side of the equality to be established, by virtue of Lemma~\ref{lemma transferance murmursar serre}. This concludes the proof of \eqref{equation link sum real zeros H_g}. The proof of~\eqref{equation link sum real zeros big to small group} is similar using Proposition~\ref{proposition psi bridge}.
\end{proof}

We are now ready to relate $\psi(x;L/K,t)$ and $\pi(x;L/K,t)$ (resp. $\psi(x;L/K,C)$ and $\pi(x;L/K,C)$) with the zeros of Artin $L$-functions associated to the extension $L/\Q$ (resp. $L/L^{\langle g \rangle}$, for any $g\in C$), which ultimately will allow us to use the language of random variables. Under AC, the calculation of the mean and variance in Proposition~\ref{proposition link with random variables} below can be deduced from combining~\cite[Theorem 2.1]{De} (see also~\cite{Fi2}) with induction properties of Artin $L$-functions. For the sake of completeness and in order to provide a full decomposition into sums of independent random variables, we decided to give further details while trying to stay brief. This closely follows~\cite[Theorem 2.1]{De},~\cite{Fi2} and~\cite[\S 5.1]{Ng}.

\begin{lemma}\label{prop:explicitformulae1}
Let $L/K$ be a Galois extension of number fields for which AC holds.  If $\chi$ is an irreducible character of $G=\Gal(L/K)$ and $C\subset G$ is a conjugacy class, then for any $x,X\geq 2$,
\begin{align}
\psi(x;L/K,\chi) &=\delta_{\chi=1}x-\sum_{\substack{\rho_{\chi} \\ |\Im(\rho_{\chi})| \leq X}}\frac{x^{\rho_{\chi}}}{\rho_{\chi}} +O_{L,K}\Big(\log x +\frac xX (\log (xX))^2\Big);
\label{equation first explicit formula}\\
\psi(x;L/K,C) &= \frac{|C|}{|G|} x-\frac{|C|}{|G|} \sum_{ \chi \in \Irr(G)} \chi(C^{-1}) \sum_{\substack{\rho_{\chi} \\ |\Im(\rho_{\chi})| \leq X}}\frac{x^{\rho_{\chi}}}{\rho_{\chi}} +O_{L,K}\Big(\log x +\frac xX (\log (xX))^2\Big),
\label{second explicit formula} 
\end{align}
where $\delta_{\chi=1}$ is $1$ when $\chi$ is the trivial character and $0$ otherwise.
In both formulas the sum is over the zeros $\rho_\chi$ of the Artin $L$-function $L(s,L/K,\chi)$ 
in the critical strip $\Re(s) \in (0,1)$.  
\end{lemma}

\begin{proof}
See for instance~\cite[(5.8)]{Ng}.
\end{proof}

\begin{corollary}
\label{cor:explicitformulae1}
Let $L/K$ be a Galois extension of number fields, let $C\subset \Gal(L/K)$ be a conjugacy class and let $g_C$ be any representative of $C$. For $
x,X\geq 2$ we have the estimate 
\begin{multline} \label{equation explicit formula for residue classes unconditional}
\frac{\log x}{x^{\beta^t_{L}}}\Big( \frac{|G|}{|C|}\pi(x;L/K,C)-{\rm Li}(x) \Big)=
 -\delta_{\beta^t_{L}=\frac 12}r(C)-\frac 1{\beta^t_{L}}\sum_{\chi\in \Irr(G)} \overline{\chi}(C) \ord_{s=\beta^t_{L}} L(s,L/K,\chi)\\- \sum_{ \chi \in \Irr(\langle g_C \rangle)} \overline{\chi}(g_C) \sum_{\substack{\rho_{\chi} \\ 0< |\gamma_{\chi}|\leq X }}  \frac{x^{\rho_{\chi}-\beta^t_L}}{\rho_{\chi}}
 +O_{L,K}\Big(\frac 1{\log x}+ \frac{x^{1-\beta^t_L}}X (\log(xX))^2\Big)\,,
 \end{multline}
where $\rho_{\chi}$ runs through the nontrivial zeros of $L(s,L/K,\chi)$.
If in addition we assume that $L/\Q$ is Galois and that AC holds, then for any class function $t:G\rightarrow \C$,
\begin{multline} \label{equation explicit formula for residue classes}
\frac{\log x}{x^{\beta^t_{L}}}\Big( \pi(x;L/K,t)- \overline{\widehat{t}(1)}{\rm Li}(x) \Big)=
 -\delta_{\beta^t_{L}=\frac 12} \langle t,r \rangle_G-\frac 1{\beta^t_{L}}\sum_{\chi\in \Irr(G)} \overline{\widehat {t^+} (\chi)} \ord_{s=\beta^t_{L}} L(s,L/K,\chi)\\- \sum_{ \chi \in \Irr(G^+)} \overline{ \widehat{{t^+}} (\chi)} \sum_{\substack{\rho_{\chi} \\ 0<|\gamma_{\chi}|\leq X }}  \frac{x^{\rho_{\chi}-\beta^t_L}}{\rho_{\chi}}
 +O_{L,K}\Big(\frac 1{\log x}+ \frac{x^{1-\beta^t_L}}X (\log(xX))^2\Big)\,.
\end{multline}
\end{corollary}
\begin{proof}
We first establish \eqref{equation explicit formula for residue classes}.
Arguing as in~\cite[\S 5.1]{Ng} and~\cite[\S 4.3]{De}, we see that 
\begin{multline}
 \frac{\log x}{x^{\beta^t_{L}}}\Big( \pi(x;L/K,t)-\overline{\widehat{t}(1)}{\rm Li}(x) \Big) \\=  x^{-\beta^t_{L}} \Big(\psi(x;L/K,t)-\overline{\widehat{t}(1)}x-\overline{\widehat{t(\cdot^2)}(1)}x^{\frac 12}\Big)
 +O_{L,K}\Big( \frac {1}{\log x}\Big).
\label{equation from pi to psi}
\end{multline}
By Proposition \ref{proposition psi bridge}, this is
\begin{align*}
  &= x^{-\beta^t_{L}} \psi(x; L/\Q, t^+ ) -\overline{\widehat{t}(1)}x^{1-\beta^t_{L}}-\delta_{\beta^t_{L}=\frac 12}\langle t,r\rangle_G+O_{L,K}\Big( \frac 1{\log x}\Big)\\
  &=x^{-\beta^t_{L}} \sum_{\chi \in \Irr(G^+)} \overline{\widehat {t^+} (\chi)}\psi(x; L/\Q, \chi)-\overline{\widehat{t}(1)}x^{1-\beta^t_{L}}-\delta_{\beta^t_{L}=\frac 12}\langle t,r\rangle_G+O_{L,K}\Big( \frac 1{\log x}\Big).
\end{align*}
The estimate~\eqref{equation explicit formula for residue classes} then follows from applying Lemma~\ref{prop:explicitformulae1}.
The proof of~\eqref{equation explicit formula for residue classes unconditional} is similar (note that ${ \rm Ind}_{\langle g_C \rangle}^G (\ord(g_C)1_{g_C})=\frac{|G|}{|C|} 1_C$).
\end{proof}

To state the next Proposition we first define the following multisets of zeros of Artin $L$-functions, where $\beta_{L}$ and $\beta^{\r}_{L/K}$ are defined in Theorem~\ref{theorem mean variance}:

\begin{equation}
\label{equation definition Z L/K}
Z_{L}:=\{ \gamma \in \mathbb R_{>0} : \zeta_L(\beta_{L}+i\gamma)=0\}; 
\end{equation}
\begin{equation}
\label{equation definition Z real L/K}
Z^{t}_{L}:=\bigcup_{\substack{\chi \in \Irr(\Gal(L/\Q))\\ \widehat{t^+}(\chi)\neq 0  }}\{ \gamma \in \mathbb R_{>0} : L(\beta^{t}_{L}+i\gamma,\chi)=0\}. 
\end{equation}
Recall also the definition~\eqref{equation definition order vanishing}.

\begin{proposition}
\label{proposition link with random variables}

Let $L/K$ be a Galois extension of number fields, let $G=\Gal(L/K),$  and fix $t:G\rightarrow \C$ a class function. 
The function $E(y;L/K,t)$  admits a limiting distribution. Moreover, the associated random variable  $X(L/K;t)$ is such that
\begin{align}
\notag \E[X(L/K;t)]  &=-\delta_{\beta^t_{L}=\frac 12}\langle t,r \rangle_G-\frac 1{\beta^t_{L}}{\rm ord}_{s=\beta^t_{L}}L(s,L/K,t) \notag \\
&=-\delta_{\beta^t_{L}=\frac 12}\langle \widehat t, \epsilon_2 \rangle_{\Irr(G)}-\frac 1{\beta^t_{L}}{\rm ord}_{s=\beta^t_{L}}L(s,L/K,t).
\label{equation expectancy with characters}
\end{align}
Furthermore, 
we have that
\begin{equation}
\V[X(L/K;t)]=2\sums_{\gamma \in Z_{L} }  \frac{|{\rm ord}_{s=\beta^t_L+i\gamma}L(s,L/K,t)|^2}{(\beta^t_{L})^2+\gamma^2},
\label{equation variance small group}
\end{equation}  
where the starred sum means a sum without multiplicities. If $L/\Q$ is Galois, then we have the alternative\footnote{The advantage of this formula is that under BM, we have that $|{\rm ord}_{s=\beta^t_L+i\gamma}L(s,L/K,t)| \leq M_0 \sup |\widehat {t^+}(\chi)|$.} formula
\begin{equation}
\V[X(L/K;t)]=2\sums_{\gamma \in Z^t_{L} }  \frac{|{\rm ord}_{s=\beta^t_L+i\gamma}L(s,L/\Q,t^+)|^2}{(\beta^t_{L})^2+\gamma^2}
\label{equation variance big group}
\end{equation}  
Assuming in addition that LI$^{-}$ holds, we have the simplified formula
\begin{align*}\V[X(L/K;t)] &= 2\sum_{\chi\in \Irr(G^+)}
|\widehat {t^+}(\chi)|^2 \sum_{\gamma_{\chi}\neq 0} \frac {1}{(\beta^t_{L})^2+\gamma_{\chi}^2}.
 \end{align*}
 \end{proposition}

\begin{remark}
\label{remark variance small group}
If $L/\Q$ is Galois and $t^+\equiv 0$, then $L(s,L/\Q,t^+) \equiv 1$ and consequently $\V[X(L/K;t)]=0$. In this case, by Corollary~\ref{corollary when t+ is 0} we also have $\E[X(L/K;t)]=0$, and the measure associated to $X(L/K;t)$ is just a Dirac delta centered at $0$. This holds for example with the class function $t=|C_1|^{-1}1_{C_1}-|C_2|^{-1}1_{C_2}$, where  $C_1,C_2 \in G^{\sharp}$ are distinct and such that $C_1^+=C_2^+$ (as in Theorem~\ref{th:BiasRadical}).

\end{remark}

\begin{proof}[Proof of Proposition \ref{proposition link with random variables}]
We will combine the arguments in~\cite[\S 5.1]{Ng}, ~\cite[Lemma 2.6]{Fi2}, ~\cite[Th. 2.1]{De}, and~\cite[Theorem 1.2]{ANS} (one cannot apply those results directly, since we are not assuming GRH and moreover $t$ is complex-valued).
For any $T\geq 1$, we define 
\begin{equation}
\beta^t_{L}(T)= \begin{cases}
 \sup\{ \Re(\rho) :  |\Im(\rho)|\leq T,L(\rho,L/\Q,\chi)=0 ; \chi \in {\rm supp}(\widehat{t^+})  \} & \text{ if AC holds for } L/\Q ;\\
  \sup\{ \Re(\rho) :  |\Im(\rho)|\leq T,\zeta_L(\rho)=0 \} &\text{ otherwise},
\end{cases}
\end{equation}
so that, using the definition~\eqref{equation definition beta}, one has $\beta^t_{L}=\beta^t_L(\infty)$.
Using the decomposition $ t= \sum_{C\in G^\sharp} t(C) 1_C $ and letting $g_C$ be any element of $C$,
we deduce from Corollary \ref{cor:explicitformulae1} that for $x\geq 2$ and $X\geq T\geq 2$,
\begin{multline} 
\frac{\log x}{x^{\beta^t_{L}}} (\pi(x;L/K,t)-\widehat t(1){\rm Li}(x)) = -\delta_{\beta^t_{L}=\frac 12}\langle t,r \rangle_G-\frac{1}{\beta^t_{L}} \ord_{s=\beta^t_{L}} L(s,L/K,t)\\
-\sum_{C\in G^\sharp} t(C)  \sum_{ \chi \in \Irr(\langle g_C\rangle)} \overline{\chi}(g_C) \Big(\sum_{\substack{\rho_\chi = \beta^t_L+i\gamma_\chi\\ 0<|\gamma_\chi|\leq T}} +\sum_{\substack{\rho_\chi = \beta^t_L+i\gamma_\chi\\ T<|\gamma_\chi|\leq X}} \Big) \frac{x^{i\gamma_{\chi}}}{\rho_{\chi}}\\
 +O_{L,K}\Big(\frac 1{\log x}+ \frac{x^{1-\beta^t_L}}X (\log(xX))^2+x^{\beta^t_L(T)-\beta^t_L}(\log X)^2\Big)
 \,.
 \label{equation pre limiting distribution}
 \end{multline}
Taking $X=x=e^y$, we see that 
$$ \int_{2}^Y \Big|\sum_{C\in G^\sharp} t(C)\sum_{1\neq \chi \in \Irr(\langle g_C\rangle)} \overline{\chi}(g_C) \sum_{\substack{\rho_\chi\\ T<|\gamma_\chi|\leq \ee^y}}  \frac{\ee^{y \rho_{\chi}}}{\rho_{\chi}}\Big|^2dy \ll_{L,K} \frac {(Y+\log T) (\log(d_L T ))^2}{T}, $$
and we deduce as in~\cite[\S 5.1]{Ng}, ~\cite[Lemma 2.6]{Fi2}, ~\cite[Th. 2.1]{De}, and~\cite[Theorem 1.2]{ANS} that the function
$$ \overrightarrow E(y) = \frac{y}{e^{\beta^t_{L}}} (\Re(\pi(e^y;L/K,t)-\widehat t(1){\rm Li}(x)),\Im(\pi(e^y;L/K,t)-\widehat t(1){\rm Li}(x)))$$  
is $B^2$ almost-periodic. In particular, this function admits a limiting distribution. 

To compute the first two moments of this distribution, we deduce using the arguments in the proofs of~\cite[Lemma 2.5, 2.6]{Fi2} (see also~\cite[Theorem 1.14]{ANS} and~\cite[Th. 2.1]{De}) that
\begin{align*}
 \E[X_\nu] &= \lim_{Y\rightarrow \infty} \frac 1Y \int_2^Y \frac{y}{e^{\beta^t_{L}}}(\pi(e^y;L/K,t)-\widehat t(1){\rm Li}(e^y)) dy\\ &=-\delta_{\beta^t_{L}=\frac 12}\langle  t, r \rangle_{G}-\frac 1{\beta^t_{L}}{\rm ord}_{s=\beta^t_{L}}L(s,L/K,t). 
\end{align*}
Similarly,
\begin{align}
 \V[X_\nu] &=2\sums_{\gamma \in Z_{L} }  \frac{1}{(\beta^t_{L})^2+\gamma^2} \Big|\sum_{C\in G^\sharp} t(C)\sum_{ \chi \in \Irr(\langle g_C\rangle)}\overline{\chi}(g_C){\rm ord}_{s=\beta^t_{L}+i\gamma}L(s,L/L^{\langle g_C\rangle},\chi )\Big|^2  \label{equation variance avec abelien -1} \\
 &=2\sums_{\gamma \in Z_{L} }  \frac{|{\rm ord}_{s=\beta^t_L+i\gamma}L(s,L/K,t)|^2}{(\beta^t_{L})^2+\gamma^2},
 \notag
\end{align}
by~\eqref{equation Artin induction small group}.
Moreover, if $L/\Q$ is Galois, then by~\eqref{equation Artin induction L functions} this is 
$$
= 2\sums_{\gamma \in Z_{L} }  \frac{1}{(\beta^t_{L})^2+\gamma^2} |{\rm ord}_{s=\beta^t_{L}+i\gamma}L(s,L/\Q,t^+ )|^2.
$$
\end{proof}

Under AC, GRH and LI$^-$ and for real-valued class functions $t$, we give an explicit expression for the random variables in Proposition \ref{proposition link with random variables}. We stress that in order for the random variables appearing in this expression to be independent, it is crucial to express $\pi(x;L/K,C)$ in terms of zeros of $L(s,L/\Q,\chi)$ (rather than $L(s,L/K,\chi)$) associated to irreducible characters of $\Gal(L/\Q)$; indeed these $L$-functions are believed to be primitive.

\begin{lemma}
Let $L/K$ be an extension of number fields such that $L/\Q$ is Galois, and for which AC, GRH and LI$^-$ hold. Let $G=\Gal(L/K)$, $G^+=\Gal(L/\Q)$, and fix a class function $t:G\rightarrow \R$.
Then, we have the following equality (in distribution) of random variables:
\begin{equation} X(L/K;t)\overset{\rm d}= -\langle t,r\rangle_G -
2{\rm ord}_{s=\frac 12}L(s,L/K,t) + \sum_{\substack{1\neq \chi\in \Irr(G^+)}}|\widehat{t^+}(\chi)|\sum_{\gamma_{\chi}>0}  \frac{2 X_{\gamma_{\chi}}}{\big(\frac 14+\gamma_{\chi}^2\big)^{\frac 12}}.
\label{random variable for two classes}
\end{equation}
Here, the random variables $X_{\gamma}$ are defined by $X_{\gamma}=\Re(Z_{\gamma})$ where the $Z_{\gamma}$ are i.i.d. random variables uniformly distributed on the unit circle in $\mathbb C$. 
\label{lemma random var under LI}
\end{lemma}
\begin{proof}[Sketch of proof]
This is an extension 
of the random variable approach for the classical Chebyshev bias (where only Dirichlet $L$-functions are needed) explained in~\cite[\S 2.1]{FiMa}. 
Part of the connection with the independent random variables $Z_\gamma$ uniformly distributed on the unit circle in $\C$ comes from applying the Kronecker--Weyl Theorem in~\eqref{equation pre limiting distribution} (the details for Dirichlet $L$-functions are in \emph{loc. cit.} and in the general case of Artin $L$-functions they will appear in A. Bailleul's forthcoming PhD Thesis).  
Observe that to go from~\eqref{equation explicit formula for residue classes}  to~\eqref{random variable for two classes}, one uses the functional equation for Artin $L$-functions (see \emph{e.g.}~\cite[Chap. 2~\S2]{MuMu}) to pair up conjugate critical zeros. 
We can actually compute all the cumulants of $X(L/\Q;t)$ in this way, and thus recover its characteristic function. This will be useful in Section~\ref{section central limit theorem}.
\end{proof}

\begin{remark}
As mentioned above the importance of having linearly independent imaginary parts of $L$-function zeros goes back to Wintner \cite{Win} and is explained by the role played by the Kronecker--Weyl Theorem in our analysis. Remarkably, recent work of Martin--Ng~\cite{MN} and Devin~\cite{De}  manages to show absolute continuity of limiting logarithmic distributions under weaker assumptions \emph{via} the introduction of the notion of \emph{self-sufficient} zero.
\end{remark}

\section{Artin conductors}\label{section:ArtinCond}

\subsection{Link with ramification and representation theory} 
 
 In this section we analyze the ramification data of a given Galois extension $L/K$. This data is related with the expressions obtained for the variance 
 of the random variable $X(L/K;t)$
 in 
 Proposition~\ref{proposition link with random variables}.

 Let us first review the definition of the Artin conductor $A(\chi)$, following \cite{Fr} (this is a quite standard invariant to consider; see \emph{e.g.}~\cite[Chap. 2, \S2]{MuMu} or~\cite[(5.2)]{LaOd}). Consider a finite Galois extension of number fields $L/K$ with Galois group $G$. For $\mathfrak p$ a prime ideal of $\O_K$ and $\mathfrak P$ a prime ideal of $\O_L$ lying above 
 $\mathfrak p$, the higher ramification groups
 form a sequence 
 $(G_i(\mathfrak P/\mathfrak p))_{i\geq 0}$ of subgroups of $G$ (called 
 filtration of the inertia group 
 ${\rm I}(\mathfrak P/\mathfrak p)$) defined as follows:
$$ G_i(\mathfrak P/\mathfrak p) :=\{ \sigma\in G : \forall z \in \mathcal O_L,\, (\sigma z - z) \in 
\mathfrak P^{i+1}\}. $$ 
Each $G_i(\mathfrak P/\mathfrak p)$ only depends on $\mathfrak p$ up to conjugation 
and $G_0(\mathfrak P/\mathfrak p)={\rm I}(\mathfrak P/\mathfrak p)$. For clarity let us 
fix prime ideals $\mathfrak p$ and $\mathfrak P$ as above and write $G_i$ for 
$G_i(\mathfrak P/\mathfrak p)$.
 Given a representation $\rho\colon G \rightarrow GL(V)$ on a complex vector space $V$,  the subgroups $G_i$ act on $V$ through $\rho$ and we will denote by $V^{G_i}\subset V$ the subspace of $G_i$-invariant vectors. Let $\chi$ be the character of $\rho$ and 
\begin{equation}
n(\chi,\mathfrak p):=\sum_{i=0}^{\infty} \frac{|G_i|}{|G_0|} {\rm codim} V^{G_i},
\label{eq:definition exponent codimensions}
\end{equation}
which was shown by Artin to be an integer (see \emph{e.g.}~\cite[Chap. $6$, Th. $1$']{SeCL}). The \emph{Artin conductor 
of $\chi$} is the ideal
$$
\mathfrak f(L/K,\chi):=\prod_\mathfrak p \mathfrak p^{n(\chi,\mathfrak p)}\,.
$$
Note that the set indexing the above product is finite since only finitely many prime ideals $\mathfrak p$ of $\O_K$ ramify in $L/K$. We set

\begin{equation}
\label{equation definition artin conductor}
A(\chi):=d_K^{\chi(1)}\mathcal N_{K/\Q}(\mathfrak f(L/K,\chi))\,,
\end{equation}

where $d_K$ is the absolute value of the absolute discriminant of the number field $K$ and $\mathcal N_{K/\Q}$ is the relative ideal norm with respect to $K/\Q$ (we will use the slight abuse of notation that identifies the value taken by this relative norm map with the positive generator of the corresponding ideal).
One can show (see \emph{e.g.}~\cite[Chap. $6$, consequences of Prop. $6$]{SeCL}) the following 
equalities of ideals in $\O_K$ (known as the conductor--discriminant formula):
\begin{equation}
\mathfrak f(L/K,\chi_{\rm reg})=\prod_{\chi \in \Irr(G)} \mathfrak f(L/K,\chi)^{\chi(1)} =D_{L/K}\,,
\label{equation conductor disc relative}
\end{equation} 
where $D_{L/K}$ is the relative discriminant of $L/K$ and $\chi_{\rm reg}$ is the character of the regular representation of $G$. In particular, by~\cite[Chap. 5, Th. 31]{ZarSam} we have the identity
\begin{equation}
d_K^{|G|}\mathcal N_{K/\Q}(D_{L/K})= A(\chi_{{\rm reg}})=d_L.
\label{equation conductor discriminant}
\end{equation}
We now estimate $A(\chi)$ for irreducible characters $\chi\in \Irr(G)$.

\begin{lemma}
\label{lemma first bound on Artin conductor}
Let $L/K$ be a finite Galois extension. Let $\chi$ be an irreducible character of $G=\Gal(L/K)$, and assume that either $K\neq \Q$, or that  $\chi$ is non-trivial. Then, one has the bounds
$$ \max(1,[K:\Q]/2)\chi(1) \leq \log A(\chi) \leq 2\chi(1) [K:\Q]\log ({\rm rd}_L)\,, $$
where the root discriminant ${\rm rd}_L$ is defined by~\eqref{eq:defrd}.
The upper bound, due to \cite{MuMu}, is unconditional. The lower bound is unconditional\footnote{It actually also holds for the trivial character in this case.} if $K\neq \Q$, and holds assuming $L(s,L/\Q,\chi)$ can be extended to an entire function otherwise.
\end{lemma}

\begin{proof}
In Lemma \ref{lemma finer bounds on Artin conductor} we will reproduce the proof of the upper bound found in \cite{MuMu}.

For the lower bound we consider two cases. Suppose first that $[K:\Q]\geq 2$. Then we use the lower bound for the absolute discriminant of a number field 
obtained \emph{e.g.} in~\cite[Th. 2.4(1)]{BD} (noting that the sum over 
$\mathfrak P$ on the right hand side of their formula is positive) and which holds for any $y>0$:
\begin{equation}\label{eq:OdlyzkoLowerBound}
\log d_K\geq r_1(1-I_1(y))+[K:\Q](\gamma+\log(4\pi)-I_2(y))-\frac{12\pi}{5\sqrt{y}}
\end{equation}
where $\gamma$ is the Euler constant, $r_1$ is the number of real embeddings of $K/\Q$, 
and
$$
f(x)=\left(\frac{3}{x^3}(\sin(x)-x\cos(x))\right)^2\,,$$ 
$$I_1(y)=
\int_0^\infty\frac{1-f(x\sqrt{y})}{2\cosh^2(x/2)}{\rm d}x  \,,\hspace{1cm} I_2(y)=
\int_0^\infty\frac{1-f(x\sqrt{y})}{\sinh(x)}{\rm d}x\,. 
$$
Setting $y=20$ and using the numerical integration method implemented in 
{\tt SageMath}~(\cite{sage}) we obtain that $I_1(20)<0.08$ and $I_2(20)<1.73$. In particular the quantity 
$\gamma-1/2+\log(4\pi)-I_2(20)$ is positive and we deduce from~\eqref{eq:OdlyzkoLowerBound} that
\begin{align*}
\log d_K-\frac{[K:\Q]}{2}&\geq [K:\Q]\bigl(\gamma-\frac{1}{2}+\log(4\pi)-I_2(20)\bigr)-\frac{6\pi}{5\sqrt{5}}\\
&\geq 2\bigl(\gamma-\frac{1}{2}+\log(4\pi)-I_2(20)\bigr)-\frac{6\pi}{5\sqrt{5}}\\
&\geq 0.07\,.
\end{align*}
 Moreover, $\mathcal N_{K/\Q}(\mathfrak f(\chi))\geq 1$ 
so that one trivially deduces $\log A(\chi)\geq \chi(1)\log d_K\geq \chi(1)[K:\Q]/2$. 
If $K=\Q$ the lower bound is a consequence of Odlyzko type lower bounds on 
Artin conductors (see \emph{e.g.}~\cite[Th. 3.2]{Pi} where the author proves $\mathfrak f(\chi)\geq 2.91^{\chi(1)}$), that are conditional on Artin's conjecture.
\end{proof}

It is known (see~\cite{Pi2}) that the lower bound in Lemma~\ref{lemma first bound on Artin conductor} is optimal. Nevertheless, will show in the next lemma that if one has good estimates on character values, then it is possible to improve both bounds in Lemma~\ref{lemma first bound on Artin conductor}, and in some cases to deduce the exact order of magnitude of $\log A(\chi)$ (for instance in the case of $G=S_n$; see Lemma \ref{lemma pre lower bound variance S_n big conj classes}).

\begin{lemma}
\label{lemma finer bounds on Artin conductor}
Let $L/K$ be a finite Galois extension. For any character $\chi$ of $G=\Gal(L/K)$, we define\footnote{The inequality $M_\chi\leq 1$ is a straightforward consequence of the standard fact 
	according to which a complex linear representation of a finite group can always be 
	considered as a unitary representation with respect to some inner product on the representation space.  Moreover, if $L=K$, then we set $M_\chi := 0$.
}
 $$M_{\chi}:=\max_{1\neq \sigma \in G}\frac{|\chi(\sigma)|}{\chi(1)}\leq 1\,.$$
Then we have the bounds
$$ (1-M_{\chi}) \chi(1) [K:\Q]\log ({\rm rd}_L) \leq  \log A(\chi) 
\leq (1+M_{\chi}) \chi(1)
[K:\Q]\log ({\rm rd}_L)\,.$$
\end{lemma}
\begin{proof}
The proof is inspired by~\cite[Proof of Prop. 7.4]{MuMu}. Let $\rho\colon G\rightarrow GL(V)$ be a complex representation with character $\tau$, 
let $\mathfrak p$ be a prime ideal of $\O_K$ and let $(G_i)$ be the attached filtration of inertia (defined up to conjugation in $G$). We start with the following identity:

\begin{equation}
{\rm codim}\ V^{G_i}=\tau(1)-\dim V^{G_i} =\tau(1)-\langle \tau,{\bf 1}\rangle_{G_i}=\frac 1{|G_i|} \sum_{a\in G_i} \bigl(\tau(1)-\tau(a)\bigr)\,,
\label{eq:evaluation of codimensions}
\end{equation}
where $\langle\,,\,\rangle_{G_i}$ is defined as in~\S\ref{section:repgroup} and ${\bf 1}$ is the trivial representation.
If $\tau=\chi_{\rm reg}$ is the character of the regular representation of $G$ then for any 
$a\in G\setminus\{{\rm id}\}$,
$$
\chi_{\rm reg}(a)=\sum_{\varphi\in\Irr(G)}\varphi(1)\varphi(a)=0
$$
by the orthogonality relation~\eqref{orthogonality of congugacy classes}. 
Hence, combining \eqref{eq:definition exponent codimensions} with \eqref{eq:evaluation of codimensions}, we obtain that 
\begin{equation}
n(\chi_{\rm reg},\mathfrak p) = \frac 1{|G_0|}\sum_{i\geq 0} \sum_{1\neq a\in G_i}\chi_{\rm reg}(1)=\frac{|G|}{|G_0|}\sum_{i\geq 0}\bigl(|G_i|-1\bigr).
\label{equation valuation p de chireg}
\end{equation}
Similarly, setting $\tau=\chi$ in \eqref{eq:evaluation of codimensions}, we have that
\begin{equation}\label{eq:condalternate}
n(\chi,\mathfrak p)=
\frac{\chi(1)}{|G_0|} \sum_{i\geq 0} \sum_{1\neq a\in G_i} \left(1-\frac{\chi(a)}{\chi(1)}\right)\,.
\end{equation}
Combining our expressions for $n(\chi_{\rm reg},\mathfrak p)$ and $n(\chi,\mathfrak p)$ yields the bound
\begin{equation}
\label{eq:encadrement}
\left|n(\chi,\mathfrak p)-\frac{\chi(1)}{|G|}n(\chi_{\rm reg},\mathfrak p)\right|
\leq M_\chi\frac{\chi(1)}{|G|}n(\chi_{\rm reg},\mathfrak p)\,.
\end{equation}

We now establish the claimed bound on $\log A(\chi)$. Let $\nu_{\mathfrak p}$ denote the 
$\mathfrak p$-adic valuation on $\O_K$, and observe that~\eqref{eq:encadrement} implies the bound
$$
n(\chi,\mathfrak p)=\nu_{\mathfrak p}(\mathfrak f(L/K,\chi))\leq 
\frac{\chi(1)}{|G|}(1+M_\chi)n(\chi_{\rm reg},\mathfrak p)=
\frac{\chi(1)}{|G|}(1+M_\chi)\nu_{\mathfrak p}(D_{L/K})\,.
$$
We deduce that
$$
\log \mathcal N_{K/\Q}\bigl(\mathfrak f(L/K,\chi)\bigr)\leq \frac{\chi(1)}{|G|}(1+M_\chi)
\log \mathcal N_{K/\Q}(D_{L/K})\,.
$$
By adding $\chi(1)\log (d_K)$ on both sides we obtain the bound
\begin{align*}
\log A(\chi)&\leq \chi(1)(1+M_\chi)\biggl(\log (d_K)+\log\bigl(\mathcal N_{K/\Q}
(D_{L/K})^{\frac{1}{|G|}}\bigr)\biggr)\\
 &\leq \chi(1)(1+M_\chi)\log\bigl(d_L^{\frac{1}{|G|}}\bigr)\\
  &\leq \chi(1)(1+M_\chi)[K:\Q]\log ( {\rm rd}_L)\,.
\end{align*}
The lower bound of the lemma is deduced from~\eqref{eq:encadrement} in an analogous fashion.
\end{proof}

\subsection{Variance associated to the limiting distribution}
We now consider a Galois extension of number fields $L/K$ of group $G$ and estimate various sums indexed by zeros of the associated Artin $L$-functions. For class functions $t\colon G\rightarrow \C$ these sums are related to the variance and fourth moment of the random variable $X(L/K;t)$ defined in Proposition \ref{proposition link with random variables}.
For $\chi \in \Irr(\Gal(L/K))$, we define
\begin{equation}
 B(\chi):=\sum_{\gamma_\chi} \frac{1}{\frac 14+\gamma_{\chi}^2}; \hspace{1cm} B_0(\chi):= \sum_{\gamma_\chi\neq 0} \frac{1}{\frac 14+\gamma_{\chi}^2}; \hspace{1cm} B_2(\chi) :=\sum_{\gamma_\chi\neq 0} \frac{1}{(\frac 14+\gamma_{\chi}^2)^2}\,, 
 \label{equation definition B(chi)}
\end{equation}
 where the sums are indexed by the imaginary parts of the ordinates of the non-trivial zeros of $L(s,L/K,\chi)$, counted with multiplicities. In the next lemma we will determine the order of magnitude of $B(\chi),B_0(\chi)$ and $B_2(\chi)$. Note that under GRH, every non-trivial zero of $L(s,L/K,\chi)$ is of the form $\rho_\chi=\frac 12+i \gamma_\chi$ with $\gamma_\chi \in \mathbb R$, and thus  $\frac 14+\gamma_\chi^2 = |\rho_\chi|^2$. However, the constant $\frac 14$ in~\eqref{equation definition B(chi)} could be replaced by any fixed real number $\beta^2 \in [\frac 14,1]$, and that would not change the orders of magnitude of $B(\chi),B_0(\chi)$ and $B_2(\chi)$ (with constants independent of $\beta)$. Indeed, for $\gamma\in\mathbb R$ we have that $ (\frac 14 +\gamma^2) \leq (\beta^2 +\gamma^2) \leq 4 (\frac 14 +\gamma^2)$.

\begin{lemma}
Let $L/K$ be a finite Galois extension for which AC holds. For any character $\chi$ of $G=\Gal(L/K)$, 
 we have the estimates 
$$ B(\chi)\asymp B_0(\chi)\asymp B_2(\chi) \asymp \log (A(\chi)+2)\,.$$
\label{lemma bounds on B(chi)}
\end{lemma}

\begin{proof}
%

We begin with the Riemann--von Mangoldt formula~\cite[Th. 5.8]{IwKo} which we combine with the bound on the analytic conductor given in~\cite[\S 5.13]{IwKo}. In the notation of \emph{loc. cit.} the degree $d$ of the $L$-function $L(s,L/K,\chi)$ is relative to $\Q$ and thus equals $[K:\Q]\chi(1)$. For $T\geq 1$ we obtain the estimate
\begin{equation} N(T,\chi):= |\{\gamma_{\chi}\colon |\gamma_{\chi}|\leq T \}|=\frac T{\pi} \log \Big(\frac{A(\chi)T^{[K:\Q]\chi(1)}}{(2\pi e)^{[K:\Q]\chi(1)}}\Big) + O(\log ((A(\chi)+2)(T+4)^{[K:\Q]\chi(1)}))\,. \label{von mangoldt estimation} \end{equation}
It follows that
$$ N(2T,\chi)-N(T,\chi) = \frac T{\pi} \log \Big(\frac{A(\chi)(2T)^{[K:\Q]\chi(1)}}{(\pi e)^{[K:\Q]\chi(1)}}\Big) +O(\log((A(\chi)+2)(2T+8)^{[K:\Q]\chi(1)}))\,.$$
It is easy to see that for $T$ larger than an absolute constant 
the main term is at least twice as big as the error term (\emph{e.g.} if we let $C_0$ be the implied constant in the error term above, it suffices to take $T$ larger than $2\pi C_0$ and such that $T\log(2T/\pi e)\geq 2\pi C_0\log(2T+8)$).

Therefore there exists an absolute constant $T_0\geq 4\pi e$ such that
$$ N(2T_0,\chi)-N(T_0,\chi) \geq  \frac {T_0}{2\pi} \log\Big( A(\chi)
\Big(\frac{2T_0}{\pi e}\Big)^{[K:\Q]\chi(1)}\Big)\geq \frac{T_0}{2\pi} \log (A(\chi)+2),  $$
and hence
$$
B_0(\chi) \geq \frac 1{\frac 14 + (2T_0)^2}\frac{T_0}{2\pi} \log (A(\chi)+2) \gg \log (A(\chi)+2). $$
For the upper bound, one easily deduces from \eqref{von mangoldt estimation} that $ N(T,\chi) \ll T \log \frac{A(\chi)T^{[K:\Q]\chi(1)}}{(2\pi e)^{[K:\Q]\chi(1)}}$ for $T\geq 4\pi e$, and hence summation by parts yields that
\begin{align*} B_0(\chi)\leq B(\chi)  &= \sum_{\gamma_\chi} \frac{1}{\frac 14+\gamma_{\chi}^2} \leq  4N(2,\chi)+\int_2^{\infty} \frac{{\rm d}N(t,\chi)}{\frac 14+t^2} \\
& \ll \log (A(\chi)+2)+[K:\Q]\chi(1) + \int_2^{\infty} \frac{t^2 (\log A(\chi)+[K:\Q]\chi(1)\log t)}{(1/4+t^2)^2}{\rm d}t \\ 
&\ll \log (A(\chi)+2)+[K:\Q]\chi(1)\,.
\end{align*}
Since we are assuming AC, 
we can apply Lemma~\ref{lemma first bound on Artin conductor} to deduce that $B_0(\chi),B(\chi)\asymp \log (A(\chi)+2)$. The proof is similar for $B_2(\chi)$.
\end{proof}
\begin{remark} 
If $\chi$ is a Dirichlet character of conductor $q^*\geq 3$, then under GRH we can give an exact formula for $B(\chi)$ and $B_2(\chi)$, and deduce the more precise estimate
$$B(\chi) =\log q^* + O(\log\log q^*).$$
(This is achieved \emph{e.g.} by applying Littlewood's conditional bound on $\frac{L'(1,\chi)}{L(1,\chi)}$ to \cite[(10.39)]{MoVa}.)
Such an estimate is harder to establish for a general extension $L/K$. We have by~\cite[(5.11)]{LaOd} that 
$$ 2B(\chi)= \log A(\chi) +2\frac{\gamma_{\chi}'(1)}{\gamma_{\chi}(1)}+2\Re \frac{L'(1,\chi)}{L(1,\chi)},$$
where the gamma factor is given by 
$$ \gamma_{\chi}(s):= \Big( \pi^{-\frac{s+1}2} \Gamma (\tfrac{s+1}2) \Big)^{b(\chi)} \Big( \pi^{-\frac{s}2} \Gamma (\tfrac{s}2) \Big)^{a(\chi)}$$
for some nonnegative integers $a(\chi),b(\chi)$ such that $a(\chi)+b(\chi)=\chi(1)$. It follows that
$$ \frac{\gamma_{\chi}'(1)}{\gamma_{\chi}(1)} \asymp \chi(1).$$
As for the ``analytic term'', we could either use the following bound (see~\cite[Prop. 2.4.2.3]{Ng})
$$ \frac{L'(1,\chi)}{L(1,\chi)} \ll \chi(1) \log\log (A(\chi)+2), $$
or an estimate for its average as in \cite[Theorem 1.7]{FiMa}.
The problem with this individual bound for a given $\chi$ is that it seems hard in general to improve the bound $\chi(1)\ll \log A(\chi)$ (one can however do this in the specific case $G=S_n$ and we put this to use in Proposition~\ref{lemma lower bound variance S_n big conj classes}). As for the bound on average, it works quite well for some abelian extensions (see \cite{FiMa}), however there are examples such as Theorem \ref{th:BiasRadical} in which there is a unique non-abelian character of degree comparable to $|G|$, hence the averaging will not succeed in this case.
\end{remark}

\subsection{Proofs of Theorems~\ref{theorem mean variance} and~\ref{theorem least prime ideal}}
We first state and prove Proposition~\ref{proposition asymptotic for the variance}, which implies Theorem~\ref{theorem mean variance}. This will require the following lemma.

\begin{lemma}
\label{lemma large zeros}
Let $L/K$ be a finite Galois extension for which AC holds, and let $\chi$ be an irreducible character of $G=\Gal(L/K)$. For $T\geq 1$, $\beta \in [\tfrac 12,1]$ and $j\in \mathbb Z_{\geq 0}$ we have the estimate
$$ \sum_{|\gamma_{\chi}|>T} \frac {(\log(|\gamma_{\chi}|+4))^j}{\beta+\gamma_{\chi}^2} \ll_j  \frac{(\log (T+4))^j\log(A(\chi)(T+4)^{[K:\Q]\chi(1)})}T,$$
where the sum on the left hand side is over imaginary parts of zeros of $L(s,L/K,\chi)$ and where the implied constant is independent of $\beta$.
\end{lemma}
\begin{proof}
By \eqref{von mangoldt estimation}, we have that
$$
 N(T,\chi)= |\{\gamma_{\chi}\colon |\gamma_{\chi}|\leq T \}|\ll T\log (A(\chi)(T+4)^{[K:\Q]\chi(1)}).
$$
With a summation by parts we obtain that for $j\geq 1$,
\begin{align*}
& \sum_{\substack{|\gamma_{\chi}| > T }} \frac {(\log(|\gamma_{\chi}|+4))^j}{\beta+\gamma_{\chi}^2} \leq \sum_{\substack{|\gamma_{\chi}| > T }} \frac {(\log(|\gamma_{\chi}|+4))^j}{\gamma_{\chi}^2}
 = \int_T^{\infty} \frac{(\log(t+4))^j{\rm d}N(t,\chi)}{t^2} \\&=-  \frac{(\log(T+4))^jN(T,\chi)}{T^2} 
 +  \int_T^{\infty}\Big( j\frac{(\log(t+4))^{j-1}}{t^2(t+4)} -2\frac{(\log(t+4))^j}{t^3}\Big) N(t,\chi){\rm d}t  \\
&\ll j \int_T^{\infty} \frac{ (\log(t+4))^j\log (A(\chi)(t+4)^{[K:\Q]\chi(1)})  }{ t^2} {\rm d}t \\ &\hspace{2cm}+ \frac{(\log(T+4))^j\log (A(\chi)(T+4)^{[K:\Q]\chi(1)})}T.
\end{align*}
The proof follows, and is similar in the case $j=0$.
\end{proof}

In Proposition~\ref{proposition asymptotic for the variance} we will use the  bound $ |\widehat t(\chi)| \leq  \chi(1) \norm t_1$, which follows from the triangle inequality.

\begin{proposition} \label{proposition asymptotic for the variance}
Let $L/K$ be a Galois extension of number fields and let $G=\Gal(L/K)$.
Then for any class function $t\colon G\rightarrow\C$, we have the upper bound

\begin{equation}
 \V[X(L/K;t)] \ll \norm {t}_1^2\log (d_L+2) \min(M_L, \log (d_L+2)), 
 \label{equation first upper bound in lemma multiplicities}
\end{equation}
where 
$$M_{L} := \max \Big \{  {\rm ord}_{s=\rho}  \zeta_L(s)  :\Re(\rho)=\beta_L^t, 0<|\Im(\rho)|< \log (d_L+2)(\log\log (d_L+2))^2\Big\}. $$

If $L/\Q$ is Galois and AC holds, then we have the bound
\begin{equation}
\V[X(L/K;t)] \ll (m^t_{L})^2\sum_{\chi\in \Irr(G^+) }|\widehat{t^+}(\chi)|^2 
\log (A(\chi)+2)\,,
\label{equation second bound proposition multiplicities}
\end{equation}
where 
$$m^t_{L} := \max \Big \{  {\rm ord}_{s=\rho} \Big( \prod_{\substack{\chi \in {\rm supp}(\widehat {t^+}) }} L(s,L/\Q,\chi)\Big)  : \Re(\rho)=\beta_L^t, 0<|\Im(\rho)|<  (T_{L}\log (T_{L}))^2\}, $$
with $T_{L}:=\log({\rm rd}_L+2)\max_{\substack{ \chi \in \Irr(G^+)  }}\chi(1)$.
Assuming moreover that $\beta_{L}^{t}=\frac 12$, and\footnote{We have already seen in Remark~\ref{remark variance small group} that $t^+\equiv 0$ implies that $\V[X(L/K;t)]=0$.} 
$\widehat{t^+}(\chi) \not \equiv 0$,
we have, under  LI$^-$, the lower bound
\begin{equation}
 \V[X(L/K;t)] \gg \sum_{ \chi\in \Irr(G^+) }|\widehat{t^+}(\chi)|^2 
\log (A(\chi)+2)\,.
\label{equation proposition multiplicities last lower bound}
\end{equation}

In the particular case $t=|G|1_{\{{\rm id}\}}$, the lower bound 
$
 \V[X(L/K;|G|1_{\{{\rm id}\}} )]\gg \log (d_L+2)$ holds assuming only the Riemann Hypothesis for $\zeta_L(s)$ (without requiring $L/\Q$ to be Galois). 

\end{proposition}

We recall that if $L/\Q$ is Galois and under AC and LI$^-$, $m^t_{L}\leq 1$ and $M_L\leq  \max_{\substack{\chi \in \Irr(G^+) }}\chi(1). $

\begin{proof}[Proof of Proposition~\ref{proposition asymptotic for the variance}]
We start by establishing~\eqref{equation first upper bound in lemma multiplicities}. Note that for any $s_0\in \C$, $C\in G^\sharp$ and $g_C\in C$,
$$
\Big|\sum_{\chi \in \Irr(\langle g_C\rangle)} \chi(g_C) {\rm ord}_{s=s_0}L(s,L/L^{\langle g_C\rangle},\chi )\Big| \leq {\rm ord}_{s=s_0}\zeta_L(s).$$
We have used the crucial fact that ${\rm ord}_{s=s_0}L(s,L/L^{\langle g_C\rangle},\chi )\geq 0$ (since AC holds for the abelian extension $L/L^{\langle g_C \rangle }$).
Hence, \eqref{equation variance avec abelien -1} implies that
\begin{align*}
\V[X(L/K;t)]&\leq 2\norm {t}^2_1\sums_{\gamma \in Z_{L} }  \frac{({\rm ord}_{s=\beta^t_L+i\gamma}\zeta_L(s))^2}{(\beta^t_{L})^2+\gamma^2}\,,
\end{align*}  
where $Z_L$ is defined in~\eqref{equation definition Z L/K}.
Now, we have the classical unconditional upper bound
\begin{equation}
 {\rm ord}_{s=\rho}\zeta_L(s) \ll \log (d_L(|\Im(\rho)|+4)^{[L:\Q]}) 
 \label{equation bound multiplicities unconditional}
\end{equation}
(see \cite[(5.27)]{IwKo}); we deduce that 
$$ \V[X(L/K;t)] \ll \norm {t}^2_1 (\log (d_L+2))^2+ \norm {t}^2_1|G^+|\log (d_L+2) \ll \norm {t}^2_1(\log (d_L+2))^2, $$
by Lemma~\ref{lemma first bound on Artin conductor}.
To prove~\eqref{equation first upper bound in lemma multiplicities}, we apply 
Lemma~\ref{lemma large zeros} to the trivial extension $L/L$; this takes the form
$$ \sums_{\substack{\gamma \in Z_{L}  \\ |\gamma|>T}}  \frac{{\rm ord}_{s=\beta^t_L+i\gamma}\zeta_L(s)}{(\beta_L^t)^2+\gamma^2}  \ll \frac{\log(d_L(T+4)^{[L:\Q]})}{T}.
$$
 We deduce that for any $T\geq 1,$ 
\begin{align*}
 \sums_{\substack{\gamma \in Z_{L}  \\ |\gamma|>T}}  \frac{({\rm ord}_{s=\beta^t_L+i\gamma}\zeta_L(s))^2}{(\beta^t_L)^2+\gamma^2}  &\ll  \sums_{\substack{\gamma \in Z_{L}  \\ |\gamma|>T}}  \frac{\log (d_L(|\gamma|+4)^{[L:\Q]}){\rm ord}_{s=\beta^t_L+i\gamma}\zeta_L(s)}{\beta_{L}^2+\gamma^2}  \\ &\ll \frac{(\log(d_L(T+4)^{[L:\Q]}))^2}{T}.
\end{align*}
 Moreover, by Lemma~\ref{lemma bounds on B(chi)} (see the comments before this lemma about replacing $\frac 14$ by $(\beta_L^t)^2\in [\frac 14,1 ]$), taking $T= \log (d_L+2) (\log\log(d_L+2))^2$ and applying Lemma~\ref{lemma first bound on Artin conductor},
$$ \sums_{\substack{\gamma \in Z_{L}  \\ |\gamma|\leq T}}  \frac{({\rm ord}_{s=\beta^t_L+i\gamma}\zeta_L(s))^2}{(\beta^t_L)^2+\gamma^2}  \ll M_L \log (d_L+2). 
$$
The upper bound~\eqref{equation first upper bound in lemma multiplicities} follows. Also, under GRH, Proposition \ref{proposition link with random variables} reads
$$\V[X(L/K;|G|1_{\{{\rm id}\}})]\\=2\sums_{\gamma \in Z_{L} }  \frac{({\rm ord}_{s=\beta^t_L+i\gamma}\zeta_L(s))^2}{\tfrac 14+\gamma^2}\geq \sum_{\gamma \in Z_{L} }  \frac{1}{\tfrac 14+\gamma^2} \gg \log (d_L+2).
$$

 We now move to~\eqref{equation second bound proposition multiplicities}. We enumerate the characters $\chi \in \Irr (G^+)  = \{ \chi_1,\chi_2, \cdots , \chi_k\}$ in such a way that for each $1\leq j\leq k-1$,
 $ |\widehat{t^+}(\chi_j)| \geq |\widehat{t^+}(\chi_{j+1})|. $
 Then, by Proposition \ref{proposition link with random variables}, we have that
 \begin{align*}
\V[X(L/K;t)] &=  2\sums_{\gamma \in Z_{L} }  \frac{1}{\beta_{L}^2+\gamma^2} |{\rm ord}_{s=\beta^t_L+i\gamma}L(s,L/\Q,t^+ )|^2 \\
&\leq 2 \sum_{j=1}^k \sums_{ \substack{\gamma \in Z_{L}  \\ L(\beta_{L}+i\gamma,\chi_j)=0 \\ L(\beta_{L}+i\gamma,\chi_{\ell})\neq 0 \text{ for } \ell <j}}  \frac{1}{\beta_{L}^2+\gamma^2} \Big(\sum_{ \chi \in \Irr(G^+)} |\widehat{t^+}(\chi)|{\rm ord}_{s=\beta^t_L+i\gamma}L(s,L/\Q,\chi )\Big)^2 \\
&= V_{\leq T}+V_{>T}
 \end{align*}
where $V_{\leq T}$ denotes the sum over $\gamma\leq T$ and $V_{> T}$ that over $\gamma >T$.
Now, if $L(\beta_L^t+i\gamma,L/\Q,\chi_j)=0$ and  $L(\beta_L^t+i\gamma,L/\Q,\chi_{\ell})\neq 0 \text{ for } \ell <j$, then
$$ \sum_{\chi \in \Irr(G^+)}|\widehat{t^+}(\chi)|{\rm ord}_{s=\beta^t_L+i\gamma}L(s,L/\Q,\chi ) 
 \leq 
 |\widehat{t^+}(\chi_j)|m^t_{L}(\beta_{L}+i\gamma) , $$
where $m^t_{L}(\rho)$ is the order of vanishing of $\prod_{\chi \in {\rm supp}(\widehat{t^+})}L(s,L/\Q,\chi)$ at $s=\rho$. Hence, for any $T>1$ and denoting $$m^t_{L}(T) := \max \Big \{  {\rm ord}_{s=\rho} \Big( \prod_{\substack{\chi \in {\rm supp}(\widehat{t^+}) }} L(s,L/\Q,\chi)\Big)  : \Re(\rho)=\beta_L^t, 0<|\Im(\rho)| \leq  T\}, $$
we have that
\begin{align*}
V_{\leq T}&\ll (m^t_{L}(T))^2 \sum_{j=1}^k |\widehat{t^+}(\chi_j)|^2 \sums_{ \substack{\gamma \in Z_{L}  \\ L(\beta^t_{L}+i\gamma,\chi_j)=0 \\ L(\beta^t_{L}+i\gamma,\chi_{\ell})\neq 0 \text{ for } \ell <j \\ \gamma \leq T}}  \frac{1}{(\beta_L^t)^2+\gamma^2} \\
& \ll  (m^t_{L}(T))^2 \sum_{\chi \in \Irr(G^+) } |\widehat{t^+}(\chi)|^2 \sums_{\gamma_{\chi}} \frac{1}{(\beta_L^t)^2+\gamma_{\chi}^2}
\\
&\ll (m^t_{L}(T))^2 \sum_{\chi \in \Irr(G^+) } |\widehat{t^+}(\chi)|^2\log (A(\chi)+2).
\end{align*}
In a similar fashion and by applying \eqref{equation bound multiplicities unconditional} and Lemma~\ref{lemma large zeros} (with $j=0,2$), we see that
\begin{align*}
V_{> T}\ll\frac 1T \sum_{\chi \in \Irr(G^+)}\chi(1) \log({\rm rd}_L+2)  \sum_{\chi \in \Irr(G^+) } |\widehat{t^+}(\chi)|^2\log (A(\chi)+2),
\end{align*}
and~\eqref{equation second bound proposition multiplicities} follows from taking  
$$ T=\Big( \Big(\log ({\rm rd}_L+2)\max_{\substack{  \chi \in \Irr(G^+)   }} \chi(1)\Big) \cdot \log\Big( \log ({\rm rd}_L+2)\max_{\substack{  \chi \in \Irr(G^+)   }} \chi(1)\Big)\Big)^2 .$$

 Coming back to the general case and assuming $\beta^t_L=\tfrac 12$ and LI$^-$, we see that for $\chi_1\neq \chi_2\in \Irr(G^+)$, the sets of nonreal zeros of $L(s,\chi_1) $ and $L(s,\chi_2)$ are disjoint.  Hence, Proposition \ref{proposition link with random variables} takes the form
 \begin{align*}\V[X(L/K;t)] &= \sum_{\chi\in \Irr(G^+)}
|\widehat {t^+}(\chi)|^2 \sum_{\gamma_{\chi}\neq 0} \frac {1}{\frac 14+\gamma_{\chi}^2}.
 \end{align*}
 The lower bound~\eqref{equation proposition multiplicities last lower bound} follows once more from applying Lemma~\ref{lemma bounds on B(chi)}.
\end{proof}

Combining Proposition~\ref{proposition asymptotic for the variance} and  Lemma \ref{lemma finer bounds on Artin conductor} will allow us in some cases to determine the exact order of magnitude of the variance of the random variable $X(L/\Q; t)$ in terms of the absolute discriminant of the number field $L$, independently of the individual Artin conductors. For this to be possible, the characters of the associated Galois group of high degree must have the property that $|\chi(C)|$ is significantly smaller than $\chi(1)$ for all conjugacy classes $C\neq \{{\rm id}\}$. We illustrate this with the following proposition. Recalling the definition~\eqref{equation definition norms},
 we note that $ \norm {t^+}_1\leq \norm {t^+}_2$, by Cauchy--Schwarz. However, in the case $t^+=\sum_{i\leq k} |G^+| |C^+_i|^{-1}1_{C^+_i}$ where $C^+_1,...,C^+_i\in (G^+)^\sharp$ are distinct, we have that  $\norm{t^+}_1= 1$ and $\norm{t^+}^2_2 =|G^+|\sum_{i\leq k} |C^+_i|^{-1}\geq |G^+|\min(|C^+_i|)^{-1}$, that is $\norm{t^+}_2$ is significantly larger than $\norm{t^+}_1$.

\begin{proposition}
\label{proposition lower bound variance}
Let $L/K$ be an extension of number fields such that $L/\Q$ is Galois, and for which AC, GRH and LI$^-$ hold. 
Then we have the bound
$$ \V[X(L/K;t)] \gg \eta_{L/K;t} [K:\mathbb{Q}]\log ({\rm rd}_L+2) \sum_{\chi \in \Irr(G^+)} \chi(1)|\widehat t(\chi)|^2,$$
where 
$$ \eta_{L/K;t} := 1-\max_{\substack{ \chi \in C_{L/K;t} }}\max_{{\rm id}\neq g\in G} \frac{|\chi(g)|}{\chi(1)}\geq 0,
$$
with 
$$ C_{L/K;t}:= \{ \chi \in {\rm supp}(\widehat {t^+}) :  \chi(1)\geq \norm{t^+}_2 \norm{t^+}_1^{-1} (4\#{\rm supp}(\widehat{t^+}))^{-\frac 12}\}.$$
\end{proposition}

\begin{proof}
We first establish a preliminary bound. Defining $N:=\norm{t^+}_2\norm{t^+}_1^{-1}(2\#{\rm supp}(\widehat{t^+}))^{-\frac 12}$ and applying Lemmas~\ref{lemma:sumsquareorthogonality} and~\ref{lemma pointwise bound Fourier transform}, we see that
\begin{align}
\notag 
\sum_{\substack{\chi \in \Irr(G^+) \\ \chi(1)>N }} \chi(1)|\widehat {t^+}(\chi)|^2  &\geq N\Big(\sum_{\substack{\chi \in \Irr(G^+)  }} |\widehat {t^+}(\chi)|^2 - \sum_{\substack{\chi \in \Irr(G^+) \\ \chi(1)\leq N }} |\widehat {t^+}(\chi)|^2 \Big)
\\ &\geq   N (\norm{t^+}_2^2 -N^2 \norm{t^+}_1^2 \#{\rm supp}(\widehat{t^+}) ) \notag \\ & =N^3\norm{t^+}_1^2 \#{\rm supp}(\widehat{t^+})  \geq \sum_{\substack{\chi \in \Irr(G^+) \\ \chi(1)\leq N }} \chi(1)|\widehat {t^+}(\chi)|^2.
   \label{equation lower bound variance first step}
\end{align}

Hence, applying Proposition~\ref{proposition asymptotic for the variance} and Lemma~\ref{lemma finer bounds on Artin conductor}\footnote{Note that this lemma implies the bounds $ (1-M_{\chi}) [K:\Q] \chi(1) \log({\rm rd}_L+2) \ll \log (A(\chi)+2) \ll (1+M_{\chi}) [K:\Q] \chi(1) \log({\rm rd}_L+2)$ },
\begin{align*}
\V[X(L/K;t)] &\gg \eta_{L/K;t} [K:\Q]\log ({\rm rd}_L+2) \sum_{\substack{\chi \in \Irr(G^+) \\ \chi(1)> N }} \chi(1)|\widehat {t^+}(\chi)|^2\\
&\geq \frac{\eta_{L/K;t} [K:\Q]\log ({\rm rd}_L+2)}2 \sum_{\substack{\chi \in \Irr(G^+) }} \chi(1)|\widehat {t^+}(\chi)|^2,
\end{align*} 
by~\eqref{equation lower bound variance first step}.
\end{proof}

\begin{proof}[Proof of Theorem \ref{theorem mean variance}]
Combine Propositions~\ref{proposition link with random variables},~\ref{proposition asymptotic for the variance} and~\ref{proposition lower bound variance}. The lower bound on the character sum will be proven in Lemma~\ref{lemma lower bound on variance in terms of group}, and the upper bound follows from the bound $\chi(1)\leq |G^+|^{\frac 12}$.
\end{proof}

\begin{proof}[Proofs of Theorem~\ref{theorem least prime ideal} and~\ref{theorem chebotarev all x}]
We begin by proving~\eqref{equation theorem least ideal second bound}.
Consider the following weighted variant of $\psi(x;L/K,t)$, where $h$ is a nonnegative, not identically zero smooth function supported in $[1,\tfrac 32]$ and $t:G\rightarrow \R$ is a class function:
$$\psi_{h}(x;L/K,t) :=\sum_{ \substack{ \mathfrak p \triangleleft \mathcal O_K \\ k\geq 1}} t(\varphi_{\mathfrak p}^k) h(\mathcal N \mathfrak p^k / x)\log(\mathcal N \mathfrak p), $$
which decomposes as 
$$ \psi_{h}(x;L/K,t)= \sum_{\chi \in \Irr(G)}\overline{\widehat{t}(\chi)} \psi_h(x;L/K,\chi).$$
Integration by parts shows that for $|s|\geq 2$, the Mellin transform 
$$ \mathcal Mh(s):= \int_0^\infty x^{s-1} h(x)dx$$
satisfies the bound 
\begin{equation}
\mathcal Mh(s) \ll \frac 1{|s|^2} \int_0^\infty x^{\Re(s)+1} |h''(x)|dx\ll_h \frac 1{|s|^2} .
\label{equation bound mellin transform h}
\end{equation} 
It follows from~\cite[Theorem 5.11]{IwKo} that for any irreducible character $\chi \in \Irr(G)$ and for $x\geq 1$,
\begin{align*}
\psi_{h}(x;L/K,\chi) - x\mathcal Mh(1) \delta_{\chi = 1}  &\ll c_h \log (A(\chi)+2) + c_h x^{\frac 12} \chi(1)  + \sum_{\rho_\chi} |\mathcal Mh(\rho_{\chi})x^{\rho_\chi}| \\
& \ll c_h x^{\frac 12} \chi(1)[K:\Q]\log ({\rm rd}_L+2), 
\end{align*}
by Lemma~\ref{lemma first bound on Artin conductor} and Stirling's formula, where $c_h := 1+\sup (|h|+|h''|)$. 
 We deduce that
\begin{align}
 \psi_{h}(x;L/K,t)
 &=  x\mathcal Mh(1)\overline{\widehat{t}(1)} +O(c_h\lambda(t) [K:\Q] x^{\frac 12} \log({\rm rd}_L+2)).
 \label{equation error term murty proof} 
\end{align}
 Applying inclusion-exclusion, we see that
\begin{align}
\theta_{h}(x;L/K,t):= \sum_{ \substack{ \mathfrak p \triangleleft \mathcal O_K }} t(\varphi_{\mathfrak p}) h( \mathcal N\mathfrak p/ x) \log (\mathcal N\mathfrak p) = \sum_{\ell \geq 1} \mu(\ell) \psi_{h(\cdot^\ell)}(x^{\frac 1\ell};L/K,t(\cdot^\ell)).
\label{equation inclusion exclusion}
\end{align}
For the terms with $\ell \geq 2$, we use the Fourier decomposition of $t$ and deduce that
\begin{multline}
|\psi_{h(\cdot^\ell)}(x^{\frac 1\ell};L/K,t(\cdot^\ell))| = \Big|\sum_{\chi\in \Irr(G)} \overline{\widehat t(\chi)} \psi_{h(\cdot^\ell)}(x^{\frac 1\ell};L/K,\chi(\cdot^\ell))\Big|\leq  \lambda(t) \psi_{h(\cdot^\ell)}(x^{\frac 1\ell};L/K,1).
\label{equation removal higher powers fourier decomp}
\end{multline}
Now, the explicit formula for $\zeta_K(s)$ (see for instance~\cite[Theorem 5.11]{IwKo}) implies that for $x\geq 1$,
\begin{align*}
\psi_{h(\cdot^\ell)}(x^{\frac 1\ell};L/K,1) &= \sum_{\substack{\mathfrak p \triangleleft \mathcal O_K \\ k\geq 1}} h((\mathcal N\mathfrak p^{k }/x^{\frac 1{\ell}})^\ell) \log \mathcal N\mathfrak p^k \\& =\mathcal M \{h(\cdot ^{\ell})\}(1)\cdot x^{\frac 1\ell}+\sum_{\rho_K} x^{\frac{\rho_K}\ell}\mathcal M \{h(\cdot ^{\ell})\}(\rho_K)+O_h([K:\Q] x^{\frac 1{2\ell}}+\log ( d_K+2)) \\ 
& \ll_h x^{\frac 1\ell}+ x^{\frac 1{2\ell}} \ell \log(d_K+2),
\end{align*} 
by the identity $M \{h(\cdot ^{\ell})\}(s) = \frac 1\ell \mathcal  M h(\frac s\ell) $ and the bound~\eqref{equation bound mellin transform h}. Hence, noting that the support condition on $h$ implies that $\psi_{h(\cdot^\ell)}(x^{\frac 1\ell};L/K,1) = 0 $ for $\ell \geq 3\log x$, we obtain the estimate
\begin{equation}
\theta_{h}(x;L/K,t)=\psi_{h}(x;L/K,t)+O( \lambda(t)(x^{\frac 12}+ x^{\frac 14}\log(d_K+2)) ). \label{equation getting rid of higher prime powers}
\end{equation} 

Combining this identity with~\eqref{equation error term murty proof} and~\eqref{equation conductor discriminant}, we deduce that
$$
\theta_{h}(x;L/K,t) -  x\mathcal Mh(1)\overline{\widehat{t}(1)}   \ll  c_h\lambda(t) x^{\frac 12}[K:\Q]  \log ( { \rm rd}_L+2 ).
$$
Generalizing \cite[Proposition 6]{SeIHES}, we see that
$$ \sum_{\mathfrak p \mid  D_{L/K}} \log (\mathcal N \mathfrak p) \leq \frac{2}{|G|} \log \mathcal N_{K/\Q}(D_{L/K} )\leq 2[K:\Q] \log({\rm rd}_L+2)\,,$$
where the first inequality follows from combining~\eqref{equation conductor disc relative} with~\eqref{equation valuation p de chireg} in which for each $\mathfrak p\mid D_{L/K}$, $\sum_{i\geq 0} (|G_i|-1)\geq |G_0|/2$, and the second follows from~\eqref{equation conductor discriminant}. Hence, the contribution of ramified primes in $\theta_h(x;L/K,t)$ is
$$\ll  \lambda(t)\sum_{\substack{ \mathfrak p  \mid D_{L/K} }}  h( \mathcal N \mathfrak p/x) \log (\mathcal N \mathfrak p) \ll c_h \lambda(t) [K:\Q] \log({\rm rd}_L+2),$$
since for any $g\in G$, $|t(g)| \leq \sum_{\chi \in \Irr(G)} \chi(1)|\widehat{t}(\chi)|= \lambda(t)$. 
We conclude that
\begin{equation}
\sum_{\substack{ \mathfrak p \triangleleft\mathcal O_K \\  \mathfrak p \text{ unram.}  }} t(\varphi_\mathfrak p)h(\mathcal N\mathfrak p /x) \log (\mathcal N \mathfrak p)=  x\mathcal Mh(1)\overline{\widehat{t}(1)} +O_h(  \lambda(t)x^{\frac 12}[K:\Q]  \log ( { \rm rd}_L+2 )    ).
\label{equation just before chosing x}
\end{equation}
Taking 
$$x = K_h\Big( \frac{\log({\rm rd}_L+2) \lambda(t)[K:\Q]}{  \widehat t(1)}\Big)^2 $$ for a large enough positive constant $K_h$ in~\eqref{equation just before chosing x} (recall that $t$ is real valued and $\widehat t(1)> 0$; note also that $\lambda(t)\geq |\widehat t(1)|$, so that $x\geq K_h$), we deduce that 
$$\sum_{\substack{ \mathfrak p \triangleleft\mathcal O_K \\  \mathfrak p \text{ unram.}  }} t(\varphi_\mathfrak p)h(\mathcal N\mathfrak p /x) \log (\mathcal N \mathfrak p)> 0.$$ Since $h$ is supported in $[1,\tfrac 32]$, it follows that there exists a 
prime ideal $\mathfrak p\triangleleft \mathcal O_K$ of norm $\leq  \tfrac 32x$ and such that $t(\varphi_\mathfrak p) > 0$, that is~\eqref{equation theorem least ideal second bound} holds.

We now move to the bound~\eqref{equation G+ improvement least ideal simple}.
Arguing as in Proposition~\ref{proposition psi bridge}, we have that
\begin{equation}\psi_{h}(x;L/K,t)= \psi_{h}(x;L/\Q,t^+) = \sum_{\chi \in \Irr(G^+)}\overline{\widehat{t^+}(\chi)} \psi_h(x;L/\Q,\chi).
\label{equation transfer with smooth weight}
\end{equation}
It follows from~\cite[Theorem 5.11]{IwKo} that for any irreducible character $\chi \in \Irr(G^+)$ and for $x\geq 1$,
\begin{align}\notag
\psi_{h}(x;L/\Q,\chi) - x\mathcal Mh(1) \delta_{\chi = 1}  &\ll c_h \log A(\chi) + c_h x^{\frac 12} \chi(1)  + \sum_{\rho_\chi} |\mathcal Mh(\rho_{\chi})x^{\rho_\chi}| \\
& \ll c_h x^{\frac 12} \chi(1)\log ({\rm rd}_L+2). 
\label{equation explicit formula bounded}
\end{align}
 We deduce that (note that $\widehat t(1)=\widehat {t^+}(1)$)
\begin{align}
 \psi_{h}(x;L/\Q,t^+)
 &=  x\mathcal Mh(1) \overline{\widehat{t^+}(1)} +O(c_h\lambda(t^+) x^{\frac 12} \log({\rm rd}_L+2)).
 \label{equation error term murty proof+} 
\end{align}
Now, we may combine this with~\eqref{equation getting rid of higher prime powers} and~\eqref{equation transfer with smooth weight}, resulting in the bound 
$$
\theta_{h}(x;L/K,t) -  x\mathcal Mh(1)\overline{\widehat {t^+}(1)}  
\ll c_h \lambda(t^+)x^{\frac 12} \log ( { \rm rd}_L+2 )+O(\lambda(t)(x^{\frac 12}+x^{\frac 14} \log(d_K+2))).
$$
The ramified primes are bounded in the same way as before, and contribute an error term $\ll \lambda(t) [K:\Q] \log({\rm rd}_L+2)$. We deduce that  as soon as $\widehat {t^+}(1)>0$ and $$x>K_h\Big(\frac{\lambda(t^+)\log ({\rm rd}_L+2)+\lambda(t)}{\widehat {t^+}(1)}\Big)^2+K_h \Big( \frac{\lambda(t)\log(d_K+2)}{\widehat {t^+}(1)} \Big)^{\frac 43}$$ (note once more that $x\geq K_h$)
for some large enough $K_h>0$, there exists an unramified prime ideal $\mathfrak p\triangleleft \mathcal O_K$ such that $t(\frob_\mathfrak p)>0$ and $\mathcal N\mathfrak p \leq x$.
Taking $t= |G||C|^{-1}1_C$, we recall that $t^+=|G^+||C^+|1_{C^+}$ and $\widehat {t^+}(1)=1 $, thus
$$ \lambda(t)\leq \frac{|G|}{|C|^{\frac 12}},\qquad \lambda(t^+)\leq \frac{|G^+|}{|C^+|^{\frac 12}}.$$
The claimed bound~\eqref{equation theorem least ideal second bound} follows from noting that $|G|/|C|^{\frac 12}\leq |G^+|/|C^+|^{\frac 12}. $

To prove~\eqref{equation improved bound least ideal with G+ }, we apply~\eqref{equation transfer with smooth weight} to~\eqref{equation inclusion exclusion} and deduce that \begin{align}
\theta_{h}(x;L/K,t)
= \sum_{\ell \geq 1} \mu(\ell) \psi_{h(\cdot^\ell)}(x^{\frac 1\ell};L/\Q,(t(\cdot^\ell))^+).
\end{align}
Note moreover that 
$$\overline{\widehat{(t(\cdot^\ell))^+}(1)}= \frac 1{|G^+|}\sum_{g\in G^+} (t(\cdot^\ell))^+(g) =\frac 1{|G^+|}\sum_{aG\in G^+/G} \sum_{\substack{ g \in aGa^{-1}}} t((a^{-1}g a)^\ell)=\langle t,r_\ell\rangle_G.$$
Hence,~\eqref{equation explicit formula bounded} implies that
\begin{align*}
 \theta_{h}(x;L/K,t)&-x\mathcal Mh(1)\overline{\widehat {t^+}(1)} \\
 & \ll c_h\lambda(t^+)  x^{\frac 12} \log ( { \rm rd}_L+2 )+c_h \sum_{\substack{2\leq \ell \leq 4\log x \\ \mu^2(\ell)=1}}(x^{\frac 1\ell} |\langle  t,r_\ell\rangle_G|+\lambda(t(\cdot^\ell)^+)  x^{\frac 1{2\ell}} 	\log ( { \rm rd}_L+2 )), 
\end{align*}
Taking 
\begin{multline*}x =K_h  \Big( \frac{\lambda(t^+)}{\widehat {t}(1)}\log({\rm rd}_L+2)  \Big)^2 + K_h\frac{\lambda(t)}{\widehat t(1)}[K:\Q] \log({\rm rd}_L+2) \\+K_h\sum_{\substack{2\leq \ell \ll \log\log d_L  \\ \mu^2(\ell)=1}} \Big(\Big( \frac{|\langle t,r_\ell \rangle_G|}{\widehat t(1)} \Big)^{\frac{\ell}{\ell-1}}+\Big(\frac{\lambda((t(\cdot^\ell) )^+)}{\widehat t(1)}\log({\rm rd}_L+2) \Big)^{\frac{2\ell}{2\ell-1}}\Big)
\end{multline*} for a large enough constant $K_h$ in~\eqref{equation just before chosing x}, we see that
$ \log x\ll \log\log (d_L+2)$ (recall that $t$ is real valued and $\widehat t(1)>|G|^{-100}\sup|t|$), and the conclusion follows.
 The proof of~\eqref{equation bound improved chebotarev } goes along similar lines. Finally, we note that~\eqref{eq:parseval} implies that,
$$|\langle t,r_\ell\rangle_G|=|\langle \widehat t,\widehat{r_\ell} \rangle_{\Irr(G)}|\leq \lambda(t), $$
and moreover by the Cauchy--Schwarz inequality we obtain the bound
$$ \lambda(t(\cdot^\ell))  \leq |G|^{\frac 12} \norm{\widehat{t(\cdot^\ell)}}_2 =|G|^{\frac 12} \norm{t(\cdot^\ell)}_2   \leq |G|^{\frac 12} \sup|t|.$$
\end{proof}

\section{Probabilistic bounds}

\label{section central limit theorem}

In this section we fix a Galois extension of number fields $L/K$, define $G:=\Gal(L/K)$ (as well as $G^+:=\Gal(L/\Q)$ in the case where $L/\Q$ is Galois), and study the distribution of the random variable $X(L/K;t)$ attached to a real-valued class function $t\colon G\rightarrow\R$ (see Proposition~\ref{proposition link with random variables} and Lemma~\ref{lemma random var under LI}), using probabilistic tools. Our main goal is to estimate $\delta(L/K;t)$, which measures to which extent the error term in the Chebotarev density theorem is biased by a lower-order term of constant sign.
We first consider the conditions under which $\delta(L/K;t)$ (see~\eqref{eq:densitiest}) is close to $1$. This leads to estimates for the bias under AC, GRH and BM. Stronger bounds can be derived under LI: as we will see, large deviations results of Montgomery--Odlyzko can then be applied to this context.
Next we establish a central limit Theorem 
from which we obtain (conditionally on LI; here BM does not suffice) conditions under which $\delta(L/K;t)$ are close to $\tfrac 12$. In both cases we highlight the importance of the ratio\footnote{This is the inverse of the so-called coefficient of variation. 
When $\V[X(L/K;t)]=0$ and ${\rm sgn}(\E[X(L/K;t)])=\pm 1$, we define $B(L/K;t)$ to be $\pm\infty$; we do not define it when $\E[X(L/K;t)]=\V[X(L/K;t)]=0$.
Note also that if $L/\Q$ is Galois and assuming GRH$^-$, LI$^-$,
the condition $ \widehat{t^+} \not \equiv 0$ (recall~\eqref{equation definition t+}),
implies that $\V[X(L/K;t)]>0$. Moreover, by Proposition~\ref{proposition link with random variables} and Corollary~\ref{corollary when t+ is 0}, GRH$^-$ and LI$^-$ imply that if $\V[X(L/K;t)]=0$ and $K/\Q$ is Galois, then we also have $\E[X(L/K;t)]=0$.}
\begin{align}
 B(L/K;t) &:= \frac{\E[X(L/K;t)]}{\V[X(L/K;t)]^{\frac 12}} 
 \label{equation definition bias factors}
\end{align}

 This parameter governs the behaviour of the corresponding random variable according to the following philosophy: if it is small, then the random variable is only moderately biased, whereas if it is large, then the random variable is highly biased.

\subsection{Large deviations}
We first establish bounds on $\delta(L/K;t)$ in terms of the bias factors which hold under AC and GRH. These bounds will later be applied in conjunction with upper bounds on the bias factors holding under BM. Note that AC, GRH and BM do not suffice to prove the existence of the density $\delta(L/K;t)$ 
and thus the statement only gives information about the lower densities. Under the additional assumption LI the densities exist and sharper bounds can be deduced, as we will see in Proposition~\ref{theorem asymptotic formula for highly biased races}. 

\begin{proposition}
\label{proposition asymptotic formula for highly biased races}
Let $L/K$ be an extension of number fields for which $L/\Q$ is Galois, and for which AC, GRH and BM hold. Let $t:\Gal(L/K) \rightarrow \R$ be a class function.
If $B(L/K;t)$ is positive and large enough and $\E[X(L/K;t)]\geq 4$, then
$$ \underline{\delta}(L/K;t) \geq  1- 2B(L/K;t)^{-2}. $$

\end{proposition}
\begin{proof}
The proof is very similar to that of \cite[Lemma 2.7]{Fi2} and uses Chebyshev's inequality (see also \cite[Corollary 5.8]{De}). 
\end{proof}

The key to a more precise estimation of the bias under LI will be the following theorem of Montgomery and Odlyzko on large deviations of sums of independent random variables.

\begin{theorem}[{\cite[Theorem 2]{MoOd}}]
\label{theorem mood}
For $n\in\Z_{\geq 1}$ let $W_n$ be independent real valued random variables such that $\E[W_n]=0$ and $|W_n|\leq 1$; let also $r_n$ be a decreasing sequence of real numbers tending to zero. Suppose that there is a constant $c>0$ such that $\E[W_n^2]\geq c$ for all $n$. Put $W=\sum r_n W_n$ where $\sum r_n^2 <\infty$. Let $V$ be a nonnegative real number. 

If $\sum_{|r_n|\geq \alpha} |r_n| \leq V/2$ then
\begin{equation}
 \P[W\geq V] \leq \exp\bigg( -\frac 1{16} V^2 \bigg( \sum_{|r_n|<\alpha} r_n^2 \bigg)^{-1}\bigg). 
 \label{eq:MoOd1 upper bound}
\end{equation}

If $\sum_{|r_n|\geq \alpha} |r_n| \geq 2V$ then
\begin{equation}
 \P[W\geq V] \geq a_1 \exp\bigg( -a_2 V^2 \bigg( \sum_{|r_n|<\alpha} r_n^2 \bigg)^{-1}\bigg). 
\label{eq:MoOd2 lower bound}
\end{equation}
Here $a_1>0$ and $a_2>0$ depend only on $c$.
\end{theorem}

Note that a more precise result (in which $c_3=c_2+o(1)$) could possibly be obtained using the saddle-point method as in \cite{Mo} (see also \cite{La1}), however this would not affect our main theorems since we are only able to evaluate $B(L/K;t)$ up to a constant. We can deduce the following result concerning high biases.

\begin{proposition}
\label{theorem asymptotic formula for highly biased races}
Let $L/K$ be an extension of number fields such that $L/\Q$ is Galois, and for which AC, GRH and LI hold. Let $t:\Gal(L/K) \rightarrow \R$ be a class function,
 and let $B(L/K;t)$ be defined as in~\eqref{equation definition bias factors}.

\begin{enumerate}

\item
If $\E[X(L/K;t)]\geq 0$, then
$$ \delta(L/K;t)>1- \exp(-c_1 B(L/K;t)^2).$$

\item
If in addition $K=\Q$ (so that $t=t^+$ (recall~\eqref{equation definition t+}) and $G=G^+$), $t \not \equiv 0$ and $\widehat{t}(\chi)\in \{ 0,1,-1\}$ for every $\chi\in \Irr(G)$, 
then we also have the upper bound
$$\delta(L/\Q;t)<1- c_2 \exp(-c_3 B(L/\Q;t)^2). $$

\end{enumerate}
Here, the $c_i$ are positive absolute constants.

\end{proposition}

Besides Theorem~\ref{theorem mood}, the main ingredient for the proof of 
Proposition~\ref{theorem asymptotic formula for highly biased races} is the following 
estimate.
\begin{lemma}
Let $L/K$ be a finite Galois extension for which AC holds, and let $\chi$ be an irreducible character of $G=\Gal(L/K)$.
For $T\geq 1$ we have the estimate
$$ \sum_{\substack{|\gamma_{\chi}| < T }} \frac 1{\big(\frac 14+\gamma_{\chi}^2\big)^{\frac 12}} = \frac 1{\pi}\log \Big(A(\chi) \Big(\frac{T^{\frac 12}}{2\pi e}\Big)^{[K:\Q]\chi(1)}\Big) \log T + O\bigl(\log (A(\chi)(T+4)^{[K:\Q]\chi(1)})\bigr)\,.$$
\label{lemma divergent sum over zeros}
\end{lemma}

\begin{proof}
We start from \eqref{von mangoldt estimation}:
$$
 N(T,\chi)= |\{\gamma_{\chi}\colon |\gamma_{\chi}|\leq T \}|=\frac T{\pi} \log \frac{A(\chi)T^{[K:\Q]\chi(1)}}{(2\pi e)^{[K:\Q]\chi(1)}} + O(\log ((A(\chi)+2)(T+4)^{[K:\Q]\chi(1)}))\,.
$$
With a summation by parts we obtain that
\begin{align*}
 \sum_{\substack{|\gamma_{\chi}| < T }} \frac 1{\big(\frac 14+\gamma_{\chi}^2\big)^{\frac 12}} =&\sum_{\substack{1<|\gamma_{\chi}| < T }} \frac 1{|\gamma_{\chi}|} +O(\log (A(\chi)5^{[K:\Q]\chi(1)}) 
 = O(\log (A(\chi)5^{[K:\Q]\chi(1)}))\\
 &+\int_1^{T} \frac{{\rm d}N(t,\chi)}{t} 
= \frac{N(T,\chi)}{T}+\int_1^T \frac{N(t,\chi)}{t^2}{\rm d}t + O(\log (A(\chi)5^{[K:\Q]\chi(1)})) \\
=& \int_1^T \frac{ \log \frac{A(\chi)t^{[K:\Q]\chi(1)}}{(2\pi e)^{[K:\Q]\chi(1)}}  }{\pi t} {\rm d}t + O\bigl(\log (A(\chi)(T+4)^{[K:\Q]\chi(1)})\bigr)\\
=& \frac 1{\pi}\log \Big(A(\chi) \Big(\frac{T^{\frac 12}}{2\pi e}\Big)^{[K:\Q]\chi(1)}\Big) \log T +O\bigl(\log (A(\chi)(T+4)^{[K:\Q]\chi(1)})\bigr)\,.
\end{align*}
The proof is complete.
\end{proof}

\begin{proof}[Proof of Proposition~\ref{theorem asymptotic formula for highly biased races}]
Let us start with (1). 
We will apply Theorem \ref{theorem mood} to the random variable 
$$
W:=X(L/K;t)-\E[X(L/K;t)] = \sum_{\chi \in { \rm supp}(\widehat{t^+})}  |\widehat{t^+}(\chi)|\sum_{\gamma_{\chi}> 0} \frac{2X_{\gamma_{\chi}}}{\big(\frac 14+\gamma_{\chi}^2\big)^{\frac 12}}
$$ (recall Lemma~\ref{lemma random var under LI}). By Proposition~\ref{proposition link with random variables}, we have that
$$\V[X(L/K;t)] = 2 \sum_{\chi \in { \rm supp}(\widehat{t^+})}  |\widehat{t^+}(\chi)|^2\sum_{\gamma_{\chi}> 0} \frac{1}{\frac 14+\gamma_{\chi}^2}.  $$
 
Taking the sequence $\{r_n\}_{n\geq 1}$ to be the values $2|\widehat{t^+}(\chi)|(\frac 14+\gamma_{\chi}^2)^{-\frac 12}$ ordered by size with $\gamma_{\chi}$ ranging over the imaginary parts of nonreal zeros of $L(s,L/\Q,\chi)$ with $  \chi \in { \rm supp}(\widehat {t^+})$, we have for $\alpha \in (0,4]$ that
$$
 \sum_{|r_n|\geq \alpha}|r_n| = \sum_{\substack{ \chi \in { \rm supp}(\widehat{t^+})} } |\widehat{t^+}(\chi)| \sum_{\substack{ 0<\gamma_{\chi} \leq \sqrt{4\alpha^{-2}-\frac 14}}} \frac 2{\big(\frac 14+\gamma_{\chi}^2\big)^{\frac 12}}\,; $$
 $$ \sum_{|r_n|<\alpha}|r_n|^2= \sum_{ \chi \in { \rm supp}(\widehat{t^+}) }  |\widehat{t^+}(\chi)|^2\sum_{\substack{  \gamma_{\chi} > \sqrt{4\alpha^{-2}-\frac 14}}} \frac 4{\frac 14+\gamma_{\chi}^2}\,.
$$
We take $\alpha=4$: then we trivially have $\sum_{|r_n|\geq \alpha}|r_n|\leq \E[X(L/K;t)]/2$ (note that $\E[X(L/K;t)]\geq 0$ by our assumptions), and hence applying Proposition~\ref{proposition link with random variables}, \eqref{eq:MoOd1 upper bound} translates to
$$ \P[W\geq \E[X(L/K;t)]] \leq \exp \left(-\frac 1{16} \E[X(L/K;t)]^2 \left(2 \V[X(L/K;t)] \right)^{-1}  \right). $$
Thus (1) follows, since by symmetry of $W$ we have that for any $\chi \in \Irr (G^+)$,
$$ 1-\delta(L/K;t) = \P[W<-\E[X(L/K;t)]] = \P[W>\E[X(L/K;t)]].$$
For (2), we use the assumptions that $K=\Q$ and that for $\chi \in {\rm supp}(\widehat t)\neq \varnothing$, $ |\widehat{t}(\chi)|=1$. We let $\alpha\in (0,4]$ be small enough so that $T_0:=\sqrt{4\alpha^{-2}-1/4}$  has the property that for any $\chi \in \Irr(G)$,
$$
\sum_{\substack{0<|\gamma_{\chi}| < T_0 }} \frac 1{\big(\frac 14+\gamma_{\chi}^2\big)^{\frac 12}}\geq 6M_0
$$ 
($M_0$ comes from Hypothesis LI).
Such a number $T_0$, independent of $L/\Q$, exists in light of Lemma \ref{lemma divergent sum over zeros} and by~\eqref{equation bound multiplicities unconditional} applied with $\gamma = 0$. 
Hence, since $|\widehat{t}(\chi)|=|\widehat{t}(\overline{\chi})| $, by the symmetry of the zeros of $L(s,L/\Q,\chi)$ we have that 
$$
\sum_{\chi\in { \rm supp}(\widehat{t^+}) }  \sum_{\substack{0<\gamma_{\chi} < T_0 }} \frac 2{\big(\frac 14+\gamma_{\chi}^2\big)^{\frac 12}} \geq 6
M_0
\#{ \rm supp}(\widehat {t^+})
\geq 2\E[X(L/\Q;t)]\,,
$$
by~\eqref{equation expectancy with characters}.
 Therefore we can apply Theorem \ref{theorem mood} which gives the bound
\begin{align*}
\P[W\geq \E[X(L/\Q;\res)]] &\geq a_1 \exp \Big(-a_2 \E[X(L/\Q;t)]^2 \Big(\sum_{\chi \in \xi_t } \sum_{\gamma_{\chi}>T_0}\frac 4{\frac 14+\gamma^2}\Big)^{-1}  \Big)
\\ &\geq a_1 \exp \Big(-a_2 \E[X(L/\Q;t)]^2 \Big(a_3 \V[X(L/\Q;t)]\Big)^{-1}  \Big),  
\end{align*}
by Proposition~\ref{proposition link with random variables}; this proves the desired upper bound. 
\end{proof}

\subsection{Effective central limit theorem}\label{subsec:CLT}
As in the previous paragraph, $L/K$ denotes an extension of number fields such that $L/\Q$ is Galois. Let $G=\Gal(L/K)$ and $G^+=\Gal(L/\Q)$. For any real-valued class function $t\colon G\rightarrow \R$, 
we first prove a preliminary result on the ``fourth moment'' of $X(L/K;t)$. 
We define the following useful quantity\footnote{If $t^+\not \equiv 0$, then the denominator is clearly positive by Lemma \ref{lemma first bound on Artin conductor}.} attached to $t$.
\begin{equation}
\label{equation definition of W_4}
W_4(L/K;t):=\frac{\sum_{\chi \in \Irr(G^+)}|\widehat{t^+}(\chi)|^4 \log (A(\chi)+2)}{\left(\sum_{\chi \in \Irr(G^+)} |\widehat{t^+}(\chi)|^2 \log (A(\chi)+2)\right)^2}\,.
\end{equation}  

\begin{lemma}
Let $L/K$ be a an number field extension such that $L/\Q$ is Galois. Assume that AC holds, and write $G=\Gal(L/K), G^+=\Gal(L/\Q)$. 
If $t\colon G \rightarrow \R$ is a class function, then
\label{lemma fourth moment}
\begin{equation}\label{equation bound on fourth moment}
W_4(L/K;t) \ll\norm{t^+}_1^{\frac 23}\Big(\sum_{\chi\in\Irr(G^+)} |\widehat{t^+}(\chi)|^2 \log (A(\chi)+2)\Big)^{-\frac 13}\,.
\end{equation}
Moreover, for any finite group $\Gamma$ and any class function $\tau\colon \Gamma \rightarrow \C$,
\begin{equation}
\label{equation bound on fourth moment characters only}
\frac{\sum_{\chi \in \Irr(\Gamma)}\chi(1)|\widehat{\tau}(\chi)|^4 }{\left(\sum_{\chi \in \Irr(\Gamma)} \chi(1)|\widehat{\tau}(\chi)|^2 \right)^2} \ll \norm{\tau}_1^{\tfrac{2}{3}} \Big(\sum_{\chi\in\Irr(\Gamma)} \chi(1)|\widehat{\tau}(\chi)|^2 \Big)^{-\frac 13}.
\end{equation}
\end{lemma}

\begin{remark}
\label{remark better bound on W4}
If $G^+$ and $\Gamma$ are abelian, then the exponent $-\frac 13$ in \eqref{equation bound on fourth moment} and \eqref{equation bound on fourth moment characters only} can trivially be improved to $-1$.
More generally, if $\widehat{t^+}\not \equiv 0$, then
$$ W_4(L/\Q;t) \leq  \Big(\sum_{\chi\in\Irr(G^+)} |\widehat{t^+}(\chi)|^2 \log (A(\chi)+2)\Big)^{-1}\norm{t^+}_1^2\max_{\chi \in \Irr(G^+)}\chi(1)^2.$$
 However, in the case $G^+=S_n$, it cannot be improved\footnote{Indeed from an analysis analogous to the one developed in Section \ref{section:ProofsSn}, one can see that selecting for instance $C_1=\{{\rm id}\}$ and $C_2$ to be the set of $n$-cycles, both sides of \eqref{equation bound on fourth moment characters only} are equal to $n!^{-\frac 12+o(1)}$.} beyond $-\frac 13$.
\end{remark}

\begin{proof}[Proof of Lemma~\ref{lemma fourth moment}]
We let $M\geq 1$ be a parameter and we split the sum appearing in the numerator of 
$W_4(L/K;C_1,C_2)$ 
according to the degree of $\chi$.
By Lemma \ref{lemma first bound on Artin conductor}, one has that 
\begin{align*}
\sum_{\substack{\chi\in \Irr(G^+)\\ \chi(1)>M}}|\widehat{t^+}(\chi)|^4 \log (A(\chi)+2) &\ll \sum_{\substack{\chi\in \Irr(G^+)\\ \chi(1)>M}}|\widehat{t^+}(\chi)|^4 \frac{(\log (A(\chi)+2))^2}{\chi(1)}\\
 &<\frac 1M \Big(\sum_{\chi\in\Irr(G^+)}|\widehat{t^+}(\chi)|^2 \log (A(\chi)+2)\Big)^2. 
\end{align*}
Applying Lemma~\ref{lemma pointwise bound Fourier transform}, we also have the bound
$$
\sum_{\substack{\chi\in\Irr(G^+)\\ \chi(1)\leq M}}|\widehat{t^+}(\chi)|^4 \log (A(\chi)+2) \leq M^2 \norm{t^+}_1^2 \sum_{\chi\in\Irr(G^+)}|\widehat{t^+}(\chi)|^2 \log (A(\chi)+2).
$$
Putting everything together we deduce that
$$  
W_4(L/K;t)\ll \frac{M^2 \norm{t^+}_1^2}{\sum_{\chi \in \Irr(G^+)} |\widehat{t^+}(\chi)|^2 \log (A(\chi)+2)} + \frac 1M, 
$$
and \eqref{equation bound on fourth moment} follows from taking $M = \norm{t^+}_1^{-\frac 23}\big(\sum_{\chi\in \Irr (G^+)} |\widehat{t^+}(\chi)|^2 \log (A(\chi)+2)\big)^{\frac 13}$.

The proof of \eqref{equation bound on fourth moment characters only} goes along the same lines, by replacing $\log (A(\chi)+2)$ with $\chi(1)$.
\end{proof}

In the central limit theorem (Proposition~\ref{proposition central limit theorem}) we are about to prove, we will keep the setting as in Lemma~\ref{lemma fourth moment} and we will use estimates on the following important quantity:
\begin{equation}
F(L/K;t):=\frac{\V[X(L/K;t)]^{\frac 12}}{\max_{\eta \in \Irr(G^+)}|\widehat{t^+}(\eta)|},
\label{equation definition F}
\end{equation}
where $t\colon G\rightarrow \C$ is a class function such that $\widehat{t^+} \not \equiv 0$ (so that, as we have seen already, $F(L/K;t)\neq 0$).
This quantity will determine 
the range of validity of our bounds on the characteristic function of $X(L/\Q;t)$. 

\begin{remark}
\label{remark trivial bounds F}
We have the immediate bounds
$$ \frac{\norm{t^+}_2}{|(G^+)^\#|^{\frac 12}}   \leq \max_{\eta \in \Irr(G^+)}|\widehat{t^+}(\eta)|\leq  \norm{t^+}_2.$$

\end{remark}

We now state and prove estimates on the characteristic function of $X(L/K;t)$
which can be interpreted as effective central limit theorems via L\'evy's criterion and the Berry--Esseen bounds. 
These estimates will allow us to study moderate biases. Note that to obtain a precise estimate on the bias we will need bounds on the characteristic functions which hold in a wide range. For any class function $t:G \rightarrow \C$ such that $\widehat{t^+} \not \equiv 0$,
we define 
the normalized random variable
\begin{align}
\label{eq:definition Y C_1 C_2}
 Y(L/K;t)&:= \frac{X(L/K;t)-\E[X(L/K;t)]}{\V[X(L/K;t)]^{\frac 12}}\,
\end{align}
where $X(L/K;t)$
satisfies~\eqref{random variable for two classes}.
The corresponding characteristic function will be denoted by $\widehat Y(L/K;t)$.

\begin{proposition}[Characteristic function bounds]
\label{proposition central limit theorem}
Let $L/K$ be an extension of number fields such that $L/\Q$ is Galois, and for which AC and LI$^{-}$ hold. 
Fix a class function $t:G\rightarrow \R$. 
If $\beta_L^t=\tfrac 12$ (recall~\eqref{equation definition beta}), then in the range $|\eta|\leq \frac{3}{5}F(L/K;t)$ (see \eqref{equation definition F}) we have the bounds
$$-\frac{\eta^2}{2}-O(\eta^4 W_4(L/K;t))\leq \log \hat Y(L/K;t)(\eta)\leq -\frac{\eta^2}2.$$

\label{central limit theorem}
\end{proposition}

\begin{proof}
We first see that the characteristic function of $X(L/K;t)$ is given by  
\begin{equation}
 \hat{X}(L/K;t)(\eta)=\ee^{i\E[X(L/K;t)]\eta} \prod_{ \chi \in \Irr(G^+)} \prod_{\gamma_{\chi}>0} J_0\bigg(  \frac{2|\widehat{t^+}(\chi)|\eta}{\big(\frac 14+\gamma_{\chi}^2\big)^{\frac 12}}\bigg). 
\label{characteristic function two conjugacy classes}
\end{equation}
This comes from the fact that the characteristic function is multiplicative on independent random variables, and that $\hat X_{\gamma_{\chi}}(t)=J_0(t)$ (see the proof of \cite[Prop. 2.13]{FiMa}). 
From the properties of characteristic functions, it follows that
\begin{equation}
\log \hat Y(L/K;t)(\eta) = \sum_{ \chi \in \Irr(G^+)} \sum_{\gamma_{\chi}>0} \log J_0\bigg(  \frac{2|\widehat{t^+}(\chi)|\eta}{\V[X(L/K;t)]^{\frac 12}\big(\frac 14+\gamma_{\chi}^2\big)^{\frac 12}}\bigg).
\label{characteristic function normalized}
\end{equation} 
In the range $|u|\leq \frac{12}5$, we have the following bounds on the Bessel function (see~\cite[\S 2.2]{FiMa}):
\begin{equation}
-\frac{u^2}{4}-O(u^4)\leq \log J_0(u)\leq -\frac{u^2}{4}.
\label{equation first order Taylor of Bessel}
\end{equation}  
Inserting the bounds \eqref{equation first order Taylor of Bessel} in \eqref{characteristic function normalized}, we obtain that in the range $|\eta|\leq \frac 35 F(L/K;t)$,
 \begin{multline}
- \frac 1{\V[X(L/K;t)]}\sum_{\chi \in \Irr(G^+)} \sum_{\gamma_{\chi}>0} \frac{|\widehat{t^+}(\chi)|^2\eta^2}{\frac 14+\gamma_{\chi}^2}\geq \log \hat Y(L/K;t)(\eta) \geq  \\ - \frac 1{\V[X(L/K;t)]}\sum_{\chi \in \Irr(G^+)} \sum_{\gamma_{\chi}>0} \frac{|\widehat{t^+}(\chi)|^2\eta^2}{\frac 14+\gamma_{\chi}^2} 
-O\Big( \frac 1{\V[X(L/K;t)]^2}\sum_{\chi \in \Irr(G^+)} \sum_{\gamma_{\chi}>0} \frac{|\widehat{t^+}(\chi)|^4\eta^4}{(\frac 14+\gamma_{\chi}^2)^2}  \Big).
\label{eq:Taylor expansion of characteristic function of Y}
\end{multline} 
Note that the upper bound in \eqref{eq:Taylor expansion of characteristic function of Y} is equal to $-\eta^2/2$. As for the error term in the lower bound on the right hand side, we apply~\eqref{equation proposition multiplicities last lower bound} to conclude that it is
$$
\ll \eta^4\bigg(\sum_{\chi \in \Irr(G^+)} |\widehat{t^+}(\chi)|^2\log (A(\chi)+2)\bigg)^{-2}\sum_{\chi \in \Irr(G^+)} |\widehat{t^+}(\chi)|^4B_2(\chi)  \,.
$$
Here we have used the symmetry principle already mentioned (see the proof of Corollary~\ref{cor:explicitformulae1}) asserting that if $\rho$ is a complex zero of $L(s,L/K,\chi)$ then 
$\bar{\rho}$ is a zero of $L(s,L/K,\overline{\chi})$.
Invoking Lemma~\ref{lemma bounds on B(chi)} we recognize the fourth moment 
$W_4(L/K;t)$, and the claim follows. 
\end{proof}

From this central limit theorem we derive our general result on moderate biases. We first state and prove the following preliminary lower bound on the size of the quantity $F(L/K;t)$ defined in~\eqref{equation definition F}.

\begin{lemma}
\label{lemma bound on the argument}
Let $L/K$ be a number field extension such that $L/\Q$ is Galois. Assume AC, GRH and LI$^-$ hold, and fix 
$t\colon G\rightarrow \R$ a class function for which $t^+\not \equiv 0$. 
Then we have the bounds\footnote{The lower bound here is more convenient to work with than the upper bound in Remark \ref{remark trivial bounds F}.}
\begin{equation}\label{conclusion in lemma argument tends to zero}
\frac{\V[X(L/K;t)]^{\frac 16}}{ \norm{t^+}^{\frac 13}_1} \ll F(L/K;t) \ll \frac{ |(G^+)^{\sharp}|^{\frac 12} }{\norm{t^+}_2} \V[X(L/K;t)]^{\frac 12}.
\end{equation}
\end{lemma}

\begin{proof}
The upper bound follows from Remark~\ref{remark trivial bounds F}.
For the lower bound, by Proposition~\ref{proposition asymptotic for the variance} we have that
\begin{equation}
F(L/K;t)^{-1}\ll\frac{\max_{\psi \in \Irr(G^+)}|\widehat{t^+}(\chi)|}{\Big(\sum_{\chi \in \Irr(G^+)} |\widehat{t^+}(\chi)|^2 \log (A(\chi)+2)\Big)^{\frac 12}} \,.
\label{eq:lemma bound argument first equation}
\end{equation}
Let $\psi_0$ be an irreducible character 
of $G^+$ having the property that $|\widehat{t^+}(\psi_0)|=\max_{\psi \in \Irr(G^+)}|\widehat{t^+}(\psi)|$. Note that by Lemma~\ref{lemma pointwise bound Fourier transform}, $|\widehat{t^+}(\psi_0)|\leq \psi_0(1) \norm{t^+}_1 $.
We deduce from Lemmas~\ref{lemma first bound on Artin conductor} and~\ref{lemma bounds on B(chi)}, and by positivity of the summands in the denominator of \eqref{eq:lemma bound argument first equation} that
$$
F(L/K;t)^{-1}\ll \min \left(  \frac 1{\psi_0(1)^{\frac 12}},\frac{\psi_0(1) \norm{t^+}_1}{\V[X(L/K;t)]^{\frac 12}}   \right)\,. 
$$
Now, if $\psi_0(1)\geq \V[X(L/K;t)]^{\frac 13}\norm{t^+}_1^{-\frac 23}$ then $\psi_0(1)^{-\frac 12}\leq \norm{t^+}_1^{\frac 13}\V[X(L/K;t)]^{-\frac 16}$, and if $\psi_0(1)\leq \V[X(L/K;t)]^{\frac 13}\norm{t^+}_1^{-\frac 23}$ then $\psi_0(1)\norm{t^+}_1\V[X(L/K;t)]^{-\frac 12}\leq \norm{t^+}_1^{\frac 13}\V[X(L/K;t)]^{-\frac 16}$. We conclude that 
$$F(L/K;t)^{-1}\ll \norm{t^+}_1^{\frac 13}\V[X(L/K;t)]^{-\frac 16}. $$
\end{proof}

We are now ready to show that small values of $B(L/\Q;t)$ result in densities $\delta(L/K;t)$ that are close to $\tfrac 12$. In a sense this is a converse to Theorem~\ref{proposition asymptotic formula for highly biased races}.  

\begin{theorem}
\label{theorem asymptotic formula for moderately biased races}
Let $L/K$ be an extension of number fields such that $L/\Q$ is Galois, and for which AC and LI$^{-}$ hold. Fix a class function $t:G\rightarrow \R$ for which $\widehat{t^+}\not \equiv 0$.
Assuming that $\beta_{L}^{t}=\tfrac 12$, the following estimate holds:
\begin{equation} \delta(L/K;t) =\frac 12 +\frac {B(L/K;t)}{\sqrt{2\pi}} 
+O\Big( B(L/K;t)^3+\Big(\frac{\norm{t^+}_1^{2}}{\V[X(L/K;t)]}\Big)^2 +W_4(L/K;t)\Big)\,. 
\label{equation delta Ci close to 1/2}
\end{equation}
(Recall the definition \eqref{equation definition of W_4}.)
\end{theorem}
Note that under the hypotheses of Theorem~\ref{theorem asymptotic formula for moderately biased races}, we have that $\norm{t^+}_1^{2}\V[X(L/K;t)]^{-1} $ $\leq \norm{t^+}_2^{2}\V[X(L/K;t)]^{-1} \ll 1$. 
Note also that the exponent $2$ in the second error term can be replaced by an arbitrarily large real number, however the third summand $W_4(L/K;C_1,C_2)$ is expected to be the main contribution to the error term.

\begin{remark}
\label{Remark:after CLT}
In the particular case where $\widehat{t^+}\in \{ 0,-1,1\}$ (for example $K=\Q$ and $t=r$), we have that $W_4(L/K;t) \asymp \V[X(L/K;t)] $, and hence the second error term in Theorem~\ref{theorem asymptotic formula for moderately biased races} can be removed. 

Coming back to the general case,  Theorem \ref{theorem asymptotic formula for moderately biased races}
 implies, using Lemma~\ref{lemma fourth moment}, the simpler bound:
$$ \delta(L/K;t) =\frac 12 +\frac {B(L/K;t)}{\sqrt{2\pi}} 
+O\Big(B(L/K;t)^{3}+ \Big(\frac{\norm{t^+}_1^{2}}{\V[X(L/K;t)]}\Big)^{\frac 13}\Big)\,. 
$$
If moreover $ \V[X(L/K;t)]^{\frac 16}\norm{t^+}_1^{-\frac 13}=o(|\E[X(L/K;t)]|\norm{t^+}_1^{-1})$ but yet $\E[X(L/K;t)]$ $=o(\V[X(L/K;t)]^{\frac 12})$, then
$$ \delta(L/K;t)-\frac 12 \sim \frac {B(L/K;t)}{\sqrt{2\pi}}.$$

\end{remark}

\begin{proof}[Proof of Theorem~\ref{theorem asymptotic formula for moderately biased races}]
If $\E[X(L/K;t)]=0$, then in light of \eqref{random variable for two classes} and by independence, the random variable $X(L/K;t)$ is symmetric. We deduce that $\P[X(L/K;t)>0] =\frac 12$ and so the statement is trivial. It is also trivial when $B(L/K;t)$ or $\V[X(L/K;t)]^{-1}$ is bounded below by a positive constant. Therefore we may assume from now on that $B(L/K;t)$ is small enough and that $\V[X(L/K;t)]$ is large enough.

We now use the Berry--Esseen inequality in the form given by Esseen~(\cite[Chap. 2, Th. 2a]{Es}). The statement is as follows.
 If we denote by $F_Y$ the cumulative density function of a given real-valued random variable $Y$ and by $F_G$ that of the Gaussian distribution, then for any $T>0$,
$$ \sup_{x\in \mathbb R} |F_{Y}(x)-F_{G}(x)| \ll \int_{-T}^T \Big|\frac{\hat Y(\eta)-\ee^{-\frac {\eta^2}2}}{\eta}\Big| {\rm d}\eta+\frac 1T. $$
Taking $Y=Y(L/K;t)$ (recall \eqref{eq:definition Y C_1 C_2})
and setting 
$T:=\V[X(L/K;t)]^2\norm{t^+}_1^{-4}$, we have that
\begin{multline} \P[X(L/K;t)>0] = \P[Y>-B(L/K;t)] \\
= \frac 1{\sqrt{2\pi}}\int_{-B(L/K;t)}^{\infty} \ee^{-\frac {t^2}2} {\rm d}t +O\Big( \int_{-T}^T \Big|\frac{\hat Y(\eta)-\ee^{-\frac {\eta^2}2}}{\eta}\Big| {\rm d}\eta+\frac {\norm{t^+}_1^{4}}{\V[X(L/K;t)]^2}\Big). \label{equation two ranges of Esseen} \end{multline}
To bound the part of the integral in the error term of \eqref{equation two ranges of Esseen} in which $|\eta|\leq  \frac 35 F(L/K;t)$, we apply Proposition~\ref{proposition central limit theorem} which implies that for some absolute constant $c>0$,
$$  0 \geq  \hat Y(\eta)-\ee^{-\frac {\eta^2}2} \geq \ee^{-\frac {\eta^2}2}(\ee^{-c W_4(L/K;t)\eta^4} -1)\geq -c W_4(L/K;t)\eta^4  \ee^{-\frac {\eta^2}2},$$
  by the convexity bound $\ee^{-x}\geq 1-x$. We deduce that
\begin{align*}
 \int_{-F(L/K;t)}^{ F(L/K;t)} \Big|\frac{\hat Y(\eta)-\ee^{-\frac {\eta^2}2}}{\eta}\Big| {\rm d}\eta 
 &\ll
  W_4(L/K;t) \int_{\mathbb R} |\eta|^3 \ee^{-\frac {\eta^2}2} {\rm d}\eta \ll  W_4(L/K;t)\,.
 \end{align*}
As for the rest of the integral in the error term in \eqref{equation two ranges of Esseen}, we will use the properties of the Bessel function $J_0$, in a way analogous to~\cite[Prop. 2.14]{FiMa}. 
We have that for $|\eta|>  \tfrac 5{12}F(L/K;C_1,C_2)$,
$$
 J_0 \Big( 2\cdot \frac{\eta |\widehat{t^+}(\chi)|}{\V[X(L/K;t)]^{\frac 12}} \cdot \Big(\frac 14+\gamma_{\chi}^2 \Big)^{-\frac 12} \Big) \leq \\  J_0 \Big( 2\cdot \frac{5}{12}\frac{F(L/K;t)|\widehat{t^+}(\chi)|}{\V[X(L/K;t)]^{\frac 12}} \cdot \Big(\frac 14+\gamma_{\chi}^2 \Big)^{-\frac 12} \Big),
$$
and hence by~\eqref{characteristic function normalized}, $|\hat Y(\eta)|\leq |\hat Y(\frac 5{12}F(L/K;t))|$. It follows that
\begin{align*}\int_{F(L/K;t)<|\eta|\leq T} \Big|\frac{\hat Y(\eta)-\ee^{-\frac{\eta^2}2}}{\eta}\Big|{\rm d}\eta \ll & \hat Y\Big(\frac 5{12} F(L/K;t)\Big) \log T + \int_{|\eta|> F(L/K;t)} \frac{\ee^{-\frac{\eta^2}2}}{|\eta|}{\rm d}\eta\\
 \ll & \ee^{-\frac{25}{289}F(L/K;t)^{2}}+\ee^{-\frac{1}{3}F(L/K;t)^{2}},
\end{align*}
by Proposition~\ref{proposition central limit theorem} and the bound~$F(L/K;t)\gg \V[X(L/K;t)]^{\frac 16}\norm{t^+}_1^{-\frac 13}$ of Lemma~\ref{lemma bound on the argument}. We deduce that
\begin{equation}
 \P[X(L/K;t)>0] = \frac 1{\sqrt{2\pi}} \int_{-B(L/K;t)}^{\infty} \ee^{-\frac {t^2}2} {\rm d}t + O\Big(W_4(L/K;t)+ \frac {\norm{t^+}_1^{4}}{\V[X(L/K;t)]^2}\Big), 
 \label{equation comparison with gaussian res}
\end{equation}
and the claimed result follows by expanding the main term of~\eqref{equation comparison with gaussian res} into Taylor series.
\end{proof}

\section{General Galois extensions: proofs of Theorems~\ref{theorem general criterion conjugacy classes biased},~\ref{theorem general criterion conjugacy classes not biased},
 and~\ref{th:DistTo1orHalf}}

\label{section:proofsgeneral}

We fix the setup as before: $L/K$ is an extension of number fields for which $L/\Q$ is Galois.
 We first give general bounds on the mean, variance and bias factor (see~\eqref{equation definition bias factors}, and Proposition~\ref{proposition link with random variables}) associated to the random variable $X(L/K;t)$ that will also be used to prove the statements about extensions of number fields with specific Galois groups.

\begin{lemma}
Let $L/K$ be an extension such that $L/\Q$ is Galois, and for which AC, GRH and LI$^-$ hold. Fix a class function $t:G\rightarrow \C$ such that $\widehat{t^+}\not \equiv 0$.
We have the general bounds
\begin{equation}
\V[X(L/K;t)] \gg \sum_{\chi \in \Irr(G^+)} \chi(1)|\widehat{t^+}(\chi)|^2\gg \frac{\norm{t^+}^{3}_2}{\norm{t^+}_1(\#{ \rm supp}(\widehat{t^+}))^{\frac 12}} \,.
\label{eq:lemma lower bound var 1}
\end{equation} 
 Under the additional assumption of LI, 
\begin{equation}
\E[X(L/K;t)] 
 \ll(\norm t_2+\norm{t^+}_2)(\# \{\chi \in \Irr(G)\cup \Irr(G^+)\colon  \chi \text{ real}\})^{\frac 12},
 \label{eq:lemma lower bound mean 1}
\end{equation}
and as a result
\begin{align}
\label{eq:lemma lower bound var 2}
B(L/K;t) &\ll \frac{|\langle t,r\rangle_G+ 2\ord_{s=\frac 12} L(s,L/K,t)|}{\Big(\sum_{\chi \in \Irr(G^+)} \chi(1)|\widehat{t^+}(\chi)|^2\Big)^{\frac 12}} \\ & \ll \frac{\norm{t^+}_1^{\frac 12} (\norm{t}_2+\norm{t^+}_2)(\#\Irr(G^+))^{\frac 14}\cdot (\#\{ \chi \in \Irr(G)\cup \Irr(G^+) \colon \chi \text{ real} )^{\frac 12}}{\norm{t^+}_2^{\frac 32}}\,. \nonumber
\end{align}  

\label{lemma lower bound on variance in terms of group}
\end{lemma}

\begin{remark}
The second bounds in \eqref{eq:lemma lower bound var 1} and \eqref{eq:lemma lower bound var 2} are unconditional.
Moreover, the upper bound \eqref{eq:lemma lower bound var 2} implies that Galois groups with few irreducible characters correspond to small values of $B(L/K;t)$, and hence exhibit moderate discrepancies in the error term of the Chebotarev density theorem.
\end{remark}

\begin{proof}[Proof of Lemma \ref{lemma lower bound on variance in terms of group}]
The first bound in \eqref{eq:lemma lower bound var 1} follows from combining
Proposition~\ref{proposition link with random variables} with Lemmas~\ref{lemma first bound on Artin conductor} and \ref{lemma bounds on B(chi)}.
As for the second, we argue as in the proof of Lemma \ref{lemma fourth moment}. Introducing a parameter $M\geq 1$ we see that
\begin{align*}
 \sum_{\chi \in \Irr(G^+)} \chi(1)|\widehat{t^+}(\chi)|^2 &\geq M\sum_{\substack{\chi \in \Irr(G^+) \\ \chi(1)\geq M}} |\widehat{t^+}(\chi)|^2  \geq  M\Big(\sum_{\substack{\chi \in \Irr (G^+)}} |\widehat{t^+}(\chi)|^2 - \norm{t^+}_1^2M^2\#{ \rm supp}(\widehat{t^+})\Big)\\
&= M(\norm{t^+}_2^2-\norm{t^+}_1^2M^2\#{ \rm supp}(\widehat{t^+})).
\end{align*}
The right most bound in \eqref{eq:lemma lower bound var 1} follows by taking $M=\norm{t^+}_2\norm{t^+}_1^{-1}(2 \#{ \rm supp}(\widehat{t^+}))^{-\frac 12}$.

The bound~\eqref{eq:lemma lower bound mean 1} is established as follows:
\begin{align*}
|\E[X(L/K;t)] | &\leq  \sum_{\substack{\chi \in \Irr(G) \\ \chi \text{ real}}} |\widehat{t}(\chi)|+M_0\sum_{\substack{\chi \in \Irr(G^+) \\ \epsilon_2(\chi)=-1 }} |\widehat{t^+}(\chi)| \\
 & \leq \norm{t}_2 (\# \{\chi \in \Irr(G)\colon  \chi \text{ real}\})^{\frac 12} +\norm{t^+}_2 (\# \{\chi \in \Irr(G^+)\colon  \chi \text{ real}\})^{\frac 12}.
\end{align*} 
Recalling the definition \eqref{equation definition bias factors}, the first bound in \eqref{eq:lemma lower bound var 2} follows from combining
\eqref{equation expectancy with characters}
with the first bound in \eqref{eq:lemma lower bound var 1}. 
As for the second one, it follows from combining~\eqref{eq:lemma lower bound var 1} and~\eqref{eq:lemma lower bound mean 1}.
\end{proof}

We are now ready to prove the results of Section~\ref{subsec:general}.

\begin{proof}[Proof of Theorem \ref{theorem general criterion conjugacy classes not biased}] 

We split the proof into two cases depending on the value taken by $\beta_L^t$ (see~\eqref{equation definition beta}).
First, we assume that $\beta^t_L>\frac 12$. Arguing once more as in~\cite[Lemma 3.6]{Fi2} and~\cite[Proof of Theorem 5.4 (i)]{De}, we have that
$$\widehat X(L/K;t)(\xi)= \ee^{ i\xi \E[X(L/K;t)] }\prod_{\substack{ \chi \in {\rm supp}(\widehat{t^+}) }}\prod_{\substack{ \rho_{\chi} \\ \Re(\rho_{\chi})=\beta_{L} \\ \Im(\rho_{\chi})>0} } J_0\Big(  \frac{2|\widehat{t^+}(\chi)|\xi}{|\rho_{\chi}|^2}\Big).    $$
Assumption LI implies $\E[X(L/K;t)]=0$. For each $\chi \in {\rm supp}(\widehat{t^+})$, the product over $\rho_{\chi}$ has at least one factor (by GRH$^-$). However, $\widehat{t^+}(\chi) \not \equiv 0$, and hence we conclude that $|\widehat X(L/K;t)(\xi)|\ll (|\xi|+1)^{-2}$ and as before,
$\delta(L/K;t)=\tfrac 12$ by symmetry.

We now assume $\beta^t_L=\frac 12$. 
Recall Propositions~\ref{proposition link with random variables} and~\ref{proposition asymptotic for the variance};
combining Lemma \ref{lemma lower bound on variance in terms of group} and the  assumptions implies the bound $
B(L/K;t)  \ll \epsilon\,.
$
Theorem~\ref{theorem asymptotic formula for moderately biased races} (in the form of Remark \ref{Remark:after CLT}) then implies
that 

\begin{equation}
\delta(L/K;t) =\frac 12 +\frac {B(L/K;t)}{\sqrt{2\pi}} +O( B(L/K;t)^3+\V[X(L/K;t)]^{-\frac 13}). 
\label{equation pour delta}
\end{equation}
Using this estimate, the first statement follows from the lower bound on the variance in Lemma \ref{lemma lower bound on variance in terms of group}. As for the second one, it follows from~\eqref{equation pour delta} and the fact that the additional hypothesis in the statement implies that $|\E[X(L/K;t)]|\geq \epsilon^{-\frac 12}$, and hence $\V[X(L/K;t)]\geq \epsilon^{-3}$.  
\end{proof}

\begin{proof}[Proof of Corollary \ref{corollary general criterion conjugacy classes not biased}]
For both statements, we apply LI and
 combine Theorem \ref{theorem general criterion conjugacy classes not biased} with \eqref{eq:lemma lower bound var 1} and \eqref{eq:lemma lower bound var 2}.
\end{proof}

\begin{proof}[Proof of Theorem~\ref{theorem general criterion conjugacy classes biased}]
Assume that $\widehat{t^+} \not \equiv 0$, as well as AC, GRH and BM.
We will combine the expression for 
$\V[X(L/K;t)]$ given in Proposition~\ref{proposition link with random variables} with Lemmas~\ref{lemma first bound on Artin conductor} and~\ref{lemma bounds on B(chi)}. 
The assumption of 
Theorem~\ref{theorem general criterion conjugacy classes biased} translates into
$$
\frac{\E[X(L/K;t)]^2}{\V[X(L/K;t)]}\gg \varepsilon^{-1}\,.
$$
For $\varepsilon$ small enough we can then apply Proposition \ref{proposition asymptotic formula for highly biased races} and Proposition~\ref{theorem asymptotic formula for highly biased races} to conclude
 the proof.
\end{proof}

\begin{proof}[Proof of Theorem~\ref{th:DistTo1orHalf}]
We begin with part (1). As in the proof of Theorem \ref{theorem general criterion conjugacy classes not biased}, one shows that if $\beta_{L}^{r}>\tfrac 12$, then $\delta(L/\Q;r)=\tfrac 12$; we can therefore assume that $\beta_{L}^{r}=\tfrac 12$. 
We will apply Proposition~\ref{theorem asymptotic formula for highly biased races} and Theorem~\ref{theorem asymptotic formula for moderately biased races} to evaluate the bias factors $B(L/\Q;\res)$.
By Proposition \ref{proposition asymptotic for the variance} and Lemma~\ref{lemma first bound on Artin conductor} we have that
\begin{align*}
B(L/\Q;\res)^2& \ll \Big(\sum_{\substack{1\neq \chi\in \Irr(G)\\ \chi\r}}1\Big)^2
\Big(\sum_{\substack{1\neq \chi\in \Irr(G)\\ \chi\r}}\log (A(\chi)+2)\Big)^{-1}\\
&\ll \Big(\sum_{\substack{1\neq \chi\in \Irr(G)\\ \chi\r}}1\Big)^2
\Big(\sum_{\substack{1\neq \chi\in \Irr(G)\\ \chi\r}}\chi(1)\Big)^{-1}\\
&\leq \sum_{\substack{1\neq \chi\in \Irr(G)\\ \chi\r}}1\,.
\end{align*}

For part (2) we note using~\eqref{equation bound mean res} 
 that the stated condition implies that 
$$ \E[X(L/\Q,r)]\gg \sum_{\substack{1\neq \chi\in \Irr(G)\\ \chi\r}}1, $$
and thus $\V[X(L/\Q,\res)]^{-1} \ll B(L/\Q;\res)^2. $ Moreover, 
\begin{align*}
B(L/\Q;\res)^2 &\gg \Big(\sum_{\substack{1\neq \chi\in \Irr(G)\\ \chi\r}}1\Big)^2
\Big(\sum_{\substack{1\neq \chi\in \Irr(G)\\ \chi\r}}\log (A(\chi)+2)\Big)^{-1}\,,\\
& \gg (\log ({\rm rd}_L+2))^{-1} \Big(\sum_{\substack{1\neq \chi\in \Irr(G)\\ \chi\r}}1\Big)^2
\Big(\sum_{\substack{1\neq \chi\in \Irr(G)\\ \chi\r}}\chi(1)\Big)^{-1}\\
&\gg (\log(d_L+2))^{-1}\Big( \sum_{ \substack{ \chi \in \Irr(G) \\  \chi \r}}\chi(1)^2\Big)^{-\frac 12},
\end{align*}
by the Cauchy-Schwarz inequality and Lemma~\ref{lemma first bound on Artin conductor}.
If $B(L/\Q;\res)$ is small enough, the result follows from Theorem~\ref{theorem asymptotic formula for moderately biased races}. Otherwise, note that our hypothesis implies that $\V[X(L/\Q,\res)]$ is large enough, and the result follows from~\eqref{equation comparison with gaussian res}.
 \end{proof}

\section{General $S_n$-extensions}
\label{section:ProofsSn}
We now move to our particular results, starting with the case of a Galois extension $L/K$ 
of number fields with group $G=S_n$. The representation theory of the symmetric group is a beautiful blend of combinatorics and algebra. We refer the reader \emph{e.g.} to~\cite[Chap. 2]{Sag} for the details. In the following, we will focus on the special cases $t=t_{C_1,C_2}$ for $C_1,C_2\in G\cup \{ 0\}$, and $t=1-r$.

\subsection{Combinatorial estimates}
The conjugacy classes of $S_n$ are determined by cycle type, and hence are canonically indexed by the partitions $\lambda = (\lambda_1 , \dots \lambda_k)$ of $n$ (that is $\lambda_1 +\dots +\lambda_k=n$ and $\lambda_1 \leq \dots \leq \lambda_k $).
 Denoting by $C_\lambda$ the conjugacy class associated to $\lambda$, one can obtain closed combinatorial formulas for the quantity $r(C_\lambda)$ (see \cite[\S 5.3.5]{Ng}).

There is also a canonical parametrisation of the irreducible representations of $S_n$ in terms of the partitions $\lambda \vdash n$. This is achieved \textit{via} the Specht modules $ S^{\lambda}$, which are generated by linear combinations of tabloids with coefficients $\pm 1$ (see \cite{Sag}). Denoting by $\chi_{\lambda}$ the irreducible character associated to $ S^{\lambda}$, it follows that $\Irr (S_n) = \{ \chi_{\lambda} : \lambda \vdash n \}$.  The number of irreducible representations is therefore equal to $p(n)$ 
  the number of partitions of $n$, for which we have the Hardy--Ramanujan asymptotic formula (\cite[(5.1.2)]{An}):
\begin{equation}\label{eq:asymppart}
p(n)\sim \frac{{\rm e}^{\pi\sqrt{ \frac{2n}3}}}{4n\sqrt 3} \qquad (n\rightarrow\infty)\,.
\end{equation} 
We can picture a partition with its associated Ferrer diagram (\cite[Def. 2.1.1]{Sag}). We denote by $r(\lambda)$
 the number of rows of this diagram, and by $c(\lambda)$ its number of columns. 
It is known \cite[Section 2.7]{Sag} that all irreducible representations of $S_n$ are orthogonal. In Table \ref{table S_6} we consider the example $n=6$, in which $C_{\lambda}, r(C_{\lambda})$ are directly computed and the dimensions $\chi_{\lambda}(1)$ are obtained via the hook-length formula~\cite[Th. 3.10.2]{Sag}. 
\ytableausetup{centertableaux,smalltableaux}
\begin{table}
\label{table S_6}
\caption{The conjugacy classes and irreducible characters of $S_6$}
$$
\begin{tabular}{c|c|c|c|c|c|c|c|c|c|c|c}
$\lambda$ & \ytableaushort{\none}*{1,1,1,1,1,1} & \ytableaushort{\none}*{2,1,1,1,1}& \ytableaushort{\none}*{2,2,1,1} & \ytableaushort{\none}*{2,2,2} & \ytableaushort{\none}*{3,1,1,1}& \ytableaushort{\none}*{3,2,1} &\ytableaushort{\none}*{3,3} & {\tiny \ytableaushort{\none}*{4,1,1}}& {\tiny \ytableaushort{\none}*{4,2}}& {\tiny \ytableaushort{\none}*{5,1}} & {\tiny\ytableaushort{\none}*{6} } \\
\hline
$|C_\lambda|$ & 1 & 15   &  45 & 15  & 40  & 120  & 40  & 90 & 90 & 144 &120 \\
\hline
$1-r(C_\lambda)$ & -75  & 1   & -3  & 1  & -3   & 1   & -3 & 1 & 1 & 0 & 1  
\\ \hline 
$\chi_\lambda(1)$ & 1 & 5 & 9 & 5 &10 &16 & 5 & 10 & 9 &5 &   1
\end{tabular}
$$
\end{table}

 Combining~\eqref{equation elements of order m} with the asymptotic~\cite[(2.2)]{Wil} of Moser and Wyman on the number of involutions in $S_n$, we have a precise control on the sum of the degrees of irreducible representations of $S_n$:
\begin{equation} \label{eq:sumdegSn}
\sum_{ \lambda \vdash n}\chi_\lambda(1)= |\{\sigma\in S_n\colon \sigma^2={\rm Id}\}| \sim \Big(\frac n\ee \Big)^{\frac n2} \frac{\ee^{\sqrt n}}{\ee^{\frac 14}\sqrt 2}  .
\end{equation}

It turns out that most character values $\chi_\lambda(\pi)$ with $\pi \neq {\rm Id}$ are of small size compared to $\chi_\lambda(1)$.
This well-known fact has applications to various
  problems such as mixing times of random walks, covering by powers of conjugacy
   classes and probabilistic and combinatorial properties of word maps (see \cite{LaSh}).
    In our case, it will allow us to obtain sharp estimates for the Artin conductors $A(\chi_{\lambda})$. 
The bound we will apply is due to Roichman.

\begin{theorem}[{\cite[Theorem 1]{Ro}}]
Let $n>4$. Then for any $\lambda \vdash n$ and $\pi \in S_n$ we have the bound
\begin{equation}
\label{equation bound Roichman}
\frac{|\chi_\lambda(\pi)|}{\chi_\lambda(1)} \leq \left( \max\left(q,\frac{r(\lambda)}n,\frac{c(\lambda)}n\right) \right)^{b \text{ supp}(\pi)},
\end{equation}
where $0<q<1$ and $b>0$ are absolute constants and supp$(\pi)$ is the number of non-fixed points of $\pi$. 
\end{theorem}
There are also more recent bounds due to F\'eray and \'Sniady \cite[Th. 1]{FeSn}, and Larsen and Shalev \cite[Th. 1.1 and Th. 1.2]{LaSh}, however Roichman's is sufficient for our purposes.
We will need a combinatorial bound on the degree of the irreducible representation associated to $S^{\lambda}$ in terms of $r(\lambda)$ and $c(\lambda)$.
\begin{lemma}
\label{lemma bound on strange partitions}
For any $\lambda \vdash n$ we have the bound
$$ \chi_{\lambda}(1) = f^{\lambda} \ll n\cdot n!^{1-\frac{r(\lambda)+c(\lambda)}n} \ee^{2n/\ee}.$$ 
\end{lemma}
One could possibly improve this bound to one of the type  $$\chi_\lambda(1) \ll n^{\theta}\cdot n!^{1-\frac{r(\lambda)+c(\lambda)}n} \ee^{2n/\ee} $$
for some $\theta <1$, however this would not affect Theorem \ref{theorem S_n races}.
\begin{proof}[Proof of Lemma \ref{lemma bound on strange partitions}]
Applying the hook-length formula~(\cite[Th. 3.10.2]{Sag}) and considering only the hook-lengths of the first row and of the first column, we see that
$$ f^{\lambda}\leq \frac{n!}{(r(\lambda)+c(\lambda)-1) (r(\lambda)-1)! (c(\lambda)-1)! } \leq n \frac{n!}{r(\lambda)!c(\lambda)!}. $$
Here we have used the fact that the conditions $1\leq x,y \leq n$, $x+y-1\leq n$ imply the bound $ xy/(x+y-1) \leq (n+1)^2/4n \leq n$.
Next we apply Stirling's formula and obtain that
$$ f^{\lambda} \ll n  \left(\frac n{\ee} \right)^{n-r(\lambda)-c(\lambda)} \left(\frac n{r(\lambda)} \right)^{r(\lambda)} \left(\frac n{c(\lambda)} \right)^{c(\lambda)} \ll n\cdot n!^{1-\frac{r(\lambda)+c(\lambda)}n} \ee^{2n/\ee}\,;$$
the last equality following from the fact that for fixed $n$ the function $t\mapsto (n/t)^t$ attains its maximal value on $(0,n]$ at $t=n/\ee$.
\end{proof}

\subsection{Proof of Theorem~\ref{theorem S_n races}}

We first treat the easier case of $\delta(L/\Q;\res)$.

\begin{lemma}
Let $n\geq 2$ and assume that AC, GRH and LI hold for the $S_n$-extension $L/\Q$. 
We have the following estimates 
(recall the definition \eqref{equation definition bias factors}):
 $$\V[X(L/\Q;\res)] \asymp
  \log ({\rm rd}_L) (n/\ee)^{n/2}\ee^{\sqrt{n}}\,.
$$
 $$B(L/\Q;\res) \asymp 
 (\log ({\rm rd}_L))^{-\frac 12}  p(n) n!^{-\frac 14} \ee^{-\frac{\sqrt{n}} 2}n^{\frac 18} \,.
$$
\label{lemma estimation of the bias factors for S_n res}
\end{lemma}

\begin{proof}
We will show that 
$$\V[X(L/\Q;\res)] \asymp \log ({\rm rd}_L) \sum_{\substack{\lambda \vdash n }} \chi_\lambda(1);$$
the claimed bound on $\V[X(L/\Q;\res)]$ then follows from~\eqref{eq:sumdegSn}. 

Proposition~\ref{proposition asymptotic for the variance} implies that
$$ \V[X(L/\Q;\res)] \asymp \sum_{\substack{\lambda \vdash n }} \log (A(\chi_\lambda)+1),$$
and hence from Lemma~\ref{lemma first bound on Artin conductor} we deduce the required upper bound. As for the lower bound,
setting $M=n!^{\frac 13}$, we see that
\begin{align*}
\V[X(L/\Q;\res)] \gg \sum_{\substack{\lambda \vdash n \\ \chi_\lambda(1) \geq M }} \log (A(\chi_\lambda)+1)\gg  \log ({\rm rd}_L) \sum_{\substack{\lambda \vdash n \\ \chi_\lambda(1) \geq M }} \chi_\lambda(1).
\end{align*}
Indeed, by Lemma \ref{lemma bound on strange partitions}, the condition $\chi_\lambda(1) \geq M$ implies that $$\max(r(\lambda),c(\lambda))\leq n\Big(1-\frac{\log M}{\log n!}\Big)+O\Big(\frac{n}{\log n}\Big),$$ which in turn by \eqref{equation bound Roichman} and Lemma \ref{lemma finer bounds on Artin conductor} implies that $\log (A(\chi_\lambda)+1) \gg \chi_\lambda(1) \log ({\rm rd}_L)$, for $n$ large enough. 
 We deduce that for some absolute $c>0$,
$$\V[X(L/\Q;\res)] \geq c  \log ({\rm rd}_L) \Big(\sum_{\substack{\lambda \vdash n \\ }} \chi_\lambda(1)-Mp(n)\Big).  $$
Hence, for $n$ large enough, \eqref{eq:asymppart} and \eqref{eq:sumdegSn} imply the required lower bound. 
 For the remaining  (finite number of) values of $n\geq 8$, we can pick 
$$ \lambda_n:= \begin{cases} (\frac{n}2, \frac{n}2) & \text{ if } n \text{ is even} \\ ( \frac{n+1}2,\frac{n-1}2 ) &\text{ otherwise}; 
\end{cases}$$
then \eqref{equation bound Roichman} and Lemma \ref{lemma finer bounds on Artin conductor} imply the required bound $$\V[X(L/\Q;\res)]\gg \log ( A(\chi_{\lambda_n})+1) \gg \chi_{\lambda_n}(1) \log ({\rm rd}_L).$$ The same bound holds for $2\leq n\leq 7$ by inspection of the character table of $S_n$.

For the claimed estimate on $B(L/\Q;\res)$, we recall that every irreducible character of $S_n$ is orthogonal, and hence LI implies that the Artin $L$-functions of irreducible representations of $\Gal(L/\Q)$ have no real zeros. Thus, Proposition~\ref{proposition link with random variables} takes the form 
$$ \E[X(L/\Q;\res)] =\sum_{\substack{\lambda \vdash n }} 1-1 = p(n)-1 $$
and the claim is proved thanks to the estimate on the variance we just proved. 
\end{proof}

The following lemma, which is stated for a general class functions $t$, will be applied in Proposition~\ref{lemma lower bound variance S_n big conj classes} in the case $t=t_{C_1,C_2}$.
\begin{lemma}
\label{lemma pre lower bound variance S_n big conj classes}

Fix $\varepsilon>0$, and let $n$ be large enough in terms of $\varepsilon$. Let $L/K$ be an extension of number fields for which $L/\Q$ is Galois, $G^+=\Gal(L/\Q)=S_n$, and such that AC, GRH and LI$^-$ hold. Fix a class function $t:G\rightarrow \C$ such that $\norm{t^+}_2 \geq \ee^{\frac {(2+\varepsilon)n} \ee}\norm{t^+}_1$. 
Then we have the bounds 

$$\V[X(L/K;t)] \gg   (1-\max(q, \tfrac{\log(kn!\norm{t^+}^{-1}_2 \norm{t^+}_1 \ee^{4n/\ee})}{\log n!})^b ) \log ({\rm rd}_L)\sum_{ \substack{\lambda \vdash n   }} \chi_\lambda(1)|\widehat{t^+}(\chi_\lambda)|^2; 
$$
\begin{equation}
 \frac{\norm{t^+}_2^3}{\norm{t^+}_1 (\# \Irr(G^+))^{\frac 12}}  \ll \sum_{ \substack{\lambda \vdash n   }} \chi_\lambda(1)|\widehat{t^+}(\chi_\lambda)|^2 \ll n!^{\frac 12} \norm{t^+}^2_2.
   \label{equation general bounds variance Sn}
\end{equation}
Here, $b,k>0$ and $0<q<1$ are absolute.

\end{lemma}

\begin{remark}
The condition $\norm{t^+}_2 \geq \ee^{\frac {(2+\varepsilon)n} \ee}\norm{t^+}_1$ can be reinterpreted by saying that $t^+$ is far from being constant. If we normalize so that $\norm{t^+}_1=1$, then this condition holds provided there exists $C\in G^\sharp$ such that $|t(C)|\geq (n!^{\frac 12}/|C|^{\frac 1 2}) \ee^{\frac {(4+2\varepsilon)n} \ee}$.
\end{remark}

\begin{proof}
We will apply Proposition~\ref{proposition lower bound variance}. Lemma~\ref{lemma bound on strange partitions} implies that for some absolute $k\geq 1$,
$$r(\lambda)+c(\lambda)\leq  n\frac{\log(\frac{kn!}{\chi_{\lambda}(1)})+\frac{2n}e}{\log n!}.  $$
Hence, by~\eqref{equation bound Roichman}, if $\lambda \vdash n$ is such that $  \chi_\lambda(1)\geq \norm{t^+}_2 \norm{t^+}_1^{-1}$, then for some absolute $b>0$ and $0<q<1$,
$$ \max_{{\rm id} \neq \pi \in S_n }\frac{\chi_{\lambda}(\pi)}{\chi_{\lambda}(1)}\leq \Big(\max\Big(q,\frac{\log(kn!\norm{t^+}^{-1}_2 \norm{t^+}_1)+\frac{2n}e}{\log n!} \Big)\Big)^b. $$
Note that for $n$ large enough and by our assumptions, $0<\log(kn!\norm{t^+}^{-1}_2 \norm{t^+}_1)+\frac{2n}e < \log n!. $ (For the lower bound, note that $  \norm{t^+}_2 \leq n!^{\frac 12}\norm{t^+}_1$.) The claimed lower bound on $\V[X(L/K;t)]$ then follows from Proposition~\ref{proposition lower bound variance}.
The lower bound in~\eqref{equation general bounds variance Sn} is just~\eqref{eq:lemma lower bound var 1}. As for the upper bound, it follows from noting that $\chi_{\lambda}(1) \leq n!^{\frac 12}$.
\end{proof}
We now evaluate the bounds of Lemma~\ref{lemma pre lower bound variance S_n big conj classes} more precisely in the particular case $t=t_{C_1,C_2}$.
\begin{proposition}
\label{lemma lower bound variance S_n big conj classes}
Fix $\varepsilon>0$ and let $n\geq 2$. Let $L/K$ be an extension of number fields for which $L/\Q$ is Galois, $G^+=\Gal(L/\Q)=S_n$, and such that AC, GRH and LI$^-$ hold. If $C_1$, $C_2$ are two elements of $ G^\sharp \cup \{0\}$ for which $\min(|C_1^+|,|C_2^+|) \leq n!^{1-\frac {4+\varepsilon}{\ee \log n}} $, then we have the bounds 

$$ 
\V[X(L/K;t_{C_1,C_2})] \gg_\varepsilon \Big(1-\frac{\log \min(|C_1^+|,|C_2^+|)}{\log n!} \Big)   \log ({\rm rd}_L)\sum_{ \substack{\lambda \vdash n   }} \chi_\lambda(1)|\chi_\lambda(C_1^+)-\chi_\lambda(C_2^+)|^2;
$$
$$\ \Big(1-\frac{\log \min(|C_1^+|,|C_2^+|)}{\log n!} \Big)  \frac{n!^{\frac 32}\log ({\rm rd}_L) }{ \min(|C_1^+|,|C_2^+|)^{\frac 32} p(n)^{\frac 12}} \ll_\varepsilon \V[X(L/K;t_{C_1,C_2})] \ll \frac{n!^{\frac 32}\log ({\rm rd}_L) }{ \min(|C_1^+|,|C_2^+|) } . $$
\end{proposition}
\begin{proof}
 For $n$ large enough, we apply Lemma \ref{lemma pre lower bound variance S_n big conj classes} with $t=t_{C_1,C_2}$. We see that under our assumptions, 
\begin{multline*}
(\tfrac{\log(kn!\norm{t^+}^{-1}_2 \norm{t^+}_1 \ee^{4n/\ee})}{\log n!})^b = (1-\tfrac{\log(k^{-1} n!(p(n) \norm{t^+}_2 \norm{t^+}_1^{-1}  \ee^{4n/\ee})^{-1})}{\log n!} )^b\\ \leq 1- c_b \tfrac{\log(k^{-1} n!(p(n) \min(|C_1^+|,|C_2^+|) \ee^{4n/\ee})^{-1})}{2\log n!},
\end{multline*}
where $c_b>0$ depends only on $b$, since $n$ is large enough (in terms of $b$). Hence, a straightforward calculation shows that both claimed lower bounds follow from Lemma \ref{lemma pre lower bound variance S_n big conj classes}. As for the finitely many remaining values of $n\geq 5$, we note that for $\lambda \notin \{ (n),(1,1,\dots 1) \}$, \eqref{equation bound Roichman} implies that $\log A(\chi) \gg_n  \chi_{\lambda}(1)\log ({\rm rd}_L)$. Hence, as before,
\begin{equation}
\V[X(L/K;t_{C_1,C_2})] \gg_n \log ({\rm rd}_L)\sum_{\lambda \notin \{ (n),(1,1,\dots 1) \} } |\chi_\lambda(C_1^+)-\chi_\lambda(C_2^+)|^2. 
\label{eq:lowerboundvarfinitelymanyn}
\end{equation}
This implies both claimed lower bounds since 
$$ \sum_{\lambda \notin \{ (n),(1,1,\dots 1) \} } |\chi_\lambda(C_1^+)-\chi_\lambda(C_2^+)|^2 \geq  \sum_{\lambda \vdash n } |\chi_\lambda(C_1^+)-\chi_\lambda(C_2^+)|^2-4= \frac{n!}{|C_1^+|} + \frac{n!}{|C_2^+|} -4,$$
which is strictly positive given our restriction on the conjugacy classes $C_1$ and $C_2$. 
For $n \in \{3, 4\}$, the right hand side of \eqref{eq:lowerboundvarfinitelymanyn} is $\gg_n  \log ({\rm rd}_L) $ by inspection of the character table of $S_n$. 
Finally, the case $n=2$ is immediate because then $\log A(\chi)\asymp \log ({\rm rd}_L)$ for the unique nontrivial character of $S_2$.

 As for the upper bound, it follows from combining Proposition \ref{proposition asymptotic for the variance} with Lemmas \ref{lemma first bound on Artin conductor} and \ref{lemma pre lower bound variance S_n big conj classes}.
\end{proof}

\begin{lemma}
\label{lemma bound expectancy S_n large conj class}
Fix $\varepsilon>0$ and let $n\geq 2 $. Let $L/K$ be an extension of number fields for which $L/\Q$ is Galois, $G^+=\Gal(L/\Q)=S_n$, and such that AC, GRH and LI hold. If $C_1$, $C_2$ are two elements of $ G^\sharp \cup \{0\}$, then have the bound 

\begin{equation}
\E[X(L/K;t_{C_1,C_2})] \ll \Big( \frac{n! p(n)}{\min(|C_1|,|C_2|)} \Big) ^{\frac 12} . 
\label{eq:bound expectancy S_n large conj classes}
\end{equation}
If moreover $\min(|C_1^+|,|C_2^+|)\leq n!^{1-\frac {1+\varepsilon}{\ee \log n}} $, then
\begin{equation}
\label{eq:bound 4th moment S_n large conj classes}
W_4[X(L/K;t_{C_1,C_2})] \ll \Big(1-\frac{\log \min(|C_1^+|,|C_2^+|)}{\log n!} \Big)^{-2}  \frac{n!^{-\frac 12}\min(|C_1^+|,|C_2^+|)^{\frac 12}p(n)^{\frac 16}}{\log ({\rm rd}_L)}.
\end{equation}

\end{lemma}

\begin{proof}
We begin with the bound~\eqref{eq:bound expectancy S_n large conj classes}. Recalling~\eqref{equation definition r(C)}, we see that for any representative $g_i\in C_i$, 
$$  r_G(C_i)  = \#\{g \in G : g^2=g_i \}\leq  \#\{g \in G^+ : g^2=g_i \} = r_{G^+}(C_i^+).$$ 
Note that by~\eqref{equation link sum real zeros big to small group},
$$  {\rm ord}_{s=\frac 12}L(s,L/K,t) =  {\rm ord}_{s=\frac 12}L(s,L/\Q,t^+) =0,$$
since we are assuming LI and since all characters of $S_n$ are orthogonal. 
By Proposition~\ref{proposition link with random variables}, it follows that 
\begin{align*}
|\E[X(L/K;t_{C_1,C_2})]| & \leq r_{G^+}(C_1^+)+r_{G^+}(C_2^+) \leq \sum_{\chi \in \Irr(G^+)} (|\chi(C_1^+)|+|\chi(C_2^+)|),
\end{align*}
and hence~\eqref{eq:bound expectancy S_n large conj classes} follows after an application of the Cauchy-Schwarz inequality. 
As for \eqref{eq:bound 4th moment S_n large conj classes}, we note that by definition of $W_4[X(L/K;t_{C_1,C_2})]$, Lemma~\ref{lemma first bound on Artin conductor} and Proposition~\ref{lemma lower bound variance S_n big conj classes}, 
$$W_4[X(L/K;t_{C_1,C_2})] \ll \Big(1-\frac{\log \min(|C_1^+|,|C_2^+|)}{\log n!} \Big)^{-2}\frac{\sum_{\lambda \vdash n}\chi_\lambda(1)|\chi_\lambda(C_1^+)-\chi_\lambda(C_2^+)|^4 }{\log ({\rm rd}_L)(\sum_{\lambda \vdash n} \chi_\lambda(1)|\chi_\lambda(C_1^+)-\chi_\lambda(C_2^+)|^2 )^2}.$$
The desired estimate follows at once from Lemmas~\ref{lemma fourth moment} and~\ref{lemma pre lower bound variance S_n big conj classes}.
\end{proof}

We are now ready to prove Theorem \ref{theorem S_n races}.
\begin{proof}[Proof of Theorem \ref{theorem S_n races}]
The mean and variance bounds follow from Proposition~\ref{lemma lower bound variance S_n big conj classes} and Lemma~\ref{lemma bound expectancy S_n large conj class}. Moreover, those bounds imply that
\begin{align}
\label{eq:bound on B ready for theorem S_n}
B(L/K;t_{C_1,C_2}) &\ll 
 \Big(1-\frac{\log \min(|C_1^+|,|C_2^+|)}{\log n!} \Big)^{-\frac 12} \frac{n!^{-\frac 14} p(n)^{\frac 34}\min(|C_1^+|,|C_2^+|)^{\frac 14}}{(\log ({\rm rd}_L))^{\frac 12}};
\end{align}
\begin{equation}
\label{eq:bound on V ready for theorem S_n}
\V[X(L/K;t_{C_1,C_2})]^{-2} \ll \Big(1-\frac{\log \min(|C_1^+|,|C_2^+|)}{\log n!} \Big)^{-2} \frac{n!^{-3} \min(|C_1^+|,|C_2^+|)^3 p(n) }{\log ({\rm rd}_L)  }.
\end{equation} 

In light of Theorem \ref{theorem asymptotic formula for moderately biased races}, the estimate~\eqref{eq:UnbiasedCCSn} on $\delta(L/K;t_{C_1,C_2})$ then follows from combining these bounds with Lemma \ref{lemma bound expectancy S_n large conj class} (note that we always have $n! \geq \min(|C_1^+|,|C_2^+|)$).

We now move to the second claimed estimate in which $K=\Q$ and $C_1\neq  C_2= \{{\rm id}\}$.
By orthogonality, positivity of $\chi_{\lambda}(1)-\chi_{\lambda}(C_1)$, and Cauchy--Schwarz:
\begin{align*}
 n!&= \sum_{\lambda \vdash n}\chi_{\lambda}(1)(\chi_{\lambda}(1)-\chi_{\lambda}(C_1))\leq 
  \left(\sum_{\lambda \vdash n}\chi_{\lambda}(1)^2\right)^{\frac 12}\left(\sum_{\lambda \vdash n}(\chi_{\lambda}(1)-\chi_{\lambda}(C_1))^2\right)^{\frac 12}\\
  &\leq n!^{\frac 12}\sum_{\lambda \vdash n}(\chi_{\lambda}(1)-\chi_{\lambda}(C_1))\,.
\end{align*}
We deduce that

\begin{equation}
\sum_{\lambda \vdash n}(\chi_{\lambda}(1)-\chi_{\lambda}(C_1))\geq n!^{\frac 12}\,,
\label{equation lower bound mean Sn}
\end{equation}
and hence Proposition~\ref{lemma lower bound variance S_n big conj classes} yields that
$$
B(L/\Q,t_{C_1,\{ {\rm id}\}}) \gg  \frac{n!^{-\frac 14} }{ (\log ({\rm rd}_L))^{\frac 12}}\,.
$$
The estimate~\eqref{equation theorem Sn races best possible} is deduced from combining \eqref{eq:bound 4th moment S_n large conj classes}, \eqref{eq:bound on B ready for theorem S_n}, \eqref{eq:bound on V ready for theorem S_n} and Theorem~\ref{theorem asymptotic formula for moderately biased races}. 

Finally, the estimate on $\delta(L/\Q;\res)$ follows directly from Proposition \ref{theorem asymptotic formula for moderately biased races} and Lemma \ref{lemma estimation of the bias factors for S_n res}.
\end{proof}

\section{Abelian extensions}\label{section:AbelGalois}

If $G=\Gal(L/\Q)$ is abelian, then all its irreducible representations are one-dimensional. In particular an irreducible character is real-valued if and only if its associated representation is realizable over the reals, hence $\epsilon_2(\chi)\neq -1$ for all $\chi\in\Irr(G)$. Therefore~\eqref{equation number of square roots} shows that the number of real characters of $G$ is equal to the number of elements of $G$ of order at most two.

Also, for distinct elements $a,b \in G$ we have that
$$ \V[X(L/K;t_{a,b})]\asymp \sum_{\chi \in \Irr(G^+) } |\chi(a)-\chi(b)|^2 \log (A(\chi)+1) \gg \sum_{\chi \in \Irr(G^+)} |\chi(a)-\chi(b)|^2=2|G^+|. $$

\subsection{2-elementary groups: proof of Theorem~\ref{th:HighlyBiasedAbelian}}
We study the Galois extension $L=\Q(\sqrt{p_1},\sqrt{p_2},...,\sqrt{p_m})/\Q$ of group $G\simeq(\Z/2\Z)^m$ in the setting of Theorem \ref{th:HighlyBiasedAbelian}, under hypotheses GRH and LI. 

\begin{proof}[Proof of Theorem \ref{th:HighlyBiasedAbelian}]

We first compute the Artin conductor explicitly. Clearly, besides $p=2$ the only ramified prime in $\Q(\sqrt{p_j})/\Q$ is $p_j$ which factorizes as $p_j\mathcal O_L=(\sqrt{p_j})^2$. Hence, the odd primes ramifying in $L/\Q$ are $p_1,\ldots,p_m$ (see \cite[Prop. 2.19]{Lem}) and the ramification is tame at each of these primes. Moreover if $\p_j$ denotes a prime ideal of $\O_L$ lying over $p_j$, we easily see that the corresponding inertia group is:
 $$
 G_0(\p_j/p_j)=\Gal(L/\Q(\sqrt{p_1},\dots,\sqrt{p_{j-1}} , \sqrt{p_{j+1}} ,\dots,\sqrt{p_m})),
 $$
  which we identify with the subgroup $\langle e_j\rangle \leqslant \{\pm 1\}^m$, where $e_j=(1,\ldots,1,-1,1,\ldots,1)$, with the coefficient $-1$ in $j$-th position. Since the ramification at each $p_j$ is tame, we use the following formula to compute $n(\chi,p_j)$ for every $\chi\in\Irr(G)$ (see~\eqref{eq:condalternate}):
  $$
  n(\chi,p_j)={\rm codim}(V^{G_0})=\frac 1{|G_0|} \sum_{a \in G_0} (1-\chi(a))=\frac{1-\chi(e_j)}{2}\,.
  $$ 
  We deduce that
$$ 2^{-n(\chi,2)} A(\chi) = \prod_{1\leq j\leq m} p_j^{\frac{1-\chi(e_j)}2}=\prod_{j : \chi(e_j)=-1} p_j.$$ 
Note also that by \cite[Prop 2.19]{Lem} and \cite[Chap. 5, Th. 31]{ZarSam} (see also~\cite[Sect. 5]{BKS}) we have that
 $$\disc(L/\Q) =\disc(\Q(\sqrt{p_1 \cdots p_m})/\Q)^{\frac{|G|}2} = \begin{cases} 2^{|G|} (p_1\cdots p_m)^{\frac{|G|}2} &\text{ if } p_1\cdots p_m \equiv 3 \bmod 4 \\ 
 (p_1\cdots p_m)^{\frac{|G|}2} &\text{ otherwise. }
 \end{cases} $$
We deduce that 
$$  \sum_{\chi \in \Irr (G)}n(\chi,2)\leq |G|. $$
We turn to the evaluation of the bias factor $B(L/\Q;t_{a,b})$ for distinct 
elements $a,b$ of $G$. By Proposition \ref{proposition asymptotic for the variance}, the variance of the random variable $X(L/\Q;t_{a,b})$ is 
\begin{align*} \V[X(L/\Q;t_{a,b})] &\ll \sum_{j\leq m} \log p_j \sum_{\substack{\chi \in \Irr (G)}} |\chi(a)-\chi(b)|^2 +  \sum_{\chi \in \Irr (G)}n(\chi,2) |\chi(a)-\chi(b)|^2 \\
&\ll |G|\sum_{\substack{j\leq m}} \log p_j.
\end{align*} 
We also have the lower bound
\begin{align*}
\V[X(L/\Q;t_{a,b})] &\gg
 \sum_{j\leq m} \log p_j \sum_{\substack{\chi  \in \Irr (G)}} \frac{1-\chi(e_j)}{2} (2-2\chi(ab)) \\
&= \sum_{j\leq m} \log p_j  |G| (1+\delta_{ab=e_j}) \,,
\end{align*} 
 where $\delta_{c=d}$ is $1$ if $c=d$ and $0$ otherwise. We conclude that 
 $$\V[X(L/\Q;t_{a,b})]\asymp |G|\sum_{\substack{j\leq m}} \log p_j. $$

Also, $\E[X(L/\Q;t_{a,b})]=r(b)-r(a)=|G|(\delta_{b=1}-\delta_{a=1})$ with notation as in Proposition~\ref{proposition link with random variables}. We deduce that for $a\neq (1,\ldots,1)$ and $b=(1,\ldots,1)$
$$
B(L/\Q;t_{a,1})^2\asymp \frac{|G|}{\displaystyle{\sum_{\substack{j\leq m}} \log p_j}}.
$$

In an analogous fashion we compute that
$$
B(L/\Q;\res)^2\asymp \frac{|G|}{\displaystyle{\sum_{\substack{j\leq m}} \log p_j}}\,.
$$
Theorem~\ref{th:HighlyBiasedAbelian} then follows from Proposition~\ref{theorem asymptotic formula for highly biased races} and Theorem~\ref{theorem asymptotic formula for moderately biased races}.
\end{proof}

\subsection{Hilbert class fields, the relative case: Proof of Theorem~\ref{th:RelHilb}}

The setting for this section is as in \S\ref{subsec:RelHilb}. We start by computing some useful invariants.

\begin{lemma}\label{lem:RelHilb}
Let $d$ be a fundamental discriminant such that $|d|> 1$. Let $K_d$ be the Hilbert class field of the 
quadratic field $\Q(\sqrt{d})$ so that $\Gal(K_d/\Q(\sqrt{d}))\simeq {\rm Cl}_d
$. We have that
$$
\log ({\rm rd}_{K_d})=\frac{\log |d|}{2}\,;\hspace{1cm} \sum_{\substack{\chi\in\Irr({\rm Cl}_d)\\ \chi\r}}\chi(1)\in\{2^{\omega(d)-1},2^{\omega(d)-2}\}\,,
$$
where $\omega(d)$ is the number of distinct prime factors of $d$.
\end{lemma}

\begin{proof}

Let us compute the discriminant of $K_d/\mathbb{Q}$. Applying~\cite[Chap. 5, Th. 31]{ZarSam} 
 to the tower of extensions $K_d/\Q(\sqrt{d})/\Q$ we have:
 $$
 |\disc (K_d/\Q)|=\mathcal N_{\Q(\sqrt{d})/\Q}\bigl(\disc (K_d/\Q(\sqrt{d}))\bigr)|d|^{[K_d:\Q(\sqrt{d})]}\,.
 $$
 Since $K_d$ is the Hilbert Class Field of $\Q(\sqrt{d})$, the extension $K_d/\Q(\sqrt{d})$ 
 is unramified and the relative discriminant $\disc (K_d/\Q(\sqrt{d}))$ equals the unit 
 ideal $\mathcal O_{\Q(\sqrt{d})}$. Therefore the ideal norm relative to $\Q(\sqrt{d})/\Q$
 of $\disc (K_d/\Q(\sqrt{d}))$ equals $1$.
The formula for $\log  ({\rm rd}_{K_d})$ follows 
 from the fact that $[K_d:\Q]=2h(d)$.
 
 For the second assertion, we first use the fact that ${\rm Cl}_d$ is abelian and then we invoke~\eqref{equation elements of order m} 
 and Theorem~\ref{theorem Frobenius Schur} following the general argument given at the beginning of \S\ref{section:AbelGalois}. This yields:
 $$
 \sum_{\substack{\chi\in\Irr({\rm Cl}_d)\\ \chi\r}}\chi(1)=\sum_{\substack{\chi\in\Irr({\rm Cl}_d)\\ \chi\r}}1=\#\{g\in{\rm Cl}_d\colon g^2=1\}\,.
 $$ 
 We conclude using the classical result from Gauss' genus theory according to which the 
 $2$-rank of the narrow class group of $\Q(\sqrt{d})$ equals $\omega(d)-1$ (see for instance \cite[Chapter 28 \S 8]{Ha}). 
 In other words the $2$-torsion of the narrow class group of $\Q(\sqrt{d})$ 
 has dimension $\omega(d)-1$ as an $\F_2$-vector space. Moreover 
 the index of the ordinary class group ${\rm Cl}_{\Q(\sqrt{d})}$ in the narrow 
 class group of $\Q(\sqrt{d})$ is either $1$ or $2$ depending on the sign of $d$ and 
 on the sign of the norm of the fundamental unit in the real quadratic case.   
\end{proof}

Using this lemma we are now ready to prove Theorem~\ref{th:RelHilb}.

\begin{proof}[Proof of Theorem~\ref{th:RelHilb}]
We identify $G=\Gal(K_d/\Q(\sqrt d))$ with the class group ${\rm Cl}_d$. The extension $K_d/\Q$ is Galois by Lemma~\ref{lem:HCF} and we denote $G^+=\Gal(K_d/\Q)$. We have $|G^+|=2h(d)$. Let $\overline{\mathfrak a}$ be a nontrivial ideal class of $\Q(\sqrt d)$ and denote by $\overline{1}$ the trivial ideal class. We apply Proposition~\ref{proposition link with random variables} with $t=t_{C_1,C_2}$,  $(C_1,C_2)=(\{\overline{\mathfrak a}\}, \{\overline{1}\})$ or 
$(C_1,C_2)=(\{\overline{1}\}, 0)$. The mean of the limiting distribution of $E(y;K_d/\Q(\sqrt d),t_{C_1,C_2})$ satisfies:
\begin{align*}
|\mu_{K_d/\Q(\sqrt d)}(C_1,C_2)|&=\bigg|\sum_{1\neq \chi\in\Irr({\rm Cl}_d)}(\chi(C_1)-\chi(C_2))(\varepsilon_2(\chi)+2{\rm ord}_{s=\tfrac 12} L(s,K_d/\Q(\sqrt d),\chi))\bigg|\,,\\
&\leq 2\sum_{\substack{\chi\in\Irr({\rm Cl}_d)\\ \chi \r}}\chi(1)+4\sum_{1\neq\chi\in\Irr({\rm Cl}_d)}{\rm ord}_{s=\tfrac 12} L(s,K_d/\Q(\sqrt d),\chi))\,,
\end{align*}
and the first upper bound on the mean follows by Lemma~\ref{lem:RelHilb}. Proposition~\ref{proposition asymptotic for the variance} and Lemma~\ref{lemma first bound on Artin conductor} (or rather a trivial form of the lower bound where we use $\chi(1)\geq 1$) yield the following lower bound on the variance, conditionally on GRH for  
$(C_1,C_2)=(\{\overline{1}\}, 0)$ and conditionally on GRH and LI$^-$ for $(C_1,C_2)=(\{\overline{\mathfrak a}\}, \{\overline{1}\})$:
\begin{align*}
\sigma^2_{K_d/\Q(\sqrt d)}(C_1,C_2)&\gg \sum_{1\neq \chi\in \Irr(G^+) }|\chi(C_1^+)-\chi(C_2^+)|^2 
\log A(\chi)\\
&\geq \sum_{1\neq \chi\in \Irr(G^+) }|\chi(C_1^+)-\chi(C_2^+)|^2\geq |G^+|\,,
\end{align*}
where the last step follows from~\eqref{eq:lemmasumsquareorthogonality}. 

Finally, LI asserts that only symplectic irreducible characters of $G^+$ may have their $L$-function vanish at $\tfrac 12$. However, the abelian group $G=\Gal(K_d/\Q(\sqrt d))$ does not admit such a character. As a consequence, 
we deduce that the second sum in the upper bound for $\mu_{K_d/\Q(\sqrt d)}(C_1,C_2)$ vanishes.  The statement on the density is then an immediate consequence of Theorem~\ref{theorem asymptotic formula for moderately biased races} combined with Remark~\ref{remark better bound on W4}.
\end{proof}

\section{Supersolvable extensions}
\label{section:supersolvable}

 We devote this section to the proofs of our results for two kinds of extensions:
 
\begin{itemize}
\item Galois extensions of number fields of group $G$ having an abelian subgroup of index $2$,
\item radical extensions which are splitting fields over $\Q$ of polynomials of type $X^p-a$ for distinct primes $a,p$.
\end{itemize}

In the first case $G$ has a quotient of order $2$, and in the second case $G$ has a normal subgroup of order $p$ and cyclic associated quotient (of order $p-1$; see below for a quick recollection of this fact). In particular both cases are instances of supersolvable extensions.
 
\subsection{Galois groups with an abelian subgroup of index $2$} \label{section:SubgpIndex2}

Let $G$ be a finite group and assume $G$ has an abelian subgroup $A$ of index $2$. The quotient $\Gamma=G/A\simeq \Z/2\Z$ acts on the abelian group $A$ \emph{via}:
$$
\tau\cdot \sigma = \tau_0\sigma\tau_0^{-1}\,,\qquad \sigma\in A,\, \Gamma=\langle \tau\rangle\,,
$$
where $\tau_0$ is any fixed lift of $\tau$ to $G$. For simplicity (and since it will be the case in our applications) we assume from now on that $\Gamma$ acts by inversion on $A$ \emph{i.e.} 
$\tau\cdot \sigma=\sigma^{-1}$ for every $\sigma\in A$. Since $G$ has an abelian subgroup of index $2$, the irreducible linear representations of $G$ (over $\C$) have degree $1$ or $2$~(\cite[Prop. 2.6]{Hup}).

We begin by computing the Frobenius-Schur indicators of these representations.  
If $\psi$ is a one-dimensional character of $G$, then we have for any $\sigma\in A$:
$$
\psi(\sigma)=\psi(\tau_0\sigma\tau_0^{-1})=\psi(\sigma^{-1})\,,
$$
Therefore $\psi(\sigma)=\pm 1$. In particular if $\tau_0$ has order $2$, we deduce from this computation and the fact that $G=A\langle \tau_0\rangle$ that $\psi$ is real 
hence $\epsilon_2(\psi)=1$ because $\psi$ has degree $1$.

As for the irreducible representations $\theta_\lambda$ of degree $2$ of $G$, they are all 
obtained~(\cite[\S 2.8]{Hup}) from a given 
$\lambda\in\Irr(A)$ by setting for $\sigma\in A$,
\begin{equation}
\theta_\lambda(\sigma)=\begin{pmatrix}
\lambda(\sigma) & 0 \\
0 & \lambda(\tau\cdot \sigma)
\end{pmatrix}
\,,\qquad
\theta_\lambda(\tau_0)=\begin{pmatrix}
0 & 1 \\
1 & 0
\end{pmatrix}\,.
\label{eq:values characters index 2}
\end{equation}

Our assumption $\tau\cdot\sigma=\sigma^{-1}$ directly implies that $\chi_{\theta_\lambda}(\sigma)=\text{tr}(\theta_\lambda(\sigma))$ is real. Moreover, for any $\lambda\in\Irr(A)$ and $\sigma\in A$, we have 
$\chi_{\theta_\lambda}(\tau_0\sigma)=0$. We deduce that 
$\chi_{\theta_\lambda}$ is a real character. By~\cite[\S 13.9]{Hup}, it follows that 
$\epsilon_2(\chi)=1$ for all $\chi\in\Irr(G)$.

The following lemma uses the information above to give a lower bound on the bias factors $B(L/K;\res)$ and $B(L/K;t_{C_1,C_2})$ defined by~\eqref{equation definition bias factors}, for suitably chosen conjugacy classes $C_1$, $C_2$.

\begin{lemma}\label{lemma: abelian subgp index 2}
Let $L/\Q$ be a Galois extension
for which GRH and BM hold. Assume that $G=\Gal(L/\Q)$ has an abelian subgroup $A$ of index $2$. Fix an element $\sigma\in A$ and let $C_1$ be the conjugacy class of $\tau_0\sigma$ where $\tau_0$ is a representative of the nontrivial left coset of $G$ modulo $A$. Assume also that $\tau_0$ has order $2$. We have the bounds
$$
\min(\E(L/\Q;t_{C_1,\{{\rm id}\}}),\E(L/\Q;\res)) \gg |G|\,;
$$ 
$$
\min(B(L/\Q;t_{C_1,\{{\rm id}\}}),B(L/\Q;\res))^2 \gg \frac{|G|}{\log ({\rm rd}_{L})}.
$$ 
\end{lemma}

\begin{proof}
We start with $\E(L/\Q;\res)$. 
We have alrealy seen that $\epsilon_2(\chi)=1$ for every $\chi\in\Irr(G)$, and hence we deduce that
$$
\sum_{\substack{1\neq\chi\in\Irr(G)\\ \chi\r}}\epsilon_2(\chi){\rm ord}_{s=\frac 12}L(s,L/\Q,\chi)\geq 0\,.
$$
Therefore, the desired lower bound follows by Proposition~\ref{proposition link with random variables}, since all characters of $G$ are real and  of dimension $\leq 2$.

As for $B(L/\Q;\res)$, by Lemmas~\ref{lemma first bound on Artin conductor}, and~\ref{proposition asymptotic for the variance}, we have that
\begin{align*}
B(L/\Q;\res)^2 & \gg  \frac{|G|^2}{
 \sum_{\substack{1\neq \chi\in \Irr(G)\\ 
 \chi \r}}\log (A(\chi)+1)} \gg
   \frac{|G|^2}{\log ({\rm rd}_{L})
 \sum_{\substack{\chi\in \Irr(G)}}\chi(1)} \gg    \frac{|G|}{\log ({\rm rd}_{L})
} \,.
\end{align*}
(We have used once more the fact that $\chi(1)\leq 2$.)

We turn to $\E(L/\Q;t_{C_1,\{{\rm id}\}})$. 
Starting from Proposition~\ref{proposition link with random variables}, one has for every $\chi\in\Irr(G)$:
$$
\epsilon_2(\chi)+2{\rm ord}_{s=1/2}L(s,L/\Q,\chi)\geq \epsilon_2(\chi)=1\,,\qquad
\chi(1)-\chi(C_1)\geq 0\,.
$$ 
We note that by \eqref{eq:values characters index 2} and by the orthogonality relations, one has
$$ 0= \sum_{\chi \in\Irr(G)} \chi(1)\chi(C_1) =\sum_{\substack{\chi\in\Irr(G) \\ \chi(1)=1}} \chi(1)\chi(C_1) = \sum_{\chi\in\Irr(G)}\chi(C_1). $$
We deduce the following simple lower bound for the expectation of $X(L/\Q,C_1,C_2)$:
\begin{align*}
\E[X(L/\Q,t_{C_1,\{{\rm id}\}})]
&\geq 
\sum _{\chi\in\Irr(G)}\chi(1)-\sum _{\chi\in\Irr(G)}\chi(C_1)= \sum_{\substack{\chi\in\Irr(G)}}\chi(1).
\end{align*}
As for the variance, one has, using 
Lemma~\ref{lemma first bound on Artin conductor} and Proposition~\ref{proposition asymptotic for the variance}:
$$
\V[X(L/\Q,t_{C_1,\{{ \rm id}\})}]\ll \log ({\rm rd}_{L}) 
\sum_{\chi\in\Irr(G)}\chi(1)|\chi(1)-\chi(C_1)|^2\,\leq 16\log ({\rm rd}_{L}) 
\sum_{\chi\in\Irr(G)}\chi(1)\,.
$$
The expected bound follows as in the previous case.
\end{proof}

\subsubsection{Dihedral extensions: proof of Theorem~\ref{th:dihedralext}}

Let us start by recalling some classical facts about dihedral groups and their representations (see \emph{e.g.}~\cite[\S 5.3]{SeRep}). 
Consider, for an odd integer $n\geq 3$,
$$
D_n=\langle \sigma,\tau\colon \sigma^n=\tau^2=1,\, \tau\sigma\tau=\sigma^{-1}\rangle\,.
$$
The nontrivial conjugacy classes of $D_n$ are 
$$
\{\sigma^j,\sigma^{-j}\}\,\, (1\leq j\leq (n-1)/2)\,, \qquad \text{and}\qquad 
\{\tau\sigma^k\colon 0\leq k\leq n-1\}\,. 
$$
There are therefore $(n+3)/2$ isomorphy classes of irreducible representations of $D_n$. Exactly two of them have degree 1: the trivial representation and the lift of the nontrivial character of $D_n/\langle \sigma\rangle$ which is defined by
$$
\psi(\sigma^j)=1\,,\qquad \psi(\tau\sigma^k)=-1\,.
$$
The remaining $(n-1)/2$ irreducible representations of $D_n$ have degree $2$; 
the associated characters are given by
$$
\chi_h(\sigma^j)=2\cos(2\pi hj/n)\,,\qquad \chi_h(\tau\sigma^k)=0\,,
$$
with $h\in\{1,\ldots (n-1)/2\}$.
Clearly Lemma~\ref{lemma: abelian subgp index 2} holds for the dihedral group $D_n$. To be in a position where Lemma~\ref{lemma: abelian subgp index 2} provides us with a family of unbounded bias factors, we need to control
 the size of the discriminant of the extensions that we consider. For that purpose we focus on a particular family of dihedral extensions of $\Q$ introduced by Kl\"uners which enjoy useful properties stated in the following lemma 
 (see~\cite[Lemma $3.4$]{kluners06}).
 
 \begin{lemma}[Kl\"uners] \label{lem:discdihedral}
 Let $d\neq 1$ be a squarefree integer and let
 $\Q(\sqrt{d})/\Q$ be a quadratic extension of discriminant $\delta_d\in\{d,4d\}$. Suppose 
 that there is an odd prime number $\ell$ 
 and two prime numbers $p,q$ which are $1$ modulo $\ell$ and 
 which split in $\Q(\sqrt{d})/\Q$. Then there exists 
 an extension $N_{\ell,p,q,d}/\Q(\sqrt{d})$ such that $N_{\ell,p,q,d}/\Q$ is Galois and
 $$
 \Gal(N_{\ell,p,q,d}/\Q)\simeq D_\ell,\qquad 
 |\disc(N_{\ell,p,q,d}/\Q)|\leq |\delta_d|^\ell (pq)^{2(\ell-1)}\,.
 $$
 \end{lemma}

\begin{proof}
For the existence of the dihedral extension $N_{\ell,p,q,d}/\Q$, 
see~\cite[Lemma $3.4$]{kluners06}. The upper bound for the discriminant is~\cite[(6)]{kluners06}.
\end{proof}

We now proceed to give an example of a family of dihedral extensions $(N_\ell/\Q)$ 
indexed by prime numbers and such that the second lower bound of Lemma~\ref{lemma: abelian subgp index 2} 
approaches infinity as $\ell$ grows.

\begin{proposition}\label{prop:HighBiasDihedral}
For each prime number $\ell\geq 7$, there exist 
and extension $N_\ell / \Q$ with Galois group $D_\ell$ such that 
 $$
\min(B(N_{\ell}/\Q;t_{C_1,\{ {\rm id}\}}),B(N_{\ell}/\Q;\res))^2\gg \frac{\ell}{\log \ell},
 $$
where $C_1$ is as in Lemma~\ref{lemma: abelian subgp index 2}.
\end{proposition}

\begin{proof}
A prime number $p$ splits completely in $\Q(\sqrt{5})$ if and only 
if $5$ is a square modulo $p$. By quadratic reciprocity this is equivalent to the condition $p\equiv \pm 1 (\bmod\, 5)$. Therefore 
if we pick primes $p,q$ that are congruent to $1$ modulo $5\ell$ then 
$p,q$ satisfy the hypotheses of Lemma~\ref{lem:discdihedral}.

By Linnik's Theorem~(\cite[(2)]{Lin}, \cite[Theorem 1.1]{Xy}), for each fixed $\ell$ we can find distinct primes $p,q$ that are $1$ modulo $5\ell$ and of size $\ll \ell^{5.18}$.
The result 
follows by applying Lemmas~\ref{lemma: abelian subgp index 2} and~\ref{lem:discdihedral}.
\end{proof}

We are now ready to prove Theorem~\ref{th:dihedralext}.

\begin{proof}[Proof of Theorem~\ref{th:dihedralext}]
The claimed bounds are a consequence of Propositions~\ref{proposition asymptotic formula for highly biased races} and~\ref{theorem asymptotic formula for highly biased races}, using as inputs Lemmas~\ref{lemma: abelian subgp index 2} and~\ref{prop:HighBiasDihedral}.  
The lower bound on $1-\delta(K_\ell/\Q;r)$ is an immediate consequence of Theorem~\ref{th:DistTo1orHalf}.
\end{proof}

\subsubsection{Hilbert class fields, absolute case: proof of Theorems~\ref{theorem vitrine} and~\ref{th:HilbertOverQ}}

We first review standard results on Hilbert class fields.
As in the dihedral case we will see that the Galois groups of the extensions considered in this section have an abelian subgroup of index $2$. This will once more allow us to apply Lemma~\ref{lemma: abelian subgp index 2}.

\begin{lemma}\label{lem:HCF}
Let $d\neq 1$ be a fundamental discriminant and let $K_d$ be the Hilbert 
Class Field of $\Q(\sqrt{d})$ so that 
$\Gal(K_d/\Q(\sqrt{d}))\simeq {\rm Cl}_d$, the class group of 
$\Q(\sqrt{d})$, of order $h(d)$. Let $\tau$ be the generator of $\Gal(\Q(\sqrt{d})/\Q)$. Then $K_d/\Q$ is Galois and, fixing  a representative 
$\tau_0$ for the left coset of $G_d=\Gal(K_d/\Q)$  modulo $\Gal(K_d/\Q(\sqrt{d}))$ determined by $\tau$, we have
$$
G_d\simeq{\rm Cl}_d\rtimes \langle \tau_0\rangle\,, \qquad \tau_0^2=1\,,\,\,\tau_0
 \sigma\tau_0^{-1}=\sigma^{-1}\,\,(\forall \sigma\in 
 {\rm Cl}_d)\,. 
$$
Moreover, $\log ({\rm rd}_{K_d})=\tfrac 12\log |d|$.
\end{lemma}

\begin{proof}
We have a short
 exact sequence
 $$
 1\rightarrow \Gal(K_d/\Q(\sqrt{d}))\simeq{\rm Cl}_d\rightarrow G_d\rightarrow \Gal(\Q(\sqrt{d})/\Q)\simeq \Z/2\Z\rightarrow 1\,.
 $$
 In particular $\Gal(K_d/\Q(\sqrt{d}))$ is an abelian subgroup of $G_d$ of index $2$,
  and, as explained at the beginning of \S\ref{section:SubgpIndex2},
 $\Gal(\Q(\sqrt{d})/\Q)=\langle \tau\rangle$ acts on it \emph{via}
 $$
 \tau\cdot\sigma=\tau_0\sigma\tau_0^{-1}\,.
 $$
  Moreover 
 the short exact sequence splits since $\Q(\sqrt{d})/\Q$ is cyclic 
 (see \emph{e.g.}~\cite[Th. 2]{Gold}) and the above action is inversion. Indeed let $\mathfrak p$ be a prime ideal of the ring of integers $\mathcal O_{\Q(\sqrt{d})}$ of $\Q(\sqrt{d})$. The Frobenius conjugacy class at $\mathfrak p$ in the (abelian, unramified) extension $K_d/\Q(\sqrt{d})$ is an actual element $\frob_\mathfrak p$ of $\Gal(K_d/\Q(\sqrt{d}))$ and we have the standard relation
 $$
 \tau_0 \frob_{\mathfrak p}\tau_0^{-1}=\frob_{\tau_0(\mathfrak p)}=\frob_{\tau(\mathfrak p)}\,,
 $$
 where we identify the restriction of elements of $G_d$ to $\Q(\sqrt{d})$ with their image by the quotient map $G_d\rightarrow G_d/\Gal(K_d/\Q(\sqrt{d}))$.
 If $p$ is the prime number lying under $\mathfrak p$, we have the ideal factorization $p\mathcal O_{\Q(\sqrt{d})}=\mathfrak p\tau(\mathfrak p)$ so that
 in ${\rm Cl}_d$ the classes of $\tau(\mathfrak p)$ and 
 $\mathfrak p^{-1}$ are the same. We conclude that $\tau_0 \frob_{\mathfrak p}\tau_0^{-1}=\frob_{\mathfrak p}^{-1}$ 
and we deduce the 
 group structure of $G_d$ from the Chebotarev Density Theorem. It remains to see that 
 $\tau_0$ has order $2$. First the order of $\tau_0$ divides $4$ 
 since $\tau_0^2\in 
 \Gal(K_d/\Q(\sqrt{d}))$ and therefore $\tau_0^4=(\tau_0
 \tau_0^2\tau_0^{-1})\tau_0^2=1$. Next for any $\sigma\in\Gal(K_d/\Q(\sqrt{d}))$, we have $(\tau_0\sigma)^2=(\tau_0\sigma\tau_0^{-1})\tau_0^2\sigma=\sigma^{-1}\tau_0^2\sigma=\tau_0^2$, since $\tau_0^2$ and $\sigma$ are both elements of the abelian group $\Gal(K_d/\Q(\sqrt{d}))$. Consequently every element 
 of the left coset of $G_d$ modulo $\Gal(K_d/\Q(\sqrt{d}))$ determined by $\tau$ 
 has the same order. Assume by contradiction that this order is $4$ and consider a prime $p$ 
 ramified in $\Q(\sqrt{d})/\Q$; since $K_d/\Q(\sqrt{d})$ is unramified, the ramification index of $\mathfrak P/p$ (here $\mathfrak P$ denotes any prime ideal of $\mathcal O_{K_d}$ lying over $p$) is $2$ and thus the inertia subgroup of $G_d$ relative to 
 $\mathfrak P/p$ has order $2$ and therefore contains no element of the left coset of $G_d$ modulo $\Gal(K_d/\Q(\sqrt{d}))$ determined by $\sigma$. Hence this inertia group is a subgroup of $\Gal(K_d/\Q(\sqrt{d}))$. Thus every element of the inertia group relative to $\mathfrak P/p$ fixes $\Q(\sqrt{d})$ pointwise, contradicting the fact that $p$ is ramified in $\Q(\sqrt{d})/\Q$.

Finally, the assertion on the root discriminant of $K_d$ was proven in Lemma \ref{lem:RelHilb}. 
\end{proof}

Lemma~\ref{lemma: abelian subgp index 2} and~\ref{lem:HCF} suggest that a family of quadratic fields $\Q(\sqrt{d})$ with class number $h(d)$ significantly larger than $\log |d|$ will produce an extreme bias. In order to achieve this we will exploit the following precise lower bound on the class group of a particular family of quadratic fields. Note that this result plays a role analogous to Proposition~\ref{prop:HighBiasDihedral} in the case of dihedral extensions.

\begin{lemma}\label{propo:HighBiasHCF}
For $d\neq 1$ a fundamental discriminant, let $h(d)$ be the class number of $\Q(\sqrt{d})$, and let $K_d$ be the Hilbert class field of 
$\Q(\sqrt{d})$.
Then there exists a sequence of positive (resp. negative) fundamental discriminants $d\neq 1$ such that one has
$$
 h(d)\gg \frac{\sqrt{|d|}\log\log |d|}{(\log |d|)^{\frac{1}{2}+\frac{{\rm sgn}(d)}{2}}}\,.
$$
Moreover, under GRH for $L(s,\chi_d)$, the bounds
$$
\frac{\sqrt{|d|}}{\log\log |d|}\ll h(d)\ll \sqrt{|d|}\log\log |d|\,
$$
hold for all $d<0$. As for $d>1$, under GRH for $L(s,\chi_d)$ we have that
$$
 h(d)\ll \frac{\sqrt{d}\log\log d}{\log d}\,.
$$
\end{lemma}

\begin{proof}
The first bound is an immediate consequence of~\cite[(3)]{MoWe} (see also~\cite[Th. 1.2(a)]{LamQuadFields}, as well as results towards a conjecture of Hooley~\cite{Ho} due to Fouvry~\cite[Th. 1.1]{Fou} with subsequent improvements in~\cite{Bou,Xi}, that further address the question of the density of values $d$ with attached fundamental unit of prescribed size) in the case $d>1$, and of Chowla~\cite{Ch} (see also~\cite{Du} for generalizations to number fields of higher degree) in the case $d<0$.

As for the GRH bounds, we apply the Littlewood bounds~(see \cite{Lit})
$$ \frac 1{\log\log |d|} \ll L(1,\chi_d) \ll \log\log |d|. $$
In the case $d<0$ both claimed bounds on $h(d)$ follow directly from the class number formula. As for the case $d>1$, it is well known that the fundamental unit satisfies $\epsilon_d\geq \sqrt d/2$, and hence the class number formula yields that
$$ \frac{h(d) \log d}{\sqrt d} \ll  \log\log d.$$
\end{proof}

\begin{remark}
It is expected~\cite[Conjecture 1]{Sar} that for positive fundamental discriminants $d$, we typically have $h(d)\ll_\varepsilon d^\varepsilon$.
The construction of Montgomery and Weinberger~\cite{MoWe} focuses on (the sparse set of) fundamental discriminants of the form $d=4n^2+1$. 
For such $d$ the fundamental unit of $\mathcal O_{\Q(\sqrt{d})}$ equals 
$\varepsilon_d=2n+\sqrt{d}$. Such fundamental units are of minimal order of magnitude 
$\sqrt{d}$ and this maximizes in turn the value of $h(d)$.
\end{remark}

\begin{remark}
The extension $K_d/\Q$, where $K_d$ is the Hilbert class field of the quadratic field $\Q(\sqrt{d})$, is not abelian in general, but contrary to the case of dihedral extensions, particular choices of $d$ may still produce an abelian extension $K_d/\Q$. Precisely if $F$ denotes the maximal abelian subextension of $K_d/\Q$ then $F\supset \Q(\sqrt{d})$ and $\Gal(F/\Q(\sqrt{d}))$ is a quotient of the class group ${\rm Cl}_d$ of $\Q(\sqrt{d})$. By the group structure of $\Gal(K_d/\Q)$ given in Lemma~\ref{lem:HCF}, $\Gal(F/\Q(\sqrt{d}))$ is the maximal quotient of $\Gal(K_d/\Q(\sqrt{d}))\simeq {\rm Cl}_d$ on which $\Gal(\Q(\sqrt{d})/\Q)$ acts trivially. Again by Lemma~\ref{lem:HCF} we conclude that $F=K_d$ if and only if ${\rm Cl}_d$ is an elementary $2$-group (such is the case, \emph{e.g.} for $\Q(\sqrt{-5})$ which has class number $2$). Weinberger~\cite[Theorem 1]{Wei} has shown that there are only finitely many negative fundamental discriminants $d$ such ${\rm Cl}_d$ is an elementary $2$-group. In the real quadratic case, let us mention that ${\rm Cl}_p$ for $p\equiv 1 (\bmod\ 4)$ prime such that $h_p>1$ is never an elementary $2$-group\footnote{By contradiction if ${\rm Cl}_p$ is a nontrivial elementary $2$-group for some prime $p\equiv 1 (\bmod\ 4)$ then $G_p=\Gal(K_p/\Q)$ is abelian and the only ramified prime (including infinite primes) in $K_p/\Q$ is $p$ and it is known \cite[Theorem 1.1]{BM} (since $G_p$ is abelian) that this minimal number of ramified primes equals the minimal number of generators of $G_p$. Therefore $G_p$ is cyclic which contradicts Lemma~\ref{lem:HCF}.}.
\end{remark}

We are now ready to prove Theorems~\ref{theorem vitrine} and~\ref{th:HilbertOverQ}.

\begin{proof}[Proof of Theorems~\ref{theorem vitrine} and~\ref{th:HilbertOverQ}]
The mean and variance are computed under GRH and BM in Lemma~\ref{lemma: abelian subgp index 2}. Under GRH alone, we need to go back to Proposition~\ref{proposition asymptotic for the variance} which we combine with the representation-theoretic calculations made in Lemma \ref{lemma: abelian subgp index 2}. Note that our assumption of GRH implies that the Riemann Hypothesis holds for $L(s,\chi_d)$. Indeed, $\chi_d$ lifts to an irreducible representation of $G_d$, and the corresponding Artin $L$-functions are identical since $K_d/\Q(\sqrt d)$ is unramified. For the computation of the densities $\delta(L/K;t_{C_1,C_2})$ and $\delta(L/K;\res)$, we apply Proposition~\ref{proposition asymptotic formula for highly biased races} and Proposition~\ref{theorem asymptotic formula for highly biased races} using as inputs Lemma~\ref{lemma: abelian subgp index 2} and Lemma~\ref{propo:HighBiasHCF}. Finally, note that by~\cite[Proposition 5.34]{IwKo}, 
$$\ord_{s=\frac 12} \zeta_{K_d}(s) \ll \frac{\log (3|d|^{h(d)})}{\log\log(3|d|^{h(d)})} \ll h(d).$$
\end{proof}

\subsection{Radical extensions: Proof of Theorem~\ref{th:BiasRadical}}

Radical extensions of the rationals are particularly well-suited to compute explicitly all the invariants required in our analysis. Notably the Artin conductors of the irreducible characters of the Galois group were computed in \cite{Vi} in a more general setting; for the sake of completeness we will show the details of this computation in our setting. Making precise the explicit value of such invariants is also interesting since it is an instance of a non-abelian extension where all the computations we need (\emph{e.g.} the filtration of inertia at ramified primes) can be explicitly performed. 

\subsubsection{The splitting field of $x^p-a$ over $\Q$}
Let $p,a$ be distinct primes with $p\neq 2$ and such that $a^{p-1}\not\equiv 1 (\bmod\ p^2)$. Consider $K_{a,p}=\Q(\zeta_p,a^{\frac 1p})$, the splitting field of $x^p-a$ over $\Q$. If $\sigma$ is an element of $G:=\Gal(K_{a,p}/\Q)$, then we have 
$$ \sigma(\zeta_p)=\zeta_p^c, \hspace{1cm} \sigma(a^{\frac 1p})=\zeta_p^d a^{\frac 1p} $$
with $c\in \F_p^*$, $d\in\F_p$; we may identify $\sigma$ with 
$$ \left(\begin{array}{cc}
c & d \\
0 & 1
\end{array} \right) \in {\rm GL}_2(\mathbb F_p)\,.$$
As such, we have the group isomorphisms
\begin{equation}\label{eq:IsomGRadical}
G\simeq (\Z/p\Z)\rtimes(\Z/p\Z)^\times\simeq \left\{\left(\begin{array}{cc}
c & d \\
0 & 1
\end{array} \right)\colon c\in \F_p^*, d\in\F_p\right\}\,.
\end{equation}
In other words, $G$ is the Frobenius group of invertible affine maps $x\mapsto cx+d$ of $\F_p$. Artin's conjecture is known for such Galois extensions. Indeed we have the following sequence:
$$ \{{\rm Id}\} \lhd H = \left \{  \left(   \begin{array}{cc}
1 & * \\ 
0 & 1
\end{array}  \right) \right\} \lhd G, $$
and the groups $G/H \cong (\Z/p\Z)^{\times}$, $H\cong \Z/p\Z$, are cyclic so $G$ is supersolvable.

 The prime numbers which ramify in $K_{a,p}/\Q$ are $p$ and $a$; more precisely we have 
(see \emph{e.g.}~\cite[end of the proof of the theorem]{Kom} and \cite[\S 3.I]{Wes})
$$
\disc(K_{a,p}/\Q)=p^{p-2}(\disc (\Q(a^{1/p})))^{p-1}=p^{p^2-2}a^{(p-1)^2}\,.
$$
We finally mention that $K_{a,p}/\Q$ enjoys the remarkable \emph{unique subfield property}: for every 
divisor $d$ of the degree $p(p-1)$ of $K_{a,p}/\Q$ there is a unique intermediate extension 
$K_{a,p}/L/\Q$ such that $[L:\Q]=d$ (see \emph{e.g.}~\cite[Th. $2.2$]{Vi}).

\subsubsection{Irreducible characters of $\Gal(K/\Q)$}\label{subsec:IrrCharRad}

The group $G\simeq (\Z/p\Z)\rtimes(\Z/p\Z)^\times$ has $p$ conjugacy classes 
(see \emph{e.g.}~\cite[Prop. 3.6]{Vi}) and thus $p$ distinct irreducible characters. The conjugacy classes are easily described through the isomorphism~\eqref{eq:IsomGRadical}: 
besides $\{{\rm Id}\}$ there is one conjugacy class of size $p-1$:
$$
U:=\left\{\left(\begin{array}{cc}
1 & \star \\
0 & 1
\end{array} \right)\colon \star\neq 0\right\}\,,
$$
and $p-2$ conjugacy classes of size $p$:
$$
T_c:=\left\{\left(\begin{array}{cc}
c & \star \\
0 & 1
\end{array} \right)\colon \star\in\F_p\right\}\,,\,\,(c\neq 1)\,.
$$ 

As for the characters of $G$, exactly $p-1$ of them have degree $1$: these are the lifts of Dirichlet characters $\chi$ modulo $p$
\begin{equation}
\psi\colon \left\{\left(\begin{array}{cc}
c & d \\
0 & 1
\end{array} \right)\colon c\in \F_p^\times, d\in\F_p\right\}
\rightarrow (\Z/p\Z)^\times\xrightarrow{\chi}\C^\times\,,\qquad 
\psi \left(\left(\begin{array}{cc}
c & d \\
0 & 1
\end{array} \right)\right)=\chi(c)\,.
\label{equation characters radical extension}
\end{equation}
The Frobenius-Schur indicator of such a character $\psi$ is easy to compute:
$$
\epsilon_2(\psi)=\frac{1}{|G|}\sum_{\substack{c\in \F_p^\times \\ d\in\F_p}}\psi\left(\left(\begin{array}{cc}
c & d \\
0 & 1
\end{array} \right)^2\right)=\frac{1}{p-1}\sum_{c\in\F_p^{\times}}\chi(c^2)=\epsilon_2(\chi).
$$
We deduce that $\epsilon_2(\psi)$ equals $1$ if $\chi$ is real, and equals $0$ otherwise.

The remaining irreducible character $\eta$ of $G$ can then be determined using orthogonality relations. By the above description of the conjugacy classes of $G$ we obtain the following values that entirely determine $\eta$:
\begin{equation}
\eta(1)=p-1\,,\qquad \eta\left(\left(\begin{array}{cc}
1 & \star \\
0 & 1
\end{array} \right)\right)=-1\,,\,\,(\star\neq 0)\,,\qquad \eta\left(\left(\begin{array}{cc}
c & \star \\
0 & 1
\end{array} \right)\right)=0\,,\,\,(c\neq 1)\,.
\label{eq:definition eta}
\end{equation}
Again we easily deduce the value of the Frobenius-Schur indicator of $\eta$:
$$
\epsilon_2(\eta)=\frac{1}{|G|}\sum_{\substack{c\in \F_p^\times \\ d\in\F_p}}\eta\left(\left(\begin{array}{cc}
c & d \\
0 & 1
\end{array} \right)^2\right)=\frac{1}{|G|}\sum_{\substack{c=-1\\d \in\F_p}}\eta(1)
+\frac{1}{|G|}\sum_{\substack{c=1\\ d\neq 0}} (-1)+\frac{1}{|G|}
\sum_{\substack{c=1\\ d=0}} \eta(1)=1\,.
$$

\subsubsection{The global Artin conductor $A(\chi)$ for $\chi\in \Irr(G)$}\label{section:ARadical}

We now compute the Artin conductor $A(\chi)$ of an irreducible representation $\chi\in \Irr(G)$. To do so we have to understand the ramification groups.
For a prime ideal $\mathfrak P$ of $\mathcal O_K$ lying above a prime number $\nu$, we recall that
$$ G_i:=G_i(\mathfrak P/\nu)=\{ \sigma \in G : \forall z \in \O_K, \sigma(z)\equiv z \bmod \mathfrak P^{i+1}  \} $$
defines a decreasing sequence of normal subgroups of $G=\Gal(K_{a,p}/\Q)$ where
$G_0$ is the inertia group, $G_1$ is the wild ramification group, and $G_i$ is trivial for large enough $i$.

We start with $\nu=a$. We invoke~\cite[Th. $4.3$]{Vi} which asserts that if
 $\mathfrak P$ is a prime ideal of $\O_{K_{a,p}}$ lying over $a$ then the corresponding ramification index is $e(\mathfrak P/a)=p$. In particular $a\nmid e(\mathfrak P/a)$ 
 so that $K/\Q$ is tamely ramified at $a$. By the unique subfield property mentioned above we conclude that:
 $$
G_0(\mathfrak p/a)\simeq \left\{   \left(\begin{array}{cc}
1 & d \\
0 & 1
\end{array} \right) : d\in \F_p \right\}\simeq \Z/p\Z\,,\qquad G_1(\mathfrak p/a)=\{{\rm Id}\}\,.
$$

This will allow us to compute the local factor at $\nu=a$ of the Artin conductor $A(\phi)$, which is equal to $a^{n(\phi,a)}$, with $n(\phi,a)={\rm codim}(V^{G_0})=\phi(1)-\frac 1{|G_0|} \sum_{a \in G_0} \phi(a)$, since $a$ is tamely ramified in 
$K_{a,p}/\Q$.
For an irreducible character $\psi$ of degree $1$ of $G$, corresponding to a Dirichlet character $\chi$ modulo $p$, we obtain:
$$
n(\psi,a)=1-\frac{1}{p}\sum_{d\bmod p}\chi(1)=0\,.
$$
For the character $\eta$ we have:
$$
n(\eta,a)=\eta(1)-\frac 1p\sum_{d\bmod p}\eta\left(\left(\begin{array}{cc}
1 & d \\
0 & 1
\end{array} \right)\right)=p-1-\frac 1p(p-1+(-1)\times(p-1))=p-1\,.
$$

We now take $\nu=p$. Since we assume that $a^{p-1}\not \equiv 1\bmod p^2$, 
the extension $K_{a,p}/\Q$ is totally ramified at $p$ (see \emph{e.g.}~\cite[Th. 5.5]{Vi}). Let 
$\mathfrak P$ be the unique prime ideal of $\O_{K_{a,p}}$ lying over $p$. We have $G_0(\mathfrak P/p)=G$. For $i\geq 1$ we observe that the intermediate cyclotomic extension $\Q(\zeta_p)/\Q$ is tame at $p$ (since $p\O_{\Q(\zeta_p)}=(1-\zeta_p)^{p-1})$ and thus $G_i(\mathfrak P/p)=G_i(\mathfrak P/(1-\zeta_p))$ 
for any $i\geq 1$ (here the ramification group $G_i(\mathfrak P/(1-\zeta_p))$ 
is relative to the extension $K_{a,p}/\Q(\zeta_p)$). As remarked in~\cite[Lem. 5.7]{Vi}, 
the element $\pi=(1-\zeta_p)/(a-a^{1/p})$ is a uniformizer for the unique valuation extending to $K$ the one attached to $1-\zeta_p$ on $\Q(\zeta_p)$. Moreover we 
have the group isomorphism
$$
\Gal(K_{a,p}/\Q(\zeta_p))\simeq \left\{   \left(\begin{array}{cc}
1 & d \\
0 & 1
\end{array} \right) : d\in \F_p \right\}\,,
$$
since an element $\sigma_d$ of $\Gal(K_{a,p}/\Q(\zeta_p))$ is entirely determined by the residue class $d$ modulo $p$ such that $\sigma_d(a^{1/p})=
\zeta_p^d a^{1/p}$. Therefore we can compute $G_1(\mathfrak P/(1-\zeta_p))$ 
(and more generally $G_i(\mathfrak P/(1-\zeta_p))$ for all $i\geq 1$) by considering the following $\pi$-adic valuation:
\begin{align*}
v_{\pi}(\sigma_1(\pi)-\pi)&=v_\pi\left((1-\zeta_p)\left(\frac{\zeta_p a^{1/p}-a^{1/p}}
{(a-\zeta_p a^{1/p})(a-a^{1/p})}\right)\right)\\
&=v_{\pi}\left(\pi\left(\frac{\zeta_p a^{1/p}-a^{1/p}}{a-\zeta_p a^{1/p}}\right)\right)=v_{\pi}\left(-\pi \sigma_1(\pi) a^{1/p}\right)\,.
\end{align*}
To compute this quantity we use the uniqueness of the extension of valuations to infer 
that $v_{\pi}(\sigma_1(\pi))=v_{\pi}(\pi)$.
As a consequence we have
$$
v_{\pi}(\sigma_1(\pi)-\pi)=2v_\pi(\pi)+v_\pi(a^{1/p})=2\,,
$$
since $p$ and $a$ are coprime and so $v_\pi(a^{1/p})=0$. We conclude that $G_1(\mathfrak P/(1-\zeta_p))\simeq\Gal(K_{a,p}/\Q(\zeta_p))\simeq \Z/p\Z$ and that 
$G_i(\mathfrak P/(1-\zeta_p))=\{{\rm Id}\}$ for $i\geq 2$. We have therefore computed the higher ramification groups at $p$ for $K_{a,p}/\Q$:
\begin{equation} \label{eq:rampradical}
G_0(\mathfrak P/p)=G\,,\,\, G_1(\mathfrak P/p)=\left\{   \left(\begin{array}{cc}
1 & d \\
0 & 1
\end{array} \right) : d\in \F_p \right\}\simeq \Z/p\Z\,,
\,\, G_i(\mathfrak P/p)=\{{\rm Id}\}\,\, (i\geq 2)\,.
\end{equation}

We now deduce the value of the Artin conductor at $p$ of each irreducible character $\phi$ of $G$. To do so we use the following formula that generalizes the one mentioned in the tame case (and used for
$\nu=a$), namely: 
$$
n(\phi,p)=\phi(1)-\frac{1}{|G_0|}\sum_{a\in G_0}\phi(a)+\frac{\phi(1)}{p-1}
-\frac{1}{|G_0|}\sum_{a\in G_1}\phi(a)\,;
$$
it is the specialisation of~\eqref{eq:condalternate} in the case where the higher ramification groups at $p$ satisfy~\eqref{eq:rampradical}.

First if $\phi=\psi$ corresponds to a nontrivial Dirichlet character modulo $p$, then $\psi$ restricts trivially to $G_1$ and we deduce:
$$
n(\psi,p)=1+\frac{1}{p-1}-\frac{|G_1|}{|G_0|}=1\,,
$$
while $n({\bf 1},p)=0$ for the trivial character ${\bf 1}$. If $\phi=\eta$ we have:
$$
n(\eta,p)=p-1+1-\frac{1}{|G_0|}(p-1-(p-1))=p\,.
$$

Thanks to the above computations we can deduce the exact value of $\log A(\chi)$ 
for every irreducible character $\chi$ of the Galois group $G$ of the splitting field of $x^p-a$ over $\Q$.

\begin{proposition}\label{prop:AforRadical}
Let $a,p$ be distinct odd primes such that $a^{p-1} \not \equiv 1 \bmod p^2$, and let $K_{a,p}$ be the splitting field of $x^p-a$ over $\Q$. Let $G=\Gal(K_{a,p}/\Q)$ and recall \eqref{eq:IsomGRadical}. Then we have the following.
\begin{itemize}
\item For an irreducible character $\psi$ of $G$ attached to a Dirichlet character 
$\chi$ modulo $p$,
$$
\log A(\psi)=\begin{cases} 0\text{ if $\chi$ is the principal character modulo $p$}\,,
\\
\log p\text{ otherwise}\,.
\end{cases}
$$
\item For the character $\psi=\eta$ defined in \eqref{eq:definition eta},
$$
\log A(\eta)=(p-1)\log a+p\log p\,.
$$
\end{itemize}
\end{proposition}

\begin{proof}
Since the base field is $\Q$ the definition of $A(\chi)$ (see section~\ref{section:ArtinCond}) is the following:
$$
A(\chi)=\mathfrak f(\chi)=\prod_p p^{n(\chi,p)}\,,
$$
for every $\chi\in\Irr(G)$. As already mentioned, the only ramified primes in $K/\Q$ are $a$ and $p$. From the computations of \S\ref{section:ARadical} we deduce that 
$n({\bf 1},a)=n({\bf 1},p)=0$ and for a character $\psi$ attached to a nontrivial Dirichlet character modulo $p$:
$$
A(\psi)=a^0p^1=p\,.
$$
Finally for the irreducible character $\eta$, the computations of \S\ref{section:ARadical} yield:
$$
A(\eta)=a^{p-1}p^p\,.
$$
\end{proof}

The goal of the next lemma is to estimate the number of couples of primes $(a,p)$ that are admissible in Proposition~\ref{prop:AforRadical}.
\begin{lemma}
\label{lemma count of admissible couples of primes}
For $A,P\geq 3$ in the range $ P\log P \leq  A \leq  \ee^{P^2/(\log P)^3}$ 
one has the estimate
\begin{multline*}
\#\{a\leq A, p\leq P\colon a,p\text{  primes}, a\neq p, a^{p-1}\not\equiv 1 \bmod\ p^2\}=\pi(A)\pi(P)\\+O\Big( A \Big(\frac{\log\log A}{\log A}\Big)^{\frac 12}+\frac{P^2}{\log P}\Big).
\end{multline*}
\end{lemma}

\begin{proof}
The cardinality we wish to compute is

\begin{equation}
\sum_{p\leq P}\sum_{\substack{a\leq A,\, a\neq p\\ a^{p-1}\not\equiv 1 \bmod\ p^2}}1
\,.
\label{equation cardinality to compute}
\end{equation}
First we apply Hensel's Lemma: for each $p\leq P$ and each $a\neq p$, the polynomial $f(X)=X^{p-1}-1\in\F_p[X]$ is separable and splits completely in $\F_p[X]$. Thus any $\alpha\in\F_p^\times$ (which is necessarily a root of $f$) 
lifts to a unique $\alpha_0\in\Z/p^2\Z$ such that $\alpha_0\equiv \alpha \bmod\ p$ and $\alpha_0^{p-1}\equiv 1 \bmod\ p^2$. Let $S_p$ 
be a set of representatives for these $p-1$ residue classes modulo $p^2$. Therefore, since $A\geq P$, \eqref{equation cardinality to compute} is equal to 
\begin{equation}\label{eq:DoubleSum}
\sum_{\substack{p\leq P}}\Big((\pi(A)-1)-\sum_{c\in S_p}\pi(A;p^2,c)\Big).
\end{equation}
Now, by the Brun-Titchmarsh Theorem we have the bound
$$
\sum_{\substack{ p\leq A^{\frac 12}/2}}\sum_{c\in S_p}\pi(A;p^2,c) \ll \sum_{\substack{ p\leq A^{\frac 12}/2}}\frac{A\# S_p } {p^2\log (A/p^2)}\ll A \Big(\frac{\log\log A}{\log A}\Big)^{\frac 12}
$$
(The last bound is obtained by cutting the sum over $p$ at the point $ A^{\frac 12 - (\frac{\log \log A}{\log A})^{\frac 12}}  $.)
Moreover, we trivially have
$$
\sum_{\substack{ A^{\frac 12}/2 \leq p \leq P }}\sum_{c\in S_p}\pi(A;p^2,c) \ll \sum_{\substack{ A^{\frac 12}/2 \leq p \leq P }} \# S_p \ll \frac{P^2}{\log P}.
$$
(Note that this last sum is empty when $A\geq 2P^2$.) The result follows.
\end{proof}

\begin{proof}[Proof of Theorem~\ref{th:BiasRadical}]

In this proof we keep the notation as in~\S\ref{subsec:IrrCharRad}.

We start with the proof of Theorem~\ref{th:BiasRadical}(1). The bias factor $B(K_{a,p}/\Q;\res)$ 
can be estimated precisely since under LI, Proposition~\ref{proposition link with random variables} and Proposition~\ref{proposition asymptotic for the variance} yield
\begin{align*}
\E[X(K_{a,p}/\Q;\res)]&=\sum_{\substack{1\neq \psi\in\Irr(G)\\ \psi\text{ real}}}1\\
\V[X(K_{a,p}/\Q;\res)]& \asymp \sum_{\substack{1\neq \psi\in\Irr(G)\\ \psi\text{ real}}}\log (A(\psi)+2)\,.
\end{align*}
 The only real irreducible characters of $G$ are the trivial character, the character $\psi_0$ of degree $1$ attached to the quadratic character of $(\Z/p\Z)^\times$ and $\eta$ (see~\S\ref{subsec:IrrCharRad}). By Proposition~\ref{prop:AforRadical}, we deduce that
$$
\E[X(K_{a,p}/\Q;\res)]=2\,,\qquad \V[X(K_{a,p}/\Q;\res)]\asymp p \log (ap)\,.
$$
By the definition~\eqref{equation definition bias factors} of the bias factor, we obtain the bounds
$$
B(K_{a,p}/\Q;\res)\asymp \frac{1}{\sqrt{p\log (ap)}}\,.
$$
To conclude the proof of~\eqref{eq:BiasRNRRadical} we apply the first estimate in Theorem~\ref{theorem asymptotic formula for moderately biased races}.

Next, we prove Theorem \ref{th:BiasRadical}(2). From \eqref{equation number of square roots} one computes that
\begin{multline}
\label{equation expectancies of radical}
\E[X(K_{a,p}/\Q;t_{U,\{{\rm id}\}})]= p; \;\;\E[X(K_{a,p}/\Q;t_{x^+,\{{\rm id}\}})]=p-\Big( \frac xp \Big); \\ \E[X(K_{a,p}/\Q;t_{U,x^+}) ]=  \Big( \frac xp \Big).
 \end{multline}
Similarly, for any $x\in \mathbb F_p\setminus \{0,1\}$ we obtain the estimates
$$
\V[X(K_{a,p}/\Q;t_{U,\{{\rm id}\}})] \asymp \V[X(K_{a,p}/\Q;t_{x^+,\{{\rm id}\}})] \asymp p^3\log(ap);
$$
$$ \V[X(K_{a,p}/\Q;t_{U,x^+}) ]=  p\log(ap).$$
We conclude the proof by invoking the second estimate in Theorem~\ref{theorem asymptotic formula for moderately biased races}.

We turn to the proof of Theorem \ref{th:BiasRadical}(3). 
If $C_1=x^+$ and $C_2=y^+$ with $x\neq y$ elements of $\F_p\setminus \{0,1\}$ then the local factor of $L(s,K_{a,p}/\Q,\psi)$ and that of the associated Dirichlet $L$-function $L(s,\chi)$ (see \eqref{equation characters radical extension}) are identical at every prime not equal to $a$ or $p$. Thus, those functions have the same critical zeros. From Lemma \ref{lemma random var under LI} and from the fact that $\eta(x^+) = \eta(y^+)=0$ we deduce that the distribution of $X(K_{a,p}/\Q;t_{x^+,y^+})$ is identical to that of $X_{p;x,y}$, the random variable associated to the classical Chebyshev bias defined in \cite[Definition 2.4]{FiMa}. The claim follows.

Finally, in the relative case $K_{a,p}/\Q(\zeta_p)$ the Galois group $H=\Gal(K_{a,p}/\Q(\zeta_p))$ has order $p$, and hence has no nontrivial real irreducible character. As explained at the beginning of this section, the group $H^+=G$ does not have any symplectic character and therefore LI and the induction property of Artin $L$-functions imply that $\varepsilon_2(\chi)+2{\rm ord}_{s=\tfrac 12} L(s,K_{a,p}/\Q(\zeta_p),\chi)=0$ for every nontrivial $\chi\in\Irr(H)$. The assertion on the mean is then immediately deduced from~\eqref{equation expectancy with characters}. As for the variance, we easily notice that $\{d_1\}^+=U$ for any nontrivial $d_1\in H$. Therefore the variance is $0$ if both $d_1$ and $d_2$ are nontrivial elements of $H$, and otherwise we have that 
$$ \V[X(K_{a,p}/\Q(\zeta_p);t_{d_1,\{{\rm id}\}}) ] = \V[X(K_{a,p}/\Q;t_{U,\{ {\rm id}\}})] \asymp p^3\log(ap).  $$
\end{proof}

\begin{proof}[Proof of Corollary \ref{corollary admissible couples in radical extensions}]
Combine Theorem~\ref{th:BiasRadical} with Lemma~\ref{lemma count of admissible couples of primes}.
\end{proof}


\begin{thebibliography}{99}

\bibitem[ANS]{ANS} A. Akbary, N. Ng and M. Shahabi, 
\emph{Limiting distributions of the classical error terms of prime number theory.} Q. J. Math. \textbf{65} (2014), no. 3, 743--780.

\bibitem[An]{An} G.E. Andrews, The theory of partitions. Encyclopedia of Mathematics and its Applications, Vol. 2. Addison-Wesley Publishing Co., Reading, Mass.-London-Amsterdam, 1976. xiv+255 pp.

\bibitem[Ar]{Ar} J. V. Armitage, \emph{Zeta functions with a zero at $s=\frac 12$.} Invent. Math. \textbf{15} (1972), 199--205.

\bibitem[Ba]{Ba} A. Bailleul, \emph{Chebyshev’s bias in dihedral and generalized
quaternion Galois groups}, available at~\url{arXiv:2001.06671}.

\bibitem[Be1]{Be} J. Bella\"iche, \textit{Théorème de Chebotarev et complexité de Littlewood.}
Ann. Sci. Éc. Norm. Supér. (4) 49 (2016), no. 3, 579--632.

\bibitem[Be2]{Be2} J. Bella\"iche, \emph{Remarks on the error term in Chebotarev's density theorem.} Math. Res. Lett. \textbf{24} (2017), no. 3, 679--687. 

\bibitem[BT]{BT} D. G. Best and T. S. Trudgian, \emph{Linear relations of zeroes of the zeta-function.} Math. Comp. \textbf{84} (2015), no. 294, 2047--2058. 

\bibitem[BM]{BM} N. Boston and N. Markin, \emph{The fewest primes ramified in a $G$-extension of $\Q$}. Ann. Sci. Math. Qu\'ebec \textbf{33} (2009), no. 2, 145--154.

\bibitem[Bo]{Bou} J. Bourgain, \emph{A remark on solutions of the Pell equation}. Int. Math. Res. Not. IMRN 2015, no. 10, 2841--2855.

\bibitem[BKS]{BKS} R. de la Bret\`eche, P. Kurlberg and I. Shparlinski, \emph{On the number of products which form perfect powers and discriminants of multiquadratic extensions.} forthcoming, IMRN, and \url{arXiv:1901.10694}.

\bibitem[BD]{BD} S. Brueggeman and D. Doud, 
\emph{Local corrections of discriminant bounds and small degree extensions of quadratic base fields}. 
Int. J. Number Theory 4 (2008), no. 3, 349--361.

\bibitem[BK]{BK} A. Bucur and K. Kedlaya,\emph{An application of the effective Sato--Tate conjecture}. Frobenius distributions: Lang--Trotter and Sato--Tate conjectures, 45--56, Contemp. Math., 663, Amer. Math. Soc., Providence, RI, 2016.

\bibitem[Cha]{Cha}
B. Cha, \textit{Chebyshev's bias in function fields}. Compos. Math. 144 (2008), no. 6, 
1351--1374.

\bibitem[CI]{CI}
B. Cha and B.-H. Im, \textit{Chebyshev's bias in Galois extensions of global function fields.} J. Number Theory 131 (2011), no. 10, 1875--1886.

\bibitem[CFJ]{CFJ}
B. Cha, D.Fiorilli and F. Jouve, \textit{Prime number races for elliptic curves over function fields}. Ann. Sci. Éc. Norm. Supér. (4) 49 (2016), no. 5, 1239--1277

\bibitem[CK1]{CK} P. J. Cho and H. H. Kim, \emph{Effective prime ideal theorem and exponents of ideal class groups.} 
Q. J. Math. \textbf{65} (2014), no. 4, 1179--1193. 

\bibitem[CK2]{CK2} P. J. Cho and H. H. Kim, \emph{The Average of the Smallest Prime in a Conjugacy Class.},  Int. Math. Res. Not. IMRN 2020, no. 6, 1718--1747. 

\bibitem[C]{Ch} S. Chowla, 
\emph{On the class-number of the corpus $P(\sqrt{-k})$.}
Proc. Nat. Inst. Sci. India 13, (1947). 197--200.


\bibitem[De]{De} L. Devin, \textit{Chebyshev's bias for analytic $L$-functions}, Math. Proc. Cambridge Philos. Soc., forthcoming, Math. Proc. Cambridge Philos. Soc., available at \url{doi.org/10.1017/S0305004119000100} and \url{arXiv:1706.06394}.

\bibitem[DM]{DM} L. Devin and X. Meng, \textit{Chebyshev's bias for products of irreducible polynomials}, available at \url{arXiv:1809.09662}.

\bibitem[Di]{Di} R. Dietmann, \textit{Probabilistic Galois theory}. Bull. Lond. Math. Soc. 45 (2013), no. 3, 453--462.

\bibitem[DGK]{DGK}D. Dummit, A. Granville and H. Kisilevsky, \textit{
Big biases amongst products of two primes.}
Mathematika 62 (2016), no. 2, 502–507.

\bibitem[Du]{Du} W. Duke, 
\emph{Number fields with large class group}. Number theory, 117--126,
CRM Proc. Lecture Notes, 36, Amer. Math. Soc., Providence, RI, 2004. 

\bibitem[Es]{Es} Carl-Gustav Esseen, \emph{Fourier analysis of distribution functions. A mathematical study of the Laplace-Gaussian law.} Acta Math. \textbf{77} (1945), 1--125.

\bibitem[EM]{EM} C. Euvrard and C. Maire, 
\emph{Sur la séparation des caractères par les Frobenius.}
Publ. Mat. 61 (2017), no. 2, 475--515.

\bibitem[FeS]{FeSn}
V. F\'eray and P. \'Sniady,\emph{
Asymptotics of characters of symmetric groups related to Stanley character formula.} 
Ann. of Math. (2) 173 (2011), no. 2, 887--906. 

\bibitem[Fio]{Fio} A. Fiori, \emph{Lower bounds for the least prime in Chebotarev.} Algebra Number Theory \textbf{13} (2019), no. 9, 2199--2203. 

\bibitem[Fi1]{Fi1} D. Fiorilli, \emph{Highly biased prime number races.}  Algebra Number Theory 8 (2014), no. 7, 1733--1767.

\bibitem[Fi2]{Fi2} D. Fiorilli, \emph{Elliptic curves of unbounded rank and Chebyshev's bias.} Int. Math. Res. Not. IMRN 2014, no. 18, 4997--5024.


\bibitem[FM]{FiMa} D. Fiorilli and G. Martin, \emph{Inequities in the Shanks-R\'enyi Prime Number Race: An asymptotic formula for the densities.} J. Reine Angew. Math. 676 (2013), 121--212.

\bibitem[FoS]{FoSn} K. Ford and J. Sneed, \textit{
Chebyshev's bias for products of two primes. }
Experiment. Math. 19 (2010), no. 4, 385--398.

\bibitem[Fo]{Fou} \'E. Fouvry, \emph{On the size of the fundamental solution of the Pell equation}. J. Reine Angew. Math. 717 (2016), 1--33.

\bibitem[FQ]{FrQu} A. Fr\"ohlich and J. Queyrut, \emph{On the Functional Equation of the Artin $L$-Function for Characters of Real Representations}, Invent. Math. 20, 125--138 (1973).


\bibitem[Fr]{Fr} A. Fr\"ohlich, \emph{Algebraic number fields: L-functions and Galois properties.} Proceedings of a Symposium held at the University of Durham, Durham, Sept. 2-12, 1975. Edited by A. Fröhlich. Academic Press [Harcourt Brace Jovanovich, Publishers], London-New York, 1977. xii+704 pp.

\bibitem[Ga]{Gal}
P. X. Gallagher, 
\emph{The large sieve and probabilistic Galois theory.} Analytic number theory (Proc. Sympos. Pure Math., Vol. XXIV, St. Louis Univ., St. Louis, Mo., 1972), pp. 91-101. Amer. Math. Soc., Providence, R.I., 1973. 


\bibitem[Go]{Gold}
   R. Gold,
   \emph{Hilbert class fields and split extensions},
   {Illinois J. Math.},
   {\bf 21},
   (1977),
   no 1,
   66--69.
   
   \bibitem[GrMa]{GM} A. Granville and G. Martin, \emph{Prime number races.} Amer. Math. Monthly \textbf{113} (2006), no. 1, 1--33.
   
  \bibitem[GrMo]{GrMo} L. Greni\'e and G. Molteni, \emph{An explicit Chebotarev density theorem under GRH.} J. Number Theory 200 (2019), 441--485.


\bibitem[Ha]{Ha} H. Hasse, Number theory. Translated from the third German edition of 1969 by Horst Günter Zimmer. Akademie-Verlag, Berlin, 1979. xvii+638 pp. 

\bibitem[Ho]{Ho} C. Hooley, 
\emph{On the Pellian equation and the class number of indefinite binary 
quadratic forms}. J. Reine Angew. Math. 353 (1984), 98--131.



\bibitem[Hu]{Hup} B. Huppert, Character theory of finite groups. de Gruyter Expositions in Mathematics, 25. Walter de Gruyter \& Co., Berlin, 1998. vi+618 pp.



\bibitem[IK]{IwKo} H. Iwaniec and E. Kowalski, Analytic number theory. American Mathematical Society Colloquium Publications, 53. American Mathematical Society, Providence, RI, 2004. xii+615 pp. ISBN: 0-8218-3633-1 



\bibitem[Kac1]{Ka} J. Kaczorowski, \emph{On the distribution of primes (mod $4$).} Analysis \textbf{15} (1995), no. 2, 159--171.

\bibitem[Kac2]{Ka2} J. Kaczorowski, \emph{On the Shanks-R\'enyi race problem}. Acta Arith. 74 (1996), no. 1, 31--46.

\bibitem[KNW]{KNW} H. Kadiri, N. Ng and P. Wong, \emph{The least prime ideal in the Chebotarev density theorem.} Proc. Amer. Math. Soc. \textbf{147} (2019), no. 6, 2289--2303.

\bibitem[Kat]{Katz} N. M. Katz,\emph{
Wieferich past and future}. Topics in finite fields, 253--270,
Contemp. Math., 632, Amer. Math. Soc., Providence, RI, 2015. 

\bibitem[Kl]{kluners06} {J. Kl{\"u}ners},
   \emph{Asymptotics of number fields and the Cohen-Lenstra heuristics},
    J. Th\'eor. Nombres Bordeaux,
   \textbf{18},
   (2006),
   no. 3,
   607--615.


\bibitem[KT]{KT}  S. Knapowski and P. Tur\'an, \emph{Comparative prime-number theory. I--III}. Acta Math. Acad. Sci. Hungar. 13 (1962), 299--364. 

\bibitem[Kom]{Kom} K. Komatsu, \textit{An integral basis of the algebraic number field $\Q(a^{1/\ell},\zeta_\ell)$}.
J. Reine Angew. Math. 288 (1976), 152--153. 

\bibitem[Kow]{Kow} E. Kowalski, The large sieve and its applications.
Arithmetic geometry, random walks and discrete groups. Cambridge Tracts in Mathematics, 175. Cambridge University Press, Cambridge, 2008. xxii+293 pp. 

\bibitem[Kup]{Kup} G. Kuperberg, 
\textit{Knottedness is in {\tt NP}, modulo GRH}. 
Adv. Math. 256 (2014), 493--506.

\bibitem[LMO]{LMO} J. C. Lagarias, H. L. Montgomery and A.M. Odlyzko, \emph{A bound for the least prime ideal in the Chebotarev density theorem}.
Invent. Math. \textbf{54} (1979), no. 3, 271--296. 


\bibitem[LO]{LaOd} J.C. Lagarias and A.M. Odlyzko, {\it Effective versions of the Chebotarev density theorem.}
in Algebraic number fields: $L$-functions and Galois properties (Proc. Sympos., Univ.
Durham, 1975), Academic Press, London, 1977, p. 409--464.


\bibitem[La1] {La1} Y. Lamzouri,
\emph{Large deviations of the limiting distribution in the Shanks-Rényi prime number race.} Math. Proc. Cambridge Philos. Soc. \textbf{153} (2012), no. 1, 147--166. 

\bibitem[La2]{LamQuadFields} Y. Lamzouri, \emph{Extreme values of class numbers of real quadratic fields.}
International Mathematics Research Notices IMRN 2015, no. 22, 11847--11860.



\bibitem[LS]{LaSh} Michael Larsen and Aner Shalev, \emph{Characters of symmetric groups: sharp bounds and applications.} Invent. Math. \textbf{174} (2008), no. 3, 645-687.

\bibitem[LOS]{LOS} R. Lemke Oliver and K. Soundararajan, \textit{
Unexpected biases in the distribution of consecutive primes. }
Proc. Natl. Acad. Sci. USA 113 (2016), no. 31, 

\bibitem[Le]{Lem} F. Lemmermeyer, Reciprocity laws. From Euler to Eisenstein. Springer Monographs in Mathematics. Springe-Verlag, Berlin, 2000. xx+487 pp. 

\bibitem[LR]{LR} X. Li and M. Radziwi\l\l, 
\textit{The Riemann zeta function on vertical arithmetic progressions.}
Int. Math. Res. Not. IMRN 2015, no. 2, 325--354.

\bibitem[Lin]{Lin} U. V. Linnik,
\emph{On the least prime in an arithmetic progression. I. The basic theorem.},
Rec. Math. [Mat. Sbornik] N.S. 15(57), (1944). 139--178.

\bibitem[Lit1]{LitPrimes} J. E. Littlewood, 
\emph{Sur la distribution des nombres premiers},
C. R. Acad. Sci. Paris 158 (1914), 1869--1872.

\bibitem[Lit2]{Lit} J. E. Littlewood, 
\emph{On the Class-Number of the Corpus $P(\sqrt{-k})$.}
Proc. London Math. Soc. (2) 27 (1928), no. 5, 358--372. 

\bibitem[Mal]{Mal} G. Malle, \emph{On the distribution of Galois groups.} J. Number Theory 92 (2002), no. 2, 315--329.

\bibitem[MN1]{MN-AP} G. Martin and N. Ng, \textit{Nonzero values of Dirichlet L-functions in vertical arithmetic progressions.} Int. J. Number Theory 9 (2013), no. 4, 813--843.

\bibitem[MN2]{MN} G. Martin and N. Ng, \emph{Inclusive prime number races}, 
forthcoming, Transactions of the American Mathematical Society.

\bibitem[MS]{MaSc} G. Martin and J. Scarfy, \emph{Comparative prime number theory: A survey}, available at \url{arXiv:1202.3408}.

\bibitem[M+]{MSetal} G. Martin \emph{et al.} \emph{A complete annotated bibliography for comparative prime number theory}, lecture notes, available at 
\url{http://www.math.ubc.ca/~gerg/teaching/592-Fall2018/evolving.pdf}.

\bibitem[Mar]{Mar} J. Martinet,
Character theory and Artin $L$-functions. Algebraic number fields: $L$-functions and Galois properties (Proc. Sympos., Univ. Durham, Durham, 1975), pp. 1--87. Academic Press, London, 1977.
 
\bibitem[Maz]{Maz} B. Mazur, \emph{Finding meaning in error terms.}
Bull. Amer. Math. Soc. \textbf{45} (2008), no. 2, 185--228. 

\bibitem[Me]{Me} X. Meng, 
\textit{
Chebyshev's bias for products of $k$ primes.}
Algebra Number Theory 12 (2018), no. 2, 305--341.

\bibitem[Mo]{Mo} W. R. Monach, \emph{Numerical investigation of several problems in number theory.} Ph.D Dissertation, University of Michigan (1980).

\bibitem[MO]{MoOd} H. L. Montgomery and A. M. Odlyzko, \emph{Large deviations of sums of independent random variables.} Acta Arith. \textbf{49} (1988), no. 4, 427--434.

\bibitem[MV]{MoVa} H. L. Montgomery and R. C. Vaughan, Multiplicative number theory. I. Classical theory. Cambridge Studies in Advanced Mathematics, 97. Cambridge University Press, Cambridge, 2007. xviii+552 pp.

\bibitem[MW]{MoWe} H. L. Montgomery and J. P. Weinberger, \emph{Real quadratic fields with large class number}, Math. Ann., 225 (1977), no. 2, 173--176.

\bibitem[MOT]{MOT} M. J. Mossinghoff, T. Oliveira e Silva and T. Trudgian, \emph{The distribution of $k$-free numbers}, available at \url{arXiv:1912.04972}.

\bibitem[Mu1]{Mur} V. K. Murty, \emph{Explicit formulae and the Lang-Trotter conjecture.} Number theory (Winnipeg, Man., 1983). Rocky Mountain J. Math. 15 (1985), no. 2, 535--551.

\bibitem[Mu2]{Mu2} V. K. Murty, \emph{The least prime in a conjugacy class.}
C. R. Math. Acad. Sci. Soc. R. Can. \textbf{22} (2000), no. 4, 129--146. 


\bibitem[MMS]{MMS} M. R. Murty, V. K. Murty and N. Saradha, \textit{Modular forms and the Chebotarev density theorem}. Amer. J. Math. 110 (1988), no. 2, 253--281.


\bibitem[MM]{MuMu} M. R. Murty and V. K. Murty, Non-vanishing of $L$-functions and applications. Progress in Mathematics, 157. Birkh\"auser Verlag, Basel, 1997. xii+196 pp. ISBN: 3-7643-5801-7


\bibitem[Ng]{Ng} N. Ng, \emph{Limiting distributions and zeros of Artin L-functions.} Ph.D. thesis, University of British Columbia, 2000. Available at \url{http://www.cs.uleth.ca/~nathanng/RESEARCH/phd.thesis.pdf}.

\bibitem[PTW]{PTW} L. Pierce, C.Turnage-Butterbaugh and M. Wood \emph{An effective Chebotarev density theorem for families of number fields with an application to $\ell$-torsion in class groups}. forthcoming, Invent. Math., available at \url{arXiv:1709.09637}.

\bibitem[Pi1]{Pi} A. Pizarro-Madariaga, \emph{Lower bounds for the Artin conductor.} Math. Comp. \textbf{80} (2011), no. 273, 539--561.

\bibitem[Pi2]{Pi2} A. Pizarro-Madariaga, \emph{Irreducible characters with bounded root Artin conductor.} Algebra Number Theory \textbf{13} (2019), no. 9, 1997--2004.


\bibitem[Pu]{Pu} J.C. Puchta, \emph{On large oscillations of the remainder of the prime number theorems.} Acta Math. Hungar. 87 (2000), no. 3, 213--227.

\bibitem[Ro]{Ro} Yuval Roichman, \emph{Upper bound on the characters of the symmetric groups.} Invent. Math. \textbf{125} (1996), no. 3, 451--485. 


\bibitem[RbS]{RubSar} M. Rubinstein and P. Sarnak, \emph{Chebyshev's bias.} Experiment. Math. \textbf{3} (1994), no. 3, 173--197. 

\bibitem[RdS]{RudSar} Z. Rudnick and P. Sarnak, \emph{Zeros of principal $L$-functions and random matrix theory},     Duke Math. J.
    Volume 81, Number 2 (1996), 269-322.

\bibitem[Sag]{Sag} B. E. Sagan, The symmetric group. Representations, combinatorial algorithms, and symmetric functions. Second edition. Graduate Texts in Mathematics, 203. Springer-Verlag, New York, 2001. xvi+238 pp.

\bibitem[Sage]{sage}
\emph{{S}ageMath, the {S}age {M}athematics {S}oftware {S}ystem ({V}ersion
  8.6)}, The Sage Developers, 2019, \url{http://www.sagemath.org}.

\bibitem[Sa1]{SarLet} P. Sarnak, \emph{Letter to Barry Mazur on ``Chebyshev's bias'' for $\tau(p)$.} \url{https://publications.ias.edu/sites/default/files/MazurLtrMay08.PDF}.

\bibitem[Sa2]{Sar} P. Sarnak, \emph{Class numbers of indefinite binary quadratic forms, II}, J. Number Theory \textbf{21} (1985), no.3, 333--346.


\bibitem[Se1]{SeCL} J.-P. Serre, Corps locaux. 
Deuxi\`eme \'edition. Publications de l'Universit\'e de Nancago, No. VIII. Hermann, Paris, 1968. 245 pp. 

\bibitem[Se2]{SeRep} J.-P. Serre, Repr\'esentations lin\'eaires des groupes finis. 
Third revised edition. Hermann, Paris, 1978.

\bibitem[Se3]{SeIHES} J.-P. Serre,
\textit{Quelques applications du th\'eor\`eme de densit\'e de Chebotarev.} 
Inst. Hautes \'Etudes Sci. Publ. Math. No. 54 (1981), 323–401.


\bibitem[Te]{Ten} G. Tenenbaum, Introduction \`a la th\'eorie analytique et probabiliste des nombres. 4\`eme \'edition mise \`a jour. Belin, Collection \'Echelles, 2015. 592 pp.


\bibitem[TZ1]{TZ1} J. Thorner and A. Zaman, \emph{An explicit bound for the least prime ideal in the Chebotarev density theorem.} Algebra Number Theory \textbf{11} (2017), no. 5, 1135--1197.

\bibitem[TZ2]{TZ2} J. Thorner and A. Zaman,
\emph{A unified and improved Chebotarev density theorem.}
Algebra Number Theory 13 (2019), no. 5, 1039--1068.

\bibitem[TZ3]{TZ3} J. Thorner and A. Zaman, \emph{A zero density estimate for Dedekind zeta functions}, available at \url{arXiv:1909.01338}. 


\bibitem[Ul]{Ul} D. Ulmer, \textit{Elliptic curves with large rank over function fields}. Ann. of Math. (2) 155 (2002), no. 1, 295--315.

\bibitem[Vi]{Vi} F. Viviani, \emph{Ramification groups and Artin conductors of radical extensions of $\Q$}, J. Th\'eor. Nombres Bordeaux 16 (2004), no. 3, 779-816. 

\bibitem[Wei]{Wei} P.J. Weinberger, \emph{Exponents of the class groups of complex quadratic fields}. Acta. Arith. \textbf{22} (1973), no. 2, 117--124.

\bibitem[Wes]{Wes} J. Westlund, \emph{On the fundamental number of the algebraic number-field $k(\sqrt[\leftroot{-2}\uproot{2}p]{m})$.}
Trans. Amer. Math. Soc. 11 (1910), no. 4, 388–392. 

\bibitem[Wil]{Wil} H. Wilf, \emph{The asypmtotics of $\ee^{P(z)}$ and the number of elements of each order in $S_n$}. Bull. AMS, Vol. 15, {\bf 82}, 1986.

\bibitem[Winc]{Winc} B. Winckler, \emph{Th\'eor\`eme de Chebotarev effectif}, available at \url{arXiv:1311.5715}.

\bibitem[Wint]{Win} A. Wintner, \emph{On the distribution function of the remainder term of the prime number theorem}. Amer. J. Math. 63, (1941). 233--248.

\bibitem[Xi]{Xi} P. Xi, \emph{Counting fundamental solutions to the Pell equation with prescribed size}. Compos. Math. 154 (2018), no. 11, 2379--2402. 

\bibitem[Xy]{Xy} T. Xylouris, \emph{On the least prime in an arithmetic progression and estimates for the zeros of Dirichlet $L$-functions.} 
Acta Arith. \textbf{150} (2011), no. 1, 65--91. 



\bibitem[Za]{Za} A. Zaman, \emph{Bounding the least prime ideal in the Chebotarev density theorem.} Funct. Approx. Comment. Math. \textbf{57} (2017), no. 1, 115--142. 

\bibitem[ZS]{ZarSam} O. Zariski and P. Samuel, \emph{Commutative Algebra Vol. 1}. With the cooperation of I. S. Cohen. Corrected reprinting of the 1958 edition. Graduate Texts in Mathematics, No. 28. Springer-Verlag, New York-Heidelberg-Berlin, 1975. xi+329 pp.

\end{thebibliography}
\end{document}